\documentclass[11pt]{article} 

\usepackage{amssymb}
\usepackage{mathrsfs}
\usepackage{hyperref}

\usepackage{mathtools}
 
\usepackage{latexsym}\usepackage{amsmath}\usepackage{xspace}\usepackage{enumerate}\usepackage{amsfonts}\usepackage{amscd}\usepackage{syntonly} 
\usepackage[alwaysadjust]{paralist}

\usepackage{latexsym}\usepackage{amsmath}\usepackage{xspace}\usepackage{enumerate}
\usepackage{amsfonts}\usepackage{amscd}\usepackage{syntonly} 
\usepackage{hyperref}
\usepackage{bbm}
\usepackage{graphicx}
\usepackage{amssymb}
 
\usepackage{mathtools}
\usepackage{latexsym}

\newtheorem{thm}{Theorem}[section]
\newtheorem{lem}[thm]{Lemma}
\newtheorem{prop}[thm]{Proposition}

\newtheorem{defn} [thm]{Definition}
\newtheorem{rem} [thm]{Remark}

\newenvironment{proof}[1][Proof]{\noindent \textbf{#1:} }{\hspace{\stretch{1}} $\Box$\\}
 
\newtheorem{step}{Step}

\usepackage{bbm}
\usepackage{t1enc}

\newcommand{\eps}{\varepsilon}
\newcommand{\ph}{\varphi}

\newcommand\crit{\operatorname{crit}}

\newcommand\supp{\operatorname{supp}}

\def\ad{\operatorname{ad}}
\def\Ad{\operatorname{Ad}}
\def\coker{\operatorname{coker}}

\def\diag{\operatorname{diag}}
\def\dist{\operatorname{dist}}
\def\dom{\operatorname{dom}}
\def\dvol{\operatorname{dvol}}
\def\Hess{\operatorname{Hess}}
\def\im{\operatorname{im}}
\def\ind{\operatorname{ind}}
\def\Ind{\operatorname{Ind}}
\def\interior{\operatorname{int}}
\def\loc{\operatorname{loc}}
\def\mod{\operatorname{mod}}
\def\pr{\operatorname{pr}}

\def\reg{\operatorname{reg}}
\def\Stab{\operatorname{Stab}}
\def\supp{\operatorname{supp}}
\def\univ{\operatorname{univ}}

\def\A{\mathcal A}
\def\AA{\mathbbm A}
\def\B{\mathcal B}
\def\E{\mathcal E}
\def\F{\mathcal F}
\def\G{\mathcal G}
\def\H{\mathcal H}
\def\L{\mathcal L}

\def\p{\mathcal P}
\def\S{\mathcal S}
\def\V{\mathcal V}
\def\Z{\mathcal Z}
\def\YM{\mathcal{YM}}
\def\YMV{\mathcal{YM}^{\V}}

\def\R{\mathbb{R}}

\title{Morse Homology for the Yang--Mills Gradient Flow}
\author{Jan Swoboda (Mathematisches Institut der LMU M\"unchen)}

\begin{document}
\maketitle

\begin{abstract}    
We use the Yang--Mills gradient flow on the space of connections over a closed Riemann surface to construct a Morse chain complex. The chain groups are generated by Yang--Mills connections. The boundary operator is defined by counting the elements of appropriately defined moduli spaces of Yang--Mills gradient flow lines that converge asymptotically to Yang--Mills connections.\\

\end{abstract}



\tableofcontents

\section{Introduction}\label{Introduction}
Let $(\Sigma,g)$ be a closed oriented Riemann surface. Let $G$ be a compact Lie group, $\mathfrak g$ its Lie algebra, and $P$ a principal $G$-bundle over $\Sigma$. On $\mathfrak g$ we choose an $\ad$-invariant inner product $\langle\cdot,\cdot\rangle$. The Riemannian metric $g$ induces for $k\in\{0,1,2\}$ the Hodge star operator $\ast\colon\Omega^k(\Sigma)\to\Omega^{2-k}(\Sigma)$ on differential $k$-forms. We denote by $\A(P)$ the affine space of $\mathfrak g$-valued connection $1$-forms on $P$. The underlying vector space is the space $\Omega^1(\Sigma,\ad(P))$ of sections of the adjoint bundle $\ad(P)\coloneqq P\times_{\Ad}\mathfrak g$. The curvature of a connection $A\in\A(P)$ is the $\ad(P)$-valued $2$-form $F_A=dA+\frac{1}{2}[A\wedge A]$. On $\A(P)$ we consider the \emph{perturbed Yang--Mills functional} defined by
\begin{eqnarray}\label{YMintroduction}
\YMV(A)=\frac{1}{2}\int_{\Sigma}\big\langle F_A\wedge\ast F_A\big\rangle+\V(A)
\end{eqnarray}
for a gauge-invariant perturbation $\V\colon\A(P)\to\R$, the precise form of which will be defined later. The corresponding Euler--Lagrange equation is the second order partial differential equation $d_A^{\ast}F_A+\nabla\V(A)=0$, called \emph{perturbed Yang--Mills equation}. The (negative) $L^2$ gradient flow equation  associated with the functional $\YMV$ is the PDE
\begin{eqnarray}\label{introdYMgrad}
\partial_sA+d_A^{\ast}F_A+\nabla\V(A)=0.
\end{eqnarray}
The group $\G(P)$ of principal $G$-bundle automorphisms of $P$ acts on the space $\A(P)$ by gauge transformations, i.e.~as $g^{\ast}A\coloneqq g^{-1}Ag+g^{-1}dg$. The functional $\YMV$ is invariant under such gauge transformations, and hence are the solutions of the perturbed Yang--Mills (gradient flow) equations. The action is not free. The occuring stabilizer subgroups are Lie subgroups of $G$, hence finite-dimensional. Restricting the action to the group $\G_0(P)$ of so-called \emph{based gauge transformations}, i.e.~those transformations which fix a prescribed fibre of $P$ pointwise, one indeed obtains a free group action. For this reason we will study solutions to the gradient flow equation \eqref{introdYMgrad} only up to based gauge transformations, cf.~however the comment below concerning a $G$-equivariant extension of the theory.\\
\medskip\\
The study of the Yang--Mills functional over a Riemann surface from a Morse theoretical point of view has been initiated by Atiyah and Bott in \cite{AB}. One essential observation made there is that the based gauge equivalence classes of unperturbed (meaning that $\V=0$) Yang--Mills connections come as a family of finite-dimensional closed submanifolds of $\A(P)/\G_0(P)$. As discussed in detail in \cite{AB}, the unperturbed Yang--Mills functional satisfies the so-called Morse--Bott condition. In our context this condition asserts that, for any Yang--Mills connection $A\in\A(P)$, the kernel of the Hessian $H_A\YM$ coincides with the subspace of $T_A\A(P)$ comprising the tangent vectors at $A$ to the critical manifold containing $A$. Equivalently, the restriction of the operator $H_A\YM$ to the normal space at $A$ of this critical manifold is injective. Furthermore, the spectrum of $H_A\YM$ consists solely of eigenvalues, with a finite number of negative ones. Hence the situation one encounters for the functional $\YM$ over a Riemann surface parallels the one for Morse--Bott functions on finite-dimensional manifolds. The Morse theoretical approach taken by Atiyah and Bott indeed turned out to be very fruitful and had remarkable applications e.g.~to the cohomology of moduli spaces of stable vector bundles over $\Sigma$ (cf.~e.g.~\cite{Kirwan} for a review of these results). However, the literature so far still lacked a proper treatment of the analytical aspects of such a Morse--Bott theory and the underlying $L^2$ gradient flow \eqref{introdYMgrad}. The present article aims to close this gap and to introduce and work out in full analytical detail a Yang--Mills Morse homology theory over $\Sigma$.\\
\medskip\\
Let us now briefly describe our setup. Invariance of the functional $\YMV$ allows us to consider it as a map on the manifold $\A(P)/\G_0(P)$ of based gauge equivalence classes. Gauge-invariance also holds for the metric defined on $\A(P)$. One is therefore led to consider an $L^2$ gradient flow for $\YMV$ in this manifold of equivalence classes, giving rise to a well-defined equation for $[A]\in\A(P)/\G_0(P)$ of the form
\begin{eqnarray}\label{eq:eclasspertYMF}
[\partial_sA+d_A^{\ast}F_A+\nabla\V(A)]=0.
\end{eqnarray}
However, for analytical reasons it seems inconvenient to deal with equivalence classes. This can be avoided if we consider instead the equation
\begin{eqnarray}\label{eq:extpertYMF}
\partial_sA+d_A^{\ast}F_A+\nabla\V(A)-d_A\Psi=0
\end{eqnarray}
for some $\Psi\in C^{\infty}(\R,\Omega^0(\Sigma,\ad(P)))$. It is easy to see that every solution $(A,\Psi)$ of \eqref{eq:extpertYMF} is mapped under $(A,\Psi)\mapsto[A]$ to a solution of \eqref{eq:eclasspertYMF}, and this mapping is bijective up to equivalence under \emph{time-dependent} gauge transformations
\begin{eqnarray*}
g^{\ast}(A,\Psi)=(g^{\ast}A,g^{-1}\Psi g+g^{-1}\partial_sg).
\end{eqnarray*}
Every solution $(A,\Psi)$ of \eqref{eq:extpertYMF} is equivalent under a time-dependent gauge transformation $g$ to one where $\Psi=0$, the gauge transformation $g$ being obtained by solving the ODE $\partial_sg=\Psi g$. Hence in principle it is sufficient to only consider those solutions of \eqref{eq:extpertYMF} where $\Psi=0$. However, we will refrain from doing so, for the following reason. Due to the lack of parabolicity of the linearization of the Yang--Mills gradient flow equation \eqref{introdYMgrad} we often have to require a local slice condition to be able to obtain useful estimates for its solutions. For this it is desirable to be able to apply a time-dependent gauge transformation in order to put a given solution $(A,\Psi)$ of \eqref{eq:extpertYMF} into local slice with respect to some reference connection $(A_0,\Psi_0)$.\\
\medskip\\
Gauge fixing terms like $-d_A\Psi$ in \eqref{eq:extpertYMF} are often used in the analysis of equations invariant under an infinite-dimensional symmetry group. Examples from gauge theory include, amongst others, the Chern--Simons instanton equations as studied in the context of instanton Floer homology \cite{Donaldson1, SalWeh}, the symplectic vortex equations \cite{CGMS, Frauenfelder1}, or the Cauchy problem for the Yang--Mills gradient flow \cite{Struwe}.\\
\medskip\\
Since it is not known whether a Morse--Smale condition is automatically satisfied for the unperturbed Yang--Mills functional $\YM\colon A\mapsto\frac{1}{2}\int_{\Sigma}\langle F_A\wedge\ast F_A\rangle$, we have to introduce perturbations for transversality reasons. Here we shall work with a Banach space $Y$ of so-called \emph{abstract perturbations} $\V\colon\A(P)\to\R$. This space $Y$ is generated by a countable set of gauge-invariant model perturbations $\V_{\ell}$ of the form
\begin{eqnarray*}
\V_{\ell}(A)\coloneqq\rho\big(\|\alpha(A)\|_{L^2(\Sigma)}^2\big)\langle\eta,\alpha(A)\rangle,
\end{eqnarray*}
with $\rho=\rho(\ell)\colon\R\to\R$ a cut-off function, $\eta=\eta(\ell)\in\Omega^1(\Sigma,\ad(P))$ a fixed $\ad(P)$-valued $1$-form, and $\alpha(A)=g^{\ast}A-A_0$. Here $g\in\G(P)$ is chosen such that the local slice condition $d_{A_0}^{\ast}\alpha=0$ is satisfied with respect to some reference connection $A_0=A_0(\ell)$. Our construction of model perturbations relies crucially on the recent $L^2$ local slice theorem by Mrowka and Wehrheim \cite{MrowkaWehrheim}. The space $Y$ of perturbations is sufficiently flexible to achieve transversality of Fredholm sections as we shall describe below. This approach to transversality draws from ideas successfully used by Weber \cite{Web} in the related situation of the heat flow for loops on a compact manifold. Let $a>0$ be a fixed regular value of $\YM$. From now on we admit only perturbations $\V\in Y$ supported outside some $L^2$ neighborhood of the critical manifolds of $\YM$ below the energy level $a$. We define
\begin{eqnarray*}
\p(a)\coloneqq\frac{\{A\in\A(P)\mid d_A^{\ast}F_A=0\;\textrm{and}\;\YM(A)\leq a\}}{\G_0(P)}
\end{eqnarray*}
to be the set of based gauge equivalence classes of Yang--Mills connections of energy at most $a$. On $\p(a)$ we fix a Morse function $h\colon\p(a)\to\R$, i.e.~a smooth function $h$ with isolated non-degenerate critical points whose stable and unstable manifolds intersect transversally. To a critical point $x$ of $h$ (which in particular is a critical point of $\YM$) we assign the non-negative number
\begin{eqnarray*}
\Ind(x)\coloneqq\ind_{\YM}(x)+\ind_h(x),
\end{eqnarray*}
where $\ind_h(x)$ is the usual Morse index of $x$ with respect to $h$ and $\ind_{\YM}(x)$ denotes the number of negative eigenvalues (counted with multiplicities) of the Yang--Mills Hessian $H_x\YM$. In order to keep the presentation as short as possible and avoid to discuss orientation issues we use coefficients in $\mathbbm Z_2=\mathbbm Z/2\mathbbm Z$ for the construction of the Yang--Mills Morse complex. For a regular value $a$ of $\YM$ we thus consider the $\mathbbm Z_2$ vector space
\begin{eqnarray*}
CM_{\ast}^a(\A(P)/\G_0(P),\YM,h)\coloneqq\bigoplus_{x\in\crit(h)}\langle x\rangle
\end{eqnarray*}                                                                              
generated by the set $\crit(h)\subseteq\p(a)$ of critical points of $h$. This is a finite-dimensional vector space which is graded by the index $\Ind$. Under certain transversality assumptions (which resemble the usual Morse--Smale transversality required in finite-dimensional Morse theory) there is a well-defined boundary operator 
\begin{eqnarray*}
\partial_{\ast}^{\V}\colon CM_{\ast}^a(\A(P)/\G_0(P),\YM,h)\to CM_{\ast-1}^a(\A(P)/\G_0(P),\YM,h)
\end{eqnarray*}
which arises from counting so-called \emph{cascade configurations} of (negative) $L^2$ gradient flow lines. These are finite tuples of gauge equivalence classes of solutions of the perturbed Yang--Mills gradient flow equation \eqref{eq:extpertYMF} whose asymptotics as $s\to\pm\infty$ obey a certain compatibility condition.

\subsection{Main results}

The purpose of the present work is to establish the following result. 

\begin{thm}[Main result]\label{thm:mainresult}
Let $a\geq0$ be a regular value of $\YM$. For any Morse function $h\colon\p(a)\to\R$ and generic perturbation $\V\in Y$ (which in addition is $a$-admissible in the sense of Definition \ref{def:regperturbation} and satisfies $\|\V\|<\delta$ for some sufficiently small constant $\delta=\delta(a)>0$), the map $\partial_{\ast}^{\V}$ satisfies $\partial_k^{\V}\circ\partial_{k+1}^{\V}=0$ for all $k\in\mathbbm N_0$ and thus there exist well-defined homology groups
\begin{eqnarray*}
HM_k^a\big(\A(P)/\G_0(P),\YMV,h\big)=\frac{\ker\partial_k^{\V}}{\im\partial_{k+1}^{\V}}.
\end{eqnarray*}
The homology $HM_{\ast}^a(\A(P)/\G_0(P),\YMV,h)$ is called {\bf{\emph{Yang--Mills Morse homology}}}. It is independent of the choice of perturbation $\V$ and Morse function $h$.
\end{thm}

A proof of this result is given in Section \ref{YMMorsehomology}. The preceding sections are of preparatory nature and comprise the more technical parts of the paper. In Section \ref{sec:YMfunctional} we recall some known facts about the unperturbed Yang--Mills functional and introduce the Banach space $Y$ of perturbations needed later on to make the transversality theory work. A collection of relevant properties and estimates involving the perturbations is contained in the appendix. In Section \ref{sec:YMgradientflow} the moduli space problem for Yang--Mills gradient flow lines with prescribed asymptotics as $s\to\pm\infty$ is put into an abstract Banach manifold setting. The moduli space $\hat{\mathcal M}(\mathcal C^-,\mathcal C^+)$ of gradient flow lines connecting a given pair $(\mathcal C^-,\mathcal C^+)$ of Yang--Mills critical manifolds is exhibited as the zero set of a section $\mathcal F$ of a suitably defined Banach space bundle. The necessary Fredholm theory for the differential operator obtained by linearizing $\mathcal F$ is developed in Section \ref{sec:Fredholm theorem}. In Section \ref{sec:Expdecay} we show exponential decay of any solution $A\colon\R\to\A(P)$ of \eqref{eq:extpertYMF} (satisfying for a fixed energy level $a$ that $\limsup_{s\to-\infty}\YMV(A(s))\leq a$) to a pair $A^{\pm}$ of Yang--Mills connections as $s\to\pm\infty$. The issue of compactness of moduli spaces is addressed in Section \ref{sec:Compactness}. As a main result we show there compactness up to gauge transformations for sequences of solutions of \eqref{introdYMgrad} on compact subsets of $\R\times\Sigma$. The proof of this result (Theorem \ref{thm:compactness}) relies on certain a priori $L^p$ estimates for the curvature form $F_A$, the weak Uhlenbeck compactness theorem (cf.~\cite{Wehrheim}), and a combination of elliptic and parabolic regularity estimates. Transversality of the section $\mathcal F$ at a zero $x\in\mathcal F^{-1}(0)$ is discussed in Section \ref{sect:transversality}. We show, along the usual lines involving Sard's lemma, that surjectivity of the linearized operator holds for generic perturbations $\mathcal V\in Y$.

\subsection{Further comments and related results}

\subsubsection*{Equivariant theory}
For the ease of presentation we develop here a non-equivariant Yang--Mills Morse theory on the space $\A(P)/\G_0(P)$ of \emph{based} gauge equivalence classes of connections. Alternatively, it seems possible to take a $G$-equivariant approach by extending the setup to the space $\A(P)\times E_nG$, with $E_nG$ a suitable finite-dimensional approximation to the classifying space $EG$. This space carries a free action by the full group $\G(P)$ of gauge transformations via
\begin{eqnarray*}
g^{\ast}(A,\lambda)=(g^{\ast}A,\hat g\lambda),
\end{eqnarray*}
where the map $g\mapsto \hat g\in G$ is given by evaluating $g$ at some fixed $p\in P$. By extending $\YMV$ in a suitable $\G(P)$-invariant way, we expect our construction of Morse homology groups to carry over almost literally to the quotient manifold $(\A(P)\times E_nG)/\G(P)$.

\subsubsection*{Connection with Morse homology of loop groups}
Yang--Mills Morse homology is strongly related to the recently introduced heat flow homology due to Weber \cite{Web}, at least in the case of the sphere $\Sigma=S^2$. This connection is due to the following result, cf.~\cite{Swoboda,Swoboda1}. For a compact Lie group $G$ we let $\Omega G$ denote the associated based loop group, i.e.~the space $\Omega G\coloneqq\{\gamma\in C^{\infty}(S^1,G)\mid\gamma(1)=\mathbbm 1\}$ with group structure given by pointwise multiplication. The classical action functional
\begin{eqnarray*}
\mathcal E\colon\Omega G\to\R,\quad\gamma\mapsto\frac{1}{2}\int_0^1\|\partial_t\gamma(t)\|^2\,dt
\end{eqnarray*}
satisfies the Morse--Bott condition. Its critical manifolds are the orbits of closed geodesics under the action of $G$ by conjugation (i.e.~by $(g\cdot\gamma)(t)=g^{-1}\gamma(t)g$ for $g\in G$ and $\gamma\in\Omega G$). Let $b>0$ be a regular value of $\E$ and $h\colon\crit^b(\E)\to\R$ a Morse function on the manifold of critical points of $\E$ below the level $b$. Then as a special case of heat flow homology we can build a chain complex $CM_{\ast}^b\big(\Omega G,\E,h)$ generated by the critical points of $h$. A boundary operator is obtained as in Yang--Mills Morse homology by counting appropriate (cascades of) $L^2$ gradient flow lines of a suitably perturbed action functional $\E^{\V}$. These are formed by solutions of the perturbed nonlinear heat equation
\begin{eqnarray*}
\partial_s\gamma-\nabla_t\partial_t\gamma+\nabla\V(\gamma)=0,  
\end{eqnarray*}
converging, as $s\to\pm\infty$, to critical points $\gamma^{\pm}\in\crit^b(\E)$. It is worth mentioning that by a result going back to Atiyah and Bott \cite{AB} there exists a bijection between the sets of Yang--Mills connections over $S^2$ and closed geodesics on $G$. Building on this correspondence we obtained in \cite{Swoboda,Swoboda1} the following result.

\begin{thm}\label{thm:Swoboda1result} 
Let $G$ be any compact Lie group and $P$ be any principal $G$-bundle $P$ over $\Sigma$. Let $a\geq0$ be a regular value of $\YM$ and set $b:=4a/\pi$. Then, for a generic perturbation $\V=(\V^-,\V^+)$ (in a suitably defined Banach space of perturbations) there exists a natural chain homomorphism
\begin{eqnarray*}
\Theta_{\ast}^{\V}\colon CM_{\ast}^a\big(\A(P)/\mathcal G_0(P),\YM,h\big)\to CM_{\ast}^b\big(\Omega G,\E,h\big),
\end{eqnarray*}
inducing an isomorphism
\begin{eqnarray*}
[\Theta_{\ast}^{\V}]\colon HM_{\ast}^a\big(\A(P)/\mathcal G_0(P),\YM^{\V^-},h\big)\to HM_{\ast}^b\big(\Omega G,\E^{\V^+},h\big)
\end{eqnarray*}
of Morse homology groups.
\end{thm}
It would be interesting to work out a similar correspondence in the case where $\Sigma$ is a Riemann surface of arbitrary genus. In yet another direction one can take into account the action of the Lie group $G$ on $\Omega G$ by conjugation to set up a $G$-equivariant version of heat flow homology for $\Omega G$ (or more generally, for any compact manifold with $G$ action). As discussed in \cite{Swoboda}, but not shown in full detail so far, one can expect a result similar to Theorem \ref{thm:Swoboda1result} to hold true, relating $G$-equivariant Yang--Mills Morse homology as described above to a $G$-equivariant version of heat flow homology.

\subsubsection*{Products}
In finite dimensional Morse homology it is well known how to implement a module structure, cf.~the monograph \cite{Schw}. In infinite dimensional situations one often encounters similar algebraic structures, like e.g.~the quantum product in Floer homology or the Chas--Sullivan loop product in the Morse homology of certain loop spaces, cf.~\cite{AbSchwarz,AbSchwarz1,BiranCornea, CohenSchwarz}. Using finite-dimensional Morse homology as a guiding principle, one should be able to implement a natural product structure in the setup presented here. In a subsequent step one could ask how this relates to products in loop space homology of $\Omega G$.

\subsubsection*{Related work}
For finite dimensional manifolds, the construction of a Morse homology theory from the set of critical points of a Morse functions and the isolated flow lines connecting them goes back to Thom \cite{Thom}, Smale \cite{Smale} and Milnor \cite{Milnor}, and had later been rediscovered by Witten \cite{Witten}. For a historical account we refer to the survey paper by Bott \cite{Bott1}. In infinite dimensions the same sort of ideas underlies the construction of Floer homology of compact symplectic manifolds (cf.~the expository notes \cite{Sal2}), although the equations encountered there are of elliptic rather than parabolic type. More in the spirit of classical finite dimensional Morse homology is the aforementioned heat flow homology for the loop space of a compact manifold due to Weber \cite{Web}, which is based on the $L^2$ gradient flow of the classical action functional. For another approach via the theory of ODEs on Hilbert manifolds and further references, see Abbondandolo and Majer \cite{AbMajer}. The cascade construction of Morse homology in the presence of critical manifolds satisfying the Morse--Bott condition is due to Frauenfelder \cite{Frauenfelder2}.

\subsubsection*{Acknowledgements}
This work extends the author's PhD thesis \cite{Swoboda}. He would like to thank his supervisor Dietmar A.~Salamon for his encouragement while working on this project. For useful comments and discussions he also thanks W.~Ballmann, R.~Janner, M.~Schwarz, M.~Struwe, and J.~Weber. He is grateful to the referee for his valuable suggestions to improve the overall readability of this article.

\section{Yang--Mills functional}\label{sec:YMfunctional}

\subsection{Preliminaries}\label{sec:prelim}
Let $(\Sigma,g)$ be a compact oriented Riemann surface. Let $G$ be a compact Lie group with Lie algebra $\mathfrak g$. On $\mathfrak g$ we fix an $\ad$-invariant inner product $\langle\,\cdot\,,\,\cdot\,\rangle$, which exists by compactness of $G$. Let $P$ be a principal $G$-bundle over $\Sigma$. A gauge transformation is a section of the bundle $\Ad(P)\coloneqq P\times_GG$ associated to $P$ via the action of $G$ on itself by conjugation $(g,h)\mapsto g^{-1}hg$. Let $\ad(P)$ denote the Lie algebra bundle associated to $P$ via the adjoint action
\begin{eqnarray*}
(g,\xi)\mapsto\left.\frac{d}{dt}\right|_{t=0}g^{-1}\exp(t\xi)g\qquad(\textrm{for}\,g\in G,\,\xi\in\mathfrak g)
\end{eqnarray*}
of $G$ on $\mathfrak g$. We denote the space of smooth $\ad(P)$-valued differential $k$-forms by $\Omega^k(\Sigma,\ad(P))$, and by $\A(P)$ the space of smooth connections on $P$. The latter is an affine space over $\Omega^1(\Sigma,\ad(P))$. The group $\G(P)$ acts on $\A(P)$ by gauge transformations. We call a connection $A\in\A(P)$ \emph{irreducible} if the stabilizer subgroup $\Stab A\subseteq\G(P)$ is trivial. Otherwise it is called \emph{reducible}. It is easy to show that $\Stab A$ is a compact Lie group, isomorphic to a subgroup of $G$. Let $z\in\Sigma$ be arbitrary but fixed. We let $\G_0(P)\subseteq\G(P)$ denote the group of \emph{based gauge transformation}, i.e.~those gauge transformations which leave the fibre $P_z\subseteq P$ above $z$ pointwise fixed. It is a well-known fact that $\G_0(P)$ acts freely on $\A(P)$.\\
\medskip\\
On $\A(P)$ we define a gauge-invariant $L^2$ inner product by $\langle\alpha,\beta\rangle=\int_{\Sigma}\langle\alpha\wedge\ast\beta\rangle$ for $\alpha,\beta\in\Omega^1(\Sigma,\ad(P))$. The curvature of the connection $A$ is the $\ad(P)$-valued $2$-form $F_A=dA+\frac{1}{2}[A\wedge A]$. It satisfies the Bianchi identity $d_AF_A=0$. Covariant differentiation with respect to the Levi--Civita connection associated with the metric $g$ and a connection $A\in\A(P)$ defines an operator $\nabla_A\colon\Omega^k(\Sigma,\ad(P))\to\Omega^1(\Sigma)\otimes\Omega^k(\Sigma,\ad(P))$. The covariant exterior differential induced by $A\in\A(P)$ is the operator
\begin{eqnarray*}
d_A\colon\Omega^k(\Sigma,\ad(P))\to\Omega^{k+1}(\Sigma,\ad(P)),\quad\alpha\mapsto d\alpha+[A\wedge\alpha].
\end{eqnarray*}
The formal adjoints of these operators are denoted by $\nabla_A^{\ast}$ and $d_A^{\ast}$. The covariant Hodge Laplacian on forms is the operator $\Delta_A\coloneqq d_A^{\ast}d_A+d_Ad_A^{\ast}$, the covariant Bochner Laplacian on forms is $\nabla_A^{\ast}\nabla_A$. They are related through the Bochner--Weitzenb\"ock formula
\begin{eqnarray}\label{BWformula}
\Delta_A=\nabla_A^{\ast}\nabla_A+\{F_A,\,\cdot\,\}+\{R_{\Sigma},\,\cdot\,\}.
\end{eqnarray}
Here the brackets $\{\,\cdot\,,\,\cdot\,\}$ denote $C^{\infty}$-bilinear expressions with coefficients independent of $A$, and $R_{\Sigma}$ is a term involving the Riemann curvature tensor of $(\Sigma,g)$. The \emph{perturbed Yang--Mills functional} $\YMV$ has been introduced in \eqref{YMintroduction}. If $\V=0$, we write $\YM$ and call this the \emph{unperturbed Yang--Mills functional}. The $L^2$ gradient of $\YMV$ at $A\in\A(P)$ is
\begin{eqnarray*}
\nabla\YMV(A)=d_A^{\ast}F_A+\nabla\V(A)\in\Omega^1(\Sigma,\ad(P)).
\end{eqnarray*}
Its Hessian is the second order differential operator
\begin{eqnarray}\label{eq:YMHess}
H_A\YMV=d_A^{\ast}d_A+\ast[\ast F_A\wedge\,\cdot\,]+H_A\V\colon\Omega^1(\Sigma,\ad(P))\to\Omega^1(\Sigma,\ad(P)).
\end{eqnarray}
We also make use of the notation $H_A\coloneqq d_A^{\ast}d_A+\ast[\ast F_A\wedge\,\cdot\,]$.\\
\medskip\\
Throughout we will use Sobolev spaces of sections of vector bundles and Banach manifolds modeled on such Sobolev spaces, like e.g.~various groups of gauge transformations. A detailed account of this subject is given in the book \cite[Appendix B]{Wehrheim}. We therefore keep the discussion of these matters short. Let $1\leq p\leq\infty$ and $k\geq0$ an integer. We fix a smooth reference connection $A\in\A(P)$. It determines a covariant derivative $\nabla_A$ on $\Omega^{\ast}(\Sigma,\ad(P))$ as above. We employ the notation $W^{k,p}(\Sigma,\ad(P))$ and $W^{k,p}(\Sigma,T^{\ast}\Sigma\otimes\ad(P))$ for the Sobolev spaces of $\ad(P)$-valued $0$- and $1$-forms whose weak derivatives (with respect to $\nabla_A$) up to order $k$ are in $L^p$. These spaces are independent of the choice of $A$. However, for $k\geq1$, the corresponding norms depend on this choice.  The standard Sobolev embedding and Rellich--Kondrachov compactness theorems apply to these spaces. The affine $(k,p)$-Sobolev space of connections on $P$ is defined as
\begin{eqnarray*}
\A^{k,p}(P)\coloneqq A+W^{k,p}(\Sigma,T^{\ast}\Sigma\otimes\ad(P)).
\end{eqnarray*}
This definition is again independent of the choice of smooth reference connection $A$. To define Sobolev spaces of gauge transformations we need to assume $kp>\dim\Sigma=2$. Then let $\G^{k,p}(P)$ denote the set of equivariant maps $P\to G$ which are of the form $g=g_0\exp(\ph)$, where $g_0$ is a smooth such map and $\ph\in W^{k,p}(\Sigma,\ad(P))$. (Here we view $\ph$ as an equivariant map $P\to\mathfrak g$). The space $\G^{k,p}(P)$ is a Banach manifold modeled on $W^{k,p}(\Sigma,\ad(P))$. As a well-known fact we remark that $\G^{k,p}(P)$ is a group with smooth group multiplication and inversion. It acts smoothly on $\A^{k-1,p}(P)$ by gauge transformations. Let $I$ be a finite or infinite interval. We often make use of the parabolic Sobolev space
\begin{eqnarray}\label{eq:defparabolicSobolev}
W^{1,2;p}(I\times\Sigma,\ad(P))\coloneqq L^p(I,W^{2,p}(\Sigma,\ad(P)))\cap W^{1,p}(I,L^p(\Sigma,\ad(P)))
\end{eqnarray} 
of $\ad(P)$-valued $0$-forms admitting one time and two space derivatives in $L^p$ (and similarly for $\ad(P)$-valued $1$-forms). The parabolic Sobolev space $\A^{1,2;p}(P)$ of connections is defined analogously, with $W^{2,p}(\Sigma,\ad(P))$ and $L^p(\Sigma,\ad(P))$ in \eqref{eq:defparabolicSobolev} replaced by $\A^{2,p}(P)$, respectively $\A^{0,p}(P)$. Note that when there is no danger of confusion, we for ease of notation write $L^p(\Sigma)$ instead of $L^p(\Sigma,\ad(P))$ or $L^p(\Sigma,T^{\ast}\Sigma\otimes\ad(P))$ (and similarly for the other Sobolev spaces). Further notation frequently used is $\dot A\coloneqq\partial_sA\coloneqq\frac{d A}{ds}$, etc.~for derivatives with respect to time.

\subsection{Critical manifolds}\label{sect:criticalmanifolds}

We introduce some further notation and recall several results from \cite{AB,Rade,Rade2} concerning the set of critical points of the \emph{unperturbed} Yang--Mills functional $\YM$. Let
\begin{eqnarray}\label{eq:defcritmanifold}
\crit(\YM)\coloneqq\{A\in\A^{1,p}(P)\mid d_A^{\ast}F_A=0\}
\end{eqnarray}
denote the set of critical points of $\YM$, the equation $d_A^{\ast}F_A=0$ being understood in the weak sense. Similarly, the notation $\crit(\YMV)$ refers to the set of critical points of the perturbed Yang--Mills functional. We let $\mathcal{CR}$ denote the set of connected components of $\crit(\YM)$. The group $\G_0^{2,p}(P)$ of based gauge transformations of class $W^{2,p}$ acts freely on each $\mathcal C\in\mathcal{CR}$. The quotient $\mathcal C/\G_0^{2,p}(P)$ is a finite-dimensional compact smooth manifold, cf.~\cite[Section 2]{Rade}. Below a given level set $a>0$ there exist at most finitely many critical manifolds as the following proposition shows.

\begin{prop}\label{prop:finiteset}
For $a\geq0$ let $\mathcal{CR}^a\subseteq\mathcal{CR}$ denote the set of critical manifolds of Yang--Mills energy at most $a$. Then $\mathcal{CR}^a$ is a finite set, for every real number $a\geq0$. 
\end{prop}

\begin{proof}
Assume by contradiction that there exists a sequence $(\mathcal C_{\nu})\subseteq\mathcal{CR}^a$ of pairwise different critical manifolds $\mathcal C_{\nu}$. For each $\nu$ we fix a Yang--Mills connection $A_{\nu}\in\mathcal C_{\nu}$. From $\mathcal C_{\nu}\in\mathcal{CR}^a$ for all $\nu$ it follows that the sequence of numbers $\|F_{A^{\nu}}\|_{L^2(\Sigma)}$ is bounded. Hence by Uhlenbeck's strong compactness theorem (cf.~\cite[Theorem E]{Wehrheim}), after modifying each $A_{\nu}\in\mathcal C_{\nu}$ by suitable gauge transformations and passing to a subsequence, the sequence $(A_{\nu})$ converges uniformly with all derivatives to a smooth connection $A_{\ast}\in\A(P)$. The limit connection $A_{\ast}$ is Yang--Mills and has energy at most $a$, hence $A\in\mathcal C_{\ast}$ for some critical manifold $\mathcal C_{\ast}\in\mathcal{CR}^a$. It follows that $\mathcal C_{\ast}$ is not isolated within the set $\mathcal{CR}^a$ (with respect to any $C^k(\Sigma)$ topology), a contradiction to the fact that $\mathcal C_{\ast}$ satisfies the Morse--Bott condition.
\end{proof}

The Yang--Mills functional $\YM$ satisfies an equivariant version of the Palais--Smale condition in dimension $2$ (which holds true also in dimension $3$ but not in higher dimensions).

\begin{defn}\upshape
A sequence $(A_{\nu})\subseteq\A(P)$ is said to be a \emph{Palais--Smale sequence} if there exists $M>0$ such that $\|F_{A_{\nu}}\|_{L^2(\Sigma)}\leq M$ for all $\nu\in\mathbbm N$, and
\begin{eqnarray*}
\|d_{A_{\nu}}^{\ast}F_{A_{\nu}}\|_{W^{-1,2}(\Sigma)}\to0\qquad\textrm{as}\quad\nu\to\infty.
\end{eqnarray*}
\end{defn}

\begin{thm}[Equivariant Palais--Smale condition]\label{thm:equiPalaisSmale}
Let $(A_{\nu})\subseteq\A(P)$ be a Palais--Smale sequence. Then there exists a subsequence, again denoted by $(A_{\nu})$, and a sequence $(g_{\nu})\subseteq\G(P)$ such that $g_{\nu}^{\ast}A_{\nu}$ converges in $\A^{1,2}(P)$ to a weak Yang--Mills connection $A_{\ast}$ as $\nu\to\infty$.
\end{thm}

\begin{proof} 
For a proof we refer to \cite[Theorem 1]{Rade2}.
\end{proof}

\begin{prop}\label{lem:PalaisSmale} 
Let $\varepsilon,M>0$. There exists a constant $\delta=\delta(\eps,M)>0$ with the following significance. Let $A\in\A^{1,2}(P)$ satisfy $\|F_A\|_{L^2(\Sigma)}\leq M$ and
\begin{eqnarray}\label{eq:L2saparate}
\|A-A_0\|_{L^2(\Sigma)}>\varepsilon
\end{eqnarray}
for every weak Yang--Mills connection $A_0\in\A^{1,2}(P)$ such that $\|F_{A_0}\|_{L^2(\Sigma)}\leq M$. Then $\|d_A^{\ast}F_A\|_{W^{-1,2}(\Sigma)}>\delta$.
\end{prop}

\begin{proof}
Assume by contradiction that there exists a sequence $(A_{\nu})\subseteq\A(P)$ satisfying $\|F_{A_{\nu}}\|_{L^2(\Sigma)}\leq M$ and \eqref{eq:L2saparate} with 
\begin{eqnarray*}
\lim_{\nu\to\infty}\|d_{A_{\nu}}^{\ast}F_{A_{\nu}}\|_{W^{-1,2}(\Sigma)}=0.
\end{eqnarray*}
Then by Theorem \ref{thm:equiPalaisSmale} there exist a subsequence, still denoted $(A_{\nu})$, a sequence of gauge transformations $(g_{\nu})\subseteq\G(P)$, and a Yang--Mills connection $A_{\ast}$ with $\lim_{\nu\to\infty}g_{\nu}^{\ast}A_{\nu}=A_{\ast}$ in $\A^{1,2}(P)$, hence also in $\A^{0,2}(P)$. Because the map $A\mapsto F_A\colon\A^{1,2}(P)\to L^2(\Sigma)$ is continuous it follows that $\|F_{A_{\ast}}\|_{L^2(\Sigma)}\leq M$. Moreover, for a sufficiently large integer $\nu$, the Yang--Mills connection $A_0\coloneqq(g_{\nu}^{-1})^{\ast}A_{\ast}$ satisfies $\|A_{\nu}-A_0\|_{L^2(\Sigma)}<\varepsilon$, contradicting \eqref{eq:L2saparate}.  
\end{proof}

\subsection{Banach space of abstract perturbations}\label{Bspaceperturbations}

Our construction of a Banach space of perturbations is based on the following $L^2$ local slice theorem due to Mrowka and Wehrheim \cite{MrowkaWehrheim}. We fix $p>2$ and let
\begin{eqnarray*}
\S_{A_0}(\varepsilon)\coloneqq\big\{A=A_0+\alpha\in\A^{0,p}(P)\,\big|\,d_{A_0}^{\ast}\alpha=0,\|\alpha\|_{L^2(\Sigma)}<\varepsilon\big\}
\end{eqnarray*}
denote the set of $L^p$-connections in the local slice of radius $\varepsilon$ with respect to the reference connection $A_0\in\A^{0,p}(P)$.

\begin{thm}[$L^2$ local slice theorem]\label{locslicethm} 
Let $p>2$. For every $A_0\in\A^{0,p}(P)$ there are constants $\varepsilon,\delta>0$ such that the map
\begin{multline*}
\mathfrak m\colon\big(\S_{A_0}(\varepsilon)\times\G^{1,p}(P)\big)/\Stab{A_0}\to\A^{0,p}(P),\\
[(A_0+\alpha,g)]\mapsto(g^{-1})^{\ast}(A_0+\alpha)
\end{multline*}
is a diffeomorphism onto its image, which contains an $L^2$ ball,
\begin{eqnarray*}
B_{\delta}(A_0)\coloneqq\big\{A\in\A^{0,p}(P)\,\big|\,\|A-A_0\|_{L^2(\Sigma)}<\delta\big\}\subseteq\im\mathfrak m.
\end{eqnarray*}
\end{thm}

\begin{proof}
For a proof we refer to \cite[Theorem 1.7]{MrowkaWehrheim}.
\end{proof}

We fix the following data.

\begin{compactenum}[(i)]
\item
A dense sequence $(A_i)_{i\in\mathbbm N}$ of irreducible smooth connections in $\A(P)$.
\item
For each $i\in\mathbbm N$ a dense sequence $(\eta_{ij})_{j\in\mathbbm N}$ of smooth $1$-forms in $\Omega^1(\Sigma,\ad(P))$ satisfying $d_{A_i}^{\ast}\eta_{ij}=0$ for all $j\in\mathbbm N$.
\item
A smooth cutoff function $\rho\colon\mathbbm R\to[0,1]$ such that $\rho=1$ on $[-1,1]$, $\supp\rho\subseteq[-4,4]$, and $\|\rho'\|_{L^{\infty}(\R)}<1$. Set $\rho_k(r)\coloneqq\rho(k^2r)$ for $k\in\mathbbm N$.
\end{compactenum}
We fix $i\in\mathbbm N$ and a constant $\eps_i>0$ such that the conclusion of Theorem \ref{locslicethm} applies for $A_0\coloneqq A_i$ and this constant $\eps_i$. Note that by assumption, $\Stab A_i=\mathbbm 1$. Theorem \ref{locslicethm} thus implies that the map
\begin{eqnarray*}
\mathfrak m_i\colon\mathcal S_{A_i}(\eps_i)\times\G^{1,p}(P)\to\A^{0,p}(P),\quad(A_i+\alpha,g)\mapsto(g^{-1})^{\ast}(A_i+\alpha)
\end{eqnarray*}
is a diffeomorphism onto its image. Hence 
\begin{eqnarray}\label{def:alphai}
\alpha_i\colon\im\mathfrak m_i\to L^p(\Sigma,\ad(P)),\quad A\mapsto(\pr_1\circ\mathfrak m^{-1})(A)-A_i
\end{eqnarray}
(with $\pr_1\colon\mathcal S_{A_i}(\eps_i)\times\G^{1,p}(P)\to\mathcal S_{A_i}(\eps_i)$ denoting projection) is a well-defined smooth map with image being contained in $\mathcal S_{A_i}(\eps_i)-A_i$. We extend $\alpha_i$ to a map on $\A(P)$ by setting $\alpha_i(A)=0$ for $A\in\A^{0,p}(P)\setminus\im\mathfrak  m_i$. Hence 
\begin{eqnarray}\label{modelperturbation}
\V_{\ell}\colon\A^{0,p}(P)\to\mathbbm R,\qquad A\mapsto\rho_k(\|\alpha_i(A)\|_{L^2(\Sigma)}^2)\langle\alpha_i(A),\eta_{ij}\rangle
\end{eqnarray}
is a well-defined map for every triple $\ell=(i,j,k)\in\mathbbm N^3$. Note also that $\V_{\ell}$ is invariant under the action of $\G^{1,p}(P)$ by gauge transformations.

\begin{prop}\label{prop:condlocslicepert}
For every $A\in\im\mathfrak m_i$ there exists a unique $g\in\G^{1,p}(P)$ such that
\begin{eqnarray}\label{condlocslicepert}
g^{\ast}A-A_i=\alpha_i(A)\qquad\textrm{and}\qquad d_{A_i}^{\ast}\alpha_i(A)=0.
\end{eqnarray}
\end{prop}

\begin{proof}
Set $g^{-1}\coloneqq(\pr_2\circ\mathfrak m_i^{-1})(A)$ with $\pr_2$ denoting projection to the second factor in $\S_{A_i}(\eps_i)\times\G^{1,p}(P)$. That $g$ satisfies the first identity in \eqref{condlocslicepert} follows from the definition of $\mathfrak m_i$. Uniqueness is a consequence of injectivity of $\mathfrak m_i$. The second identity is satisfied because $A_i+\alpha_i(A)$ is by definition contained in the local slice $\S_{A_i}(\eps_i)$.
\end{proof}

For $i\in\mathbbm N$ we fix a constant $\delta_i>0$ such that the $L^2$ ball $B_{\delta_i}(A_i)$ is contained in $\im\mathfrak m_i$. The existence of such $\delta_i$ follows from Theorem \ref{locslicethm}. We denote
\begin{eqnarray*}
X_i\coloneqq\G^{1,p}(P)^{\ast}\{A\in\A^{0,p}(P)\mid\|A-A_i\|_{L^2(\Sigma)}<\delta_i\}.
\end{eqnarray*}
Note that $X_i\subseteq\im\mathfrak m_i$ by gauge invariance of the set $\im\mathfrak m_i$.

\begin{prop}\label{prop:perturbsmooth}
Let $\ell=(i,j,k)\in\mathbbm N^3$ such that $k>\frac{10}{\delta_i}$. Then the map $\V_{\ell}\colon\A^{0,p}(P)\to\mathbbm R$ defined in \eqref{modelperturbation} is smooth.
\end{prop}

\begin{proof}
Let $A\in X_i$ be given. We first show smoothness of $\V_{\ell}$ at the point $A$. Assuming that $A$ has an $\A^{0,p}(P)$ neighborhood $U\subseteq\im\mathfrak m_i$ the claim follows from smoothness of the map $\alpha_i|_{\im\mathfrak m_i}$ as given by \eqref{def:alphai}. Namely then, $\V_{\ell}|_U$ arises as composition of $\alpha_i|_U$ with smooth maps. To prove the existence of such a neighborhood $U$ we assume by contradiction that there is a sequence $(A^{\nu})\subseteq\A^{0,p}(P)\setminus\im\mathfrak m_i$ which converges in $\A^{0,p}(P)$, hence also in $\A^{0,2}(P)$ to $A$. As $A\in X_i$ holds by assumption there exists some $g\in\G^{1,p}(P)$ such that $\delta\coloneqq\|g^{\ast}A-A_i\|_{L^2(\Sigma)}<\delta_i$. Set $\delta_1\coloneqq\frac{1}{2}(\delta_i-\delta)>0$. Because $A^{\nu}\to A$ as $\nu\to\infty$ in $\A^{0,2}(P)$ it follows for every sufficiently large $\nu\geq\nu_0$ that $\|A^{\nu}-A\|_{L^2(\Sigma)}<\delta_1$. Since the $L^2$ norm is preserved under gauge transformations this implies that $\|g^{\ast}A^{\nu}-g^{\ast}A\|_{L^2(\Sigma)}<\delta_1$. Thus by choice of $\delta_1$ we see that for $\nu\geq\nu_0$ the connection $g^{\ast}A^{\nu}\in W^{1,p}(\Sigma)$ satisfies
\begin{eqnarray*}
\|g^{\ast}A^{\nu}-A_i\|_{L^2(\Sigma)}\leq\|g^{\ast}A^{\nu}-g^{\ast}A\|_{L^2(\Sigma)}+\|g^{\ast}A-A_i\|_{L^2(\Sigma)}<\delta_1+\delta<\delta_i.
\end{eqnarray*}
Hence $g^{\ast}A^{\nu}\in\im\mathfrak m_i$ by choice of the constant $\delta_i$. Invariance of $\im\mathfrak m_i$ under the action of $\G^{1,p}(P)$ now implies that $A^{\nu}\in\im\mathfrak m_i$ for all $\nu\geq\nu_0$. This contradicts our assumption and proofs the claim. Now let $A\in\A^{0,p}(P)\setminus X_i$. We argue that for $k>\frac{10}{\delta_i}$ there exists a neighborhood $U$ of $A$ in $\A^{0,p}(P)$ such that $\V_{\ell}|_U=0$. This implies smoothness of $\V_{\ell}$ at the point $A$. Assume by contradiction that such $U$ does not exist. Then there is a sequence $A^{\nu}\in\A^{0,p}(P)$ such that $\lim_{\nu\to\infty}A^{\nu}=A$ in $\A^{0,p}(P)$, hence also in $\A^{0,2}(P)$, and $\V_{\ell}(A^{\nu})\neq0$ for all $\nu$. By definition of $\V_{\ell}$ and choice of $k<\frac{10}{\delta_i}$ this implies that $A^{\nu}\in X_i$ and $\|\alpha_i(A^{\nu})\|_{L^2(\Sigma)}<\frac{\delta_i}{5}$ for all $\nu\in\mathbbm N$. It follows from the definition of the set $X_i$ and the triangle inequality that
\begin{eqnarray*}
B^{\nu}\coloneqq\big\{A\in\A^{0,p}(P)\,\big|\,\|A-(A_i+\alpha_i(A^{\nu}))\|_{L^2(\Sigma)}<\frac{\delta_i}{5}\big\}\subseteq X_i.
\end{eqnarray*}
With $A^{\nu}\in X_i\subseteq\im\mathfrak m_i$, Proposition \ref{prop:condlocslicepert} applies and yields a gauge transformation $g^{\nu}\in\G^{1,p}(P)$ such that $(g^{\nu})^{\ast}A^{\nu}=A_i+\alpha_i(A^{\nu})$. Because $A^{\nu}\to A$ in $\A^{0,2}(P)$ as $\nu\to\infty$ it follows for sufficiently large $\nu\geq\nu_0$ that $(g^{\nu})^{\ast}A\in B^{\nu}\subseteq X_i$. Since $X_i$ is invariant under gauge transformations in $\G^{1,p}(P)$ this implies $A\in X_i$ and contradicts our assumption. Thus the claim follows. Because every $A\in\A^{0,p}(P)$ is either contained in $X_i$ or in its complement in $\A^{0,p}(P)$ we conclude that $\V_{\ell}$ is smooth on $\A^{0,p}(P)$.
\end{proof}

In the following we shall consider only those perturbations $\V_{\ell}$ which meet the assumptions of Propositions \ref{prop:perturbsmooth} and \ref{prop:L2estgrad}. These are precisely satisfied for triples $\ell=(i,j,k)\in\mathbbm N^3$ such that $k>\max\{\frac{10}{\delta_i},\frac{2}{\delta(A_i,p)}\}$ where $\delta_i$ denotes the constant of Proposition \ref{prop:perturbsmooth}, and $\delta(A_i,p)$ is as in Proposition \ref{prop:L2estgrad}.  For the remainder of this article we allow only for triples $(i,j,k)\in\mathbbm N^3$ such that $k$ satisfies this condition and renumber the subset of such triples by integers $\ell\in\mathbbm N$. Given $\ell\in\mathbbm N$, we fix a constant $C_{\ell}>0$ such that the following conditions are satisfied.

\begin{compactenum}[(i)]
\item
$\sup_{A\in\A(P)}|\V_{\ell}(A)|\leq C_{\ell}$,
\item
$\sup_{A\in\A(P)}\|\nabla\V_{\ell}(A)\|_{L^2(\Sigma)}\leq C_{\ell}$,
\item
$\|\nabla\V_{\ell}(A)\|_{C^0(\Sigma)}\leq C_{\ell}(1+\|F_A\|_{L^3(\Sigma)})$ for all $A\in\A(P)$.
\item
$\|H_A\V_{\ell}\beta\|_{L^p(\Sigma)}\leq C_{\ell}(1+\|F_A\|_{L^3(\Sigma)})\|\beta\|_{L^p(\Sigma)}$ for all $A\in\A(P)$, $\beta\in\Omega^1(\Sigma,\ad(P))$, and $1<p<\infty$. 
\end{compactenum}
Here the notation $\nabla\V_{\ell}\colon\A(P)\to\Omega^1(\Sigma,\ad(P))$ and $H_A\V_{\ell}\colon\Omega^1(\Sigma,\ad(P))\to\Omega^1(\Sigma,\ad(P))$ refers to the $L^2$ gradient and Hessian (at the point $A\in\A(P)$) of the map $\V_{\ell}$. We prove in Proposition \ref{prop:estimatespertVell} the existence of such a constant $C_{\ell}$, for every $\ell\in\mathbbm N$. The \emph{universal space of perturbations} is the normed linear space
\begin{eqnarray}\label{def:univspacepert}
Y\coloneqq\Big\{\V\coloneqq\sum_{\ell=1}^{\infty}\lambda_{\ell}\V_{\ell}\,\Big|\,\lambda_{\ell}\in\R\;\textrm{and}\;\|\V\|\coloneqq\sum_{\ell=1}^{\infty}C_{\ell}|\lambda_{\ell}|<\infty\Big\}.
\end{eqnarray}
It is a separable Banach space isomorphic to the space $\ell^1$ of summable real sequences. Further relevant properties of the perturbations $\V_{\ell}$ are discussed in \ref{sect:Properties of the perturbations}. To prove Theorem \ref{thm:mainresult} we need to consider the Yang--Mills gradient flow for connections of Yang--Mills energy below a fixed level set $a>0$. For a given regular value $a>0$ of $\YM$ and each critical manifold $\mathcal C\in\mathcal{CR}^a$ we fix a gauge-invariant closed $L^2$ neighborhood $U_{\mathcal C}$ of $\mathcal C$ such that $U_{\mathcal C_1}\cap U_{\mathcal C_2}=\emptyset$ for all $\mathcal C_1,\mathcal C_2\in\mathcal{CR}^a$ with $\mathcal C_1\neq\mathcal C_2$. From Proposition \ref{prop:finiteset} it follows that such a choice is possible. We then restrict the universal Banach space $Y$ of perturbations as follows.

\begin{defn}\label{def:regperturbation}
\upshape
Let $a>0$ be a regular value of $\YM$. A perturbation $\V=\sum_{\ell=1}^{\infty}\lambda_{\ell}\V_{\ell}\in Y$ is called $a$\emph{-admissible} if it satisfies for every $\ell\in\mathbbm N$ the condition
\begin{eqnarray*}
\supp\V_{\ell}\cap U_{\mathcal C}\neq\emptyset\quad\textrm{for some}\quad\mathcal C\in\mathcal{CR}^a\quad\Longrightarrow\quad\lambda_{\ell}=0.
\end{eqnarray*}
\end{defn}

It is straightforward to show that the space of $a$-admissible perturbation is a closed subspace of the Banach space $Y$. The following proposition shows that adding a small $a$-admissible perturbation to $\YM$ leaves the set of critical points below level $a$ unchanged.

\begin{prop}\label{prop:samecritpoints} 
For every regular value $a>0$ of $\YM$ and every $\eps>0$ there is a constant $\delta=\delta(a,\eps)>0$ with the following significance. Assume $\V$ is an $a$-admissible perturbation with $\|\V\|<\delta$. Then it holds that
\begin{multline*}
\crit(\YMV)\cap\{A\in\A^{1,p}(P)\mid \YM(A)<a\}\\
=\crit(\YM)\cap\{A\in\A^{1,p}(P)\mid \YM(A)<a\}.
\end{multline*}
\end{prop}

\begin{proof}
The inclusion $\crit(\YM)\subseteq\crit(\YMV)$ is clear because $\V$ is by Definition \ref{def:regperturbation} supported away from $\crit(\YM)$. It remains to show that the set $\crit(\YMV)\cap\{A\in\A^{1,p}(P)\mid\YM(A)<a\}$ only contains points that are also critical for $\YM$. Thus let $A$ be contained in this set. In addition we may assume that $A\in\supp\V$ and hence in particular $A\notin\cup_{\mathcal C\in\mathcal{CR}^a}U_{\mathcal C}$. Proposition \ref{lem:PalaisSmale} now shows the existence of a constant $\delta_1=\delta_1(a)>0$ such that $\|\nabla\YM(A)\|_{W^{-1,2}(\Sigma)}\geq\delta_1$. Choosing $\delta<\delta_1$ it follows for every $a$-admissible perturbation $\V$ with $\|\V\|<\delta$ that
\begin{multline*}
\|\nabla\YMV(A)\|_{W^{-1,2}(\Sigma)}\geq\|\nabla\YM(A)\|_{W^{-1,2}(\Sigma)}-\|\nabla\V(A)\|_{W^{-1,2}(\Sigma)}\\
\geq\|\nabla\YM(A)\|_{W^{-1,2}(\Sigma)}-\|\nabla\V(A)\|_{L^2(\Sigma)}>\delta_1-\delta>0.
\end{multline*}
The second inequality is by Proposition \ref{prop:l2w-12estimate}. The third one is a consequence of condition (ii) above from which it follows that
\begin{eqnarray*}
\|\nabla\V(A)\|_{L^2(\Sigma)}\leq\sum_{\ell=1}^{\infty}|\lambda_{\ell}|\cdot\|\nabla\V_{\ell}(A)\|_{L^2(\Sigma)}\leq\sum_{\ell=1}^{\infty}|\lambda_{\ell}|C_{\ell}=\|\V\|<\delta.
\end{eqnarray*}
Hence $A\notin\crit(\YMV)$, and this proves the remaining inclusion.
\end{proof}

\section{Yang--Mills gradient flow}\label{sec:YMgradientflow}

As discussed in the introduction it is convenient for analytical reasons to introduce in the perturbed Yang--Mills gradient flow equation $\partial_sA+d_A^{\ast}F_A+\nabla\V(A)=0$ an additional \emph{gauge fixing term} $-d_A\Psi$, $\Psi\in\Omega^0(\Sigma,\ad(P))$, such that solutions of the flow equation become invariant under \emph{time-dependent} gauge transformations. This requires to introduce some further notation. For an interval $I\subseteq\R$ we denote by $\hat P_I\coloneqq I\times P$ the trivial extension of the principle $G$-bundle $P$ to the base manifold $I\times\Sigma$, and set $\hat P\coloneqq\hat P_{\R}$. We let $\G(\hat P)$ denote the group of smooth gauge transformations of the principle $G$-bundle $\hat P$ and call this the group of time-dependent gauge-transformations (and similarly for $\G(\hat P_I)$). A pair $(A,\Psi)\in C^{\infty}(\R,\A(P)\times\Omega^0(\Sigma,\ad(P)))$ can naturally be identified with the connection $\AA=A+\Psi\,ds\in\A(\hat P)$. The action of the group $\G(\hat P)$ on $\A(\hat P)$ by gauge transformations is given by
\begin{eqnarray}\label{eq:extendedgaugetransf}
g^{\ast}\AA=g^{\ast}A+(g^{-1}\Psi g+g^{-1}\partial_sg)\wedge ds.
\end{eqnarray}

\begin{defn}\label{defnEYF}
\upshape
Let $\V\in Y$ be a perturbation. The $\G(\hat P)$-\emph{invariant, perturbed Yang--Mills gradient flow} is the nonlinear PDE
\begin{eqnarray}\label{EYF}
0=\partial_sA+d_A^{\ast}F_A-d_A\Psi+\nabla\mathcal V(A)
\end{eqnarray}
for connections $\AA=A+\Psi\,ds\in\A(\hat P)$.
\end{defn}

One easily checks that with $\AA\in\A(\hat P)$, also $g^{\ast}\AA$ is a solution of \eqref{EYF}, for every $g\in\G(\hat P)$. We show in Proposition \ref{prop:temporalgauge} below the existence of a gauge transformation $g\in\G(\hat P)$ which transforms every solution $\AA$ of \eqref{EYF} in \emph{temporal gauge} $g^{\ast}\AA=g^{\ast}A+0\,ds$. This gauge transformation is determined as solution of the ODE $\partial_sg=-\Psi g$ and is unique up to multiplication with a constant gauge transformation $h\in\G(P)$. In this way moduli spaces (modulo $\G(\hat P)$ equivalence) of connecting trajectories of \eqref{EYF} become naturally identified with those of the Yang--Mills gradient flow equation $\partial_sA+d_A^{\ast}F_A+\nabla\V(A)=0$ (modulo $\G(P)$ equivalence). In order to obtain an identification only up to $\G_0(P)$ equivalence we restrict $\G(\hat P)$ to gauge transformations $g(s)$ which converge to some based gauge transformation $g^-\in\G_0(P)$ as $s\to-\infty$, cf.~the following section.

\subsection{Banach manifolds}\label{Banachmanifolds}

We introduce the functional analytic setup used to construct moduli spaces of solutions to equation \eqref{defnEYF}. The Banach manifolds we shall work with are modeled on weighted Sobolev spaces in order to make the Fredholm theory work. We therefore fix a number $\delta>0$ and a smooth cut-off function $\beta$ such that $\beta(s)=-1$ if $s<0$ and $\beta(s)=1$ if $s>1$. We define the $\delta$-weighted $(k,p)$-Sobolev norm (for $1\leq p\leq\infty$ and an integer $k\geq0$) of a measurable function (respectively, a measurable section of a vector bundle) $u$ over $\R\times\Sigma$ to be the usual $(k,p)$-Sobolev norm of the function (or section) $e^{\delta\beta(s)s}u$.\\
\medskip\\
Recall the definition of parabolic Sobolev spaces at the end of Section \ref{sec:prelim}. We fix numbers $\delta>0$, $p>3$. Let $\A_{\delta}^{1,2;p}(P)$ denote the space of time-dependent connections on $P$ which are locally of class $W^{1,2;p}$ and for which there exist limiting connections $A^{\pm}\in\hat{\mathcal C}^{\pm}\subseteq\crit(\YM)$ and times $T^{\pm}\in\R$ such that the time-dependent $1$-forms $\alpha^{\pm}\coloneqq A-A^{\pm}$ satisfy 
\begin{align}\label{eq:regularityatends}
\nonumber\alpha^-\in W_{\delta}^{1,p}((-\infty,T^-],L^p(\Sigma,\ad(P)))\cap L_{\delta}^p((-\infty,T^-],W^{2,p}(\Sigma,\ad(P))),\\
\alpha^+\in W_{\delta}^{1,p}([T^+,\infty),L^p(\Sigma,\ad(P)))\cap L_{\delta}^p([T^+,\infty),W^{2,p}(\Sigma,\ad(P))).
\end{align}
Similarly, let $\G_{\delta}^{2,p}(\hat P)$ denote the group of gauge transformations of $\hat P$ which are locally of class $W^{2,p}$ and in addition satisfy the following two conditions. First, the time-dependent $\ad(P)$-valued $1$-form $g^{-1}dg$ satisfies 
\begin{eqnarray*}
g^{-1}dg\in L_{\delta}^p(\R,W^{2,p}(\Sigma,T^{\ast}\Sigma\otimes\ad(P)))
\end{eqnarray*}
(this condition being necessary to make sure that $g^{\ast}A\in L_{\delta}^p(\R,\A^{2,p}(P))$ for every $A$ with this property).
Second, there exist limiting gauge transformations $g^-\in\G_0^{2,p}(P)$, $g^+\in\G^{2,p}(P)$, numbers $T^{\pm}\in\R$, and bundle valued $0$-forms
\begin{align*}
\gamma^-\in W_{\delta}^{2,p}((-\infty,T^-]\times\Sigma,\ad(\hat P_{(-\infty,T^-]})),\\\gamma^+\in W_{\delta}^{2,p}([T^+,\infty)\times\Sigma,\ad(\hat P_{[T^+,\infty)}))
\end{align*}
with 
\begin{eqnarray*}
g(s)=g^-\exp(\gamma^-(s))\quad(s\leq T^-),\qquad g(s)=g^+\exp(\gamma^+(s))\quad(s\geq T^+).
\end{eqnarray*}
Here the notation $\hat P_{(-\infty,T^-]}$ refers to the trivial extension of the principle $G$-bundle $P$ to the base manifold $(-\infty,T^-]\times\Sigma$ (and analogously for the interval $[T^+,\infty)$).\\
\medskip\\
We now fix critical manifolds $\hat{\mathcal C}^{\pm}\subseteq\crit(\YM)$, with the set $\crit(\YM)$ being defined in \eqref{eq:defcritmanifold}, and denote $\mathcal C^{\pm}\coloneqq\hat{\mathcal C}^{\pm}/\G_0^{2,p}(P)$. (The notation $\hat{\mathcal C}$ deviates slightly from the one used in Section \ref{sec:YMfunctional}). We then denote by $\hat\B\coloneqq\hat\B(\hat{\mathcal C}^-,\hat{\mathcal C}^+,\delta,p)$ the Banach manifold of pairs
\begin{eqnarray*}
(A,\Psi)\in\A_{\delta}^{1,2;p}(P)\times W_{\delta}^{1,p}(\R\times\Sigma)
\end{eqnarray*}
such that $\lim_{s\to\pm\infty}A(s)=A^{\pm}$ holds in the sense of \eqref{eq:regularityatends} for some $A^{\pm}\in\hat{\mathcal C}^{\pm}$. We identify such pairs as before with connections $\AA=A+\Psi\,ds$ on $\hat P$. The action of the group $\G_{\delta}^{2,p}(\hat P)$ on $\hat\B$ by gauge transformations as in \eqref{eq:extendedgaugetransf} is smooth. It is free by our requirement that $\lim_{s\to-\infty}g(s)=g^-$ be a based gauge transformation. The resulting quotient space
\begin{eqnarray*}
\B\coloneqq\B(\mathcal C^-,\mathcal C^+,\delta,p)\coloneqq\frac{\hat\B(\hat{\mathcal C}^-,\hat{\mathcal C}^+,\delta,p)}{\G_{\delta}^{2,p}(\hat P)}
\end{eqnarray*}
therefore is again a smooth Banach manifold. We define the Banach space bundle $\E=\E(\mathcal C^-,\mathcal C^+,\delta,p)$ over $\B$ as follows. Let $\hat\E$ be the trivial Banach space bundle over $\hat\B$ with fibres  
\begin{eqnarray*}
\hat{\mathcal E}_{(A,\Psi)}\coloneqq L_{\delta}^p(\mathbbm R,L^p(\Sigma,T^{\ast}\Sigma\otimes\ad(P))).
\end{eqnarray*}
The action of $\G_{\delta}^{2,p}(\hat P)$ on $\hat\B$ lifts to a free action on $\hat\E$ (which is by conjugation on the fibres $\hat{\mathcal E}_{(A,\Psi)}$). Let $\E$ denote the respective quotient space. We finally define the section $\F\colon\B\to\E$ by 
\begin{eqnarray}\label{sectionF}
\F\colon[(A,\Psi)]\mapsto[\partial_sA+d_A^{\ast}F_A-d_A\Psi+\nabla\V(A)].
\end{eqnarray}

\subsection{Moduli spaces}\label{sect:modspaces}
We continue with the definition of moduli spaces of connecting trajectories of the perturbed Yang--Mills gradient flow. Let $a\geq0$ be a regular value of $\YM$. We fix an $a$-admissible perturbation $\V\in Y$ with $\|\V\|$ sufficiently small such that Proposition \ref{prop:samecritpoints} applies. Let a pair $(\hat{\mathcal C}^-,\hat{\mathcal C}^+)\in\mathcal{CR}^a\times\mathcal{CR}^a$ of critical manifolds be given and denote as before $\mathcal C^{\pm}\coloneqq\hat{\mathcal C}^{\pm}/\G_0^{2,p}(P)$. Let us define the space
\begin{multline*}
\hat{\mathcal M}(\hat{\mathcal C}^-,\hat{\mathcal C}^+)\coloneqq\Big\{(A,\Psi)\in\A_{\delta}^{1,2;p}(P)\times W_{\delta}^{1,p}(\R\times\Sigma)\;\Big|\;(A,\Psi)\,\textrm{satisfies}\,\eqref{EYF},\\
\lim_{s\to\pm\infty}A(s)=A^{\pm}\,\textrm{for some}\,A^{\pm}\in\hat{\mathcal C}^{\pm}\Big\}.
\end{multline*}
It is invariant under the action of the group $\G_{\delta}^{2,p}(\hat P)$ by gauge transformations. The limit $\lim_{s\to\pm\infty}A(s)=A^{\pm}$ is to be understood as in \eqref{eq:regularityatends}. We define the moduli space of gradient flow lines between $\mathcal C^-$ and $\mathcal C^+$ as the quotient
\begin{eqnarray}\label{modulispace}
\mathcal M(\mathcal C^-,\mathcal C^+)\coloneqq\frac{\hat{\mathcal M}(\hat{\mathcal C}^-,\hat{\mathcal C}^+)}{\G_{\delta}^{2,p}(\hat P)}.
\end{eqnarray}
This quotient equals the zero set of the section $\F$ in \eqref{sectionF}. It will be shown that for a generic perturbation $\V\in Y$ the vertical differential $d_x\mathcal F$ at every zero $x$ of $\F$ is a surjective Fredholm operator. It then follows from the implicit function theorem that $\mathcal M(\mathcal C^-,\mathcal C^+)$ is a finite-dimensional smooth manifold.

\subsection{Temporal gauges and regularity}

We establish two results concerning regularity properties and the existence of temporal gauges of solutions of the $\G(\hat P)$-invariant, perturbed Yang--Mills gradient flow equation \eqref{EYF}.

\begin{prop}\label{prop:regularity}
Let $a\geq0$ and $\V$ be an $a$-admissible perturbation with $\|\V\|<\delta$ sufficiently small such that Proposition \ref{prop:samecritpoints} applies. Let $\AA=A\in\A_{\loc}^{1,2;p}(P)$ be a solution of \eqref{EYF} in temporal gauge, meaning that $\Psi=0$, and assume that $\limsup_{s\to-\infty}\YM(A(s))\leq a$. Then there is a constant $T>0$ and a gauge transformation $g\in\G_{\loc}^{2,p}(\hat P)$ such that $g^-\coloneqq g|_{(-\infty,T]}$ and $g^+\coloneqq g|_{[T,\infty)}$ are independent of $s$, and such that the restriction $A_1\coloneqq g^{\ast}A|_{\R\setminus(-T,T)}$ is smooth. 
\end{prop}

\begin{proof}
Fix a constant $\eps>0$. Under the given assumptions we may apply Lemma \ref{lem:approxcritptaftergauge}. This result together with gauge invariance of the $L^2$ norm implies for every $\delta>0$ the existence of some $T>0$ such that $\|\partial_sA(s)\|_{L^2(\Sigma)}<\delta$ for all $|s|\geq T$. Choosing $\delta=\delta(\eps)$ sufficiently small and $T=T(\delta)$ sufficiently large we obtain from Proposition \ref{lem:PalaisSmale} the existence of a Yang--Mills connection $A_0(s)$ such that $\|A(s)-A_0(s)\|_{L^2(\Sigma)}<\eps$ for every $|s|\geq T$. By definition of $a$-admissible perturbations this implies that $\V(A(s))=0$ for all $|s|\geq T$. The rest of our argumentation follows the proof of the related result \cite[Theorem A.3]{Frauenfelder1}. For $n\in\mathbbm N$ denote $I_n\coloneqq[-T-n,-T]$. Then by Theorem \ref{thm:boundedness} and its proof there exists for each $n\in\mathbbm N$ a gauge transformation $g_n\in\G^{2,p}(\hat P|_{I_n})$ such that $g_n^{\ast}A|_{I_n}$ satisfies a local slice condition as in \eqref{locslicestep4} with respect to some smooth reference connection in $\A(\hat P|_{I_n})$. Then the same bootstrapping arguments as in the proof of Theorem \ref{thm:boundedness} show that $g_n^{\ast}A|_{I_n}$ is smooth. Bootstrapping here is not limited by lack of smoothness of $\V(A)$ because $\V(A)$ vanishes on $I_n$ as seen above. Finally,  we modify the gauge transformations $g_n$ in the following way. As shown in \cite[Theorem A.3]{Frauenfelder1} it exists for every $n\in\mathbbm N$ a smooth gauge transformation $h_n\in\G(\hat P|_{I_{n+1}})$ such that $\hat g_{n+1}\coloneqq h_n\circ g_{n+1}$ satisfies $\hat g_{n+1}|_{I_{n-1}}=g_n|_{I_{n-1}}$. Now replace $g_n$ by $\hat g_n$. Note that the gauge transformation $\hat g\in\G_{\loc}^{2,p}(\hat P|_{(-\infty,T]})$ given by $\hat g(s)\coloneqq\hat g_{n+1}(s)$ if $s\in I_n$ is well-defined. Also, $\hat g_{n+1}^{\ast}A|_{I_n}$ is smooth, and thus $\AA_1=A_1+\Psi_1\,ds\coloneqq\hat g^{\ast}A|_{(-\infty,-T]}\in\A(\hat P|_{(-\infty,-T]})$ is smooth. Solving the ODE $\partial_sh+\Psi_1h=0$ with $h(-T)=\mathbbm 1$ on $(-\infty,-T]$ yields a smooth gauge transformation $h$ such that the connection $\AA_2\coloneqq h^{\ast}\AA_1$ is smooth and in temporal gauge, i.e.~of the form $\AA_2=A_2+0\,ds$. Finally, by construction we have that $\AA_2=g^{\ast}A|_{(-\infty,-T]}$ for some gauge transformation $g^-\in\G_{\loc}^{2,p}(\hat P|_{(-\infty,-T]})$. As both $\AA_2$ and $A$ are in temporal gauge it follows that $g^-$ does not depend on $s$ and hence defines a gauge transformation $g^-\in\G^{2,p}(P)$. In the same way we can find a gauge transformation $g^+\in\G^{2,p}(P)$ such that $(g^+)^{\ast}A|_{[T,\infty)}$ is smooth. Then choose $g\in\G_{\loc}^{2,p}(\hat P)$ satisfying $g|_{(-\infty,-T]}=g^-$ and $g|_{[T,\infty)}=g^+$. With this choice of $g$, the claim follows.
\end{proof}

Subsequently we sometimes make use of the fact that any point $[(A,\Psi)]\in\mathcal M(\mathcal C^-,\mathcal C^+)$ in the moduli space has a representative in temporal gauge. This is the content of the following proposition.

\begin{prop}\label{prop:temporalgauge}
Let $(A,\Psi)\in\hat{\mathcal M}(\hat{\mathcal C}^-,\hat{\mathcal C}^+)$. Then there exists $A_1\in\A_{\delta}^{1,2;p}(P)$ such that $[(A,\Psi)]=[(A_1,0)]\in\mathcal M(\mathcal C^-,\mathcal C^+)$.
\end{prop}

\begin{proof}
Set $\AA\coloneqq A+\Psi\,ds$. For $n\in\mathbbm N$ we denote $I_n\coloneqq[-n,n]$. By Theorem \ref{thm:boundedness} there exists for each $n\in\mathbbm N$ a gauge transformation $h_n\in\G^{2,p}(\hat P|_{I_n})$ such that $\AA^n=A^n+\Psi^n\,ds\coloneqq h_n^{\ast}\AA|_{I_n}$ has at least the regularity properties $A^n\in\A^{2,p}(\hat P|_{I_n})$ and $\Psi^n\in W^{3,p}(I_n\times\Sigma)$. By standard arguments as carried out in \cite[Theorem A.3]{Frauenfelder1} we can replace each $h_n$ by a gauge transformation $\hat h_n\in\G^{2,p}(\hat P|_{I_n})$ satisfying the following. The gauge transformation $g$ of $\hat P$ given by $g(s)\coloneqq\hat h_{n+1}(s)$ if $s\in I_n$ is well-defined and contained in $\G_{\loc}^{2,p}(\hat P)$, and furthermore $\AA_0=A_0+\Psi_0\,ds\coloneqq g^{\ast}\AA$ satisfies $A_0\in\A_{\loc}^{2,p}(\hat P)$ and $\Psi_0\in W_{\loc}^{3,p}(\R\times\Sigma)$. Now define the gauge transformation $g_1$ of $\hat P$ as the unique solution of the ordinary differential equation
\begin{eqnarray*} 
\partial_sg_1+\Psi_0g_1=0,\qquad g_1(0)=\mathbbm 1.
\end{eqnarray*}
Then, as shown in \cite[Proposition 7.1]{SalWeh}, $g_1\in\G_{\loc}^{3,p}(\hat P)$ and $g_1^{\ast}\AA_0\eqqcolon\AA_1=A_1+0\,ds$. It furthermore follows that $A_1=g_1^{\ast}A_0\in\A_{\loc}^{2,p}(\hat P)$. It now follows from Proposition \ref{prop:regularity} that there exists a constant $T>0$ such that $A_1^-\coloneqq A_1|_{(-\infty,-T]}$ and $A_1^+\coloneqq A_1|_{[T,\infty)}$ differ from smooth connections $A_2^{\pm}$ on $\hat P|_{(-\infty,-T]}$, respectively $\hat P|_{[T,\infty)}$ by gauge transformations $g^{\pm}$ independent of the time variable $s$. Furthermore, $g^{\pm}\in\G^{3,p}(P)$ because of smoothness of $A_2^{\pm}$ and the regularity property $A_1\in\A_{\loc}^{2,p}(\hat P)$. Because of smoothness of $A_2^{\pm}$ we can apply Theorem \ref{YMExponentialdecay} which yields the existence of Yang--Mills connections $A^{\pm}\in\A(P)$ and exponential convergence $A_2^{\pm}(s)\to A^{\pm}$ in $C^{\ell}$ for all $\ell\geq0$, as $s\to\pm\infty$. Because the gauge transformations $g^{\pm}$ are constant in $s$ it follows that the connections $A_1^{\pm}=(g^{\pm})^{\ast}A_2^{\pm}$ converge exponentially to $(g^{\pm})^{\ast}A^{\pm}$ with respect to the $W^{2,p}$ norm on $(-\infty,-T]$, respectively $[T,\infty)$. This shows in particular that $(A_1,0)$ is contained in the Banach manifold $\hat B$ as introduced in Section \ref{Banachmanifolds}. So far we have shown that $A_1+0\,ds$ is gauge equivalent to $\AA$ via some gauge transformation $g_2\in\G_{\loc}^{2,p}(\hat P)$ and hence is again a solution of \eqref{EYF}. It remains to check that $g_2\in\G_{\delta}^{2,p}(\hat P)$, hence satisfying the exponential decay properties for $s\to\pm\infty$ as required in Section \ref{Banachmanifolds}. These follow form standard arguments as in \cite[Appendix E]{Ziltener} using the relations $g_2^{-1}dg_2=A_1-g_2^{-1}Ag_2$ and $g_2^{-1}\dot g_2=-g_2^{-1}\Psi g_2$ together with the exponential decay properties of $(A,\Psi)$ and $A_2$. Multiplying $g_2$ by a suitable gauge transformation $h\in\G^{2,p}(P)$ we can achieve that the limit $\lim_{s\to-\infty}hg_2(s)=hg_2^-$ is a based gauge transformation contained in $\G_0^{2,p}(P)$. This shows that $[(h^{\ast}A_1,0)]=[(A,\Psi)]$ and proves the claim.
\end{proof}

\begin{rem}\label{rem:identofmodspaces}
\upshape
We remark that in view of Proposition \ref{prop:temporalgauge} the moduli space $\mathcal M(\mathcal C^-,\mathcal C^+)$ can naturally be identified with the moduli space of connecting trajectories  (up to constant gauge transformations in $\G_0(P)$) between $\hat{\mathcal C^-}$ and $\hat{\mathcal C^+}$ of the Yang--Mills equation $\partial_sA+d_A^{\ast}F_A+\nabla\V(A)=0$, with analogous regularity and decay properties. We however do not make this statement precise since it will not be used in this work.
\end{rem}

\section{Exponential decay}\label{sec:Expdecay}

We fix a regular value $a\geq0$ of $\YM$ and an $a$-admissible perturbation $\V\in Y$ with $\|\V\|<\delta$ sufficiently small such that Proposition \ref{prop:samecritpoints} applies. The aim of this section is to establish exponential decay towards Yang--Mills connections for solutions of the gradient flow equation \eqref{EYF} below the energy level $a$. This result is a crucial step in proving compactness of moduli spaces of connecting gradient flow lines. Namely, the compactness Theorem \ref{thm:compactness} yields for every sequence of connecting trajectories below level $a$ (up to gauge transformations) a subsequence, converging to a (possibly broken) flow line. The exponential decay theorem below then shows that also its limit represents an element of the relevant moduli space of connecting trajectories. Thanks to Proposition \ref{prop:regularity} it suffices to consider solutions $A+\Psi\,ds$ of \eqref{EYF} such that $\Psi=0$ and $A$ is smooth outside some finite interval $(-T,T)$. In this situation there holds the following result.

\begin{thm}[Exponential decay]\label{YMExponentialdecay}
Let $a\geq0$ and $\V\in Y$ be as in the preceding paragraph. Then there is a constant $\lambda>0$ such that the following holds. Assume $A+\Psi\,ds$ is a solution of $\eqref{EYF}$ with $A\in\A_{\loc}^{1,2;p}(P)$ and $\Psi\in W^{1,p}(\R\times\Sigma)$ such that 
\begin{eqnarray}\label{eq:boundedenergysol}
\limsup_{s\to-\infty}\YMV(A(s))\leq a.
\end{eqnarray}
Assume furthermore the existence of a constant $T_1>0$ such that the restriction of $A$ to $\R\setminus(-T_1,T_1)$ is smooth. Then there exist smooth Yang--Mills connections $A^{\pm}$, a constant $T\geq T_1$, and for each integer $\ell\geq0$ a constant $c_{\ell}\geq0$ such that the estimate
\begin{eqnarray*} 
\|A-A^-\|_{C^{\ell}([-s-1,-s+1]\times\Sigma)}+\|A-A^+\|_{C^{\ell}([s-1,s+1]\times\Sigma)}\leq c_{\ell}e^{-\lambda(s-T)}
\end{eqnarray*}
holds for all $s\geq T+1$. (Here we abuse notation slightly in writing $A$ instead of $A|_{[-s-1,-s+1]}$ or $A|_{[s-1,s+1]}$).
\end{thm}

For the proof of Theorem \ref{YMExponentialdecay} we need to establish some further notation and auxiliary results. In the following we denote $I\coloneqq[-1,1]$. We fix a solution $A\in W_{\loc}^{1,2;p}(\R\times\Sigma)$ of \eqref{EYF} in temporal gauge satisfying \eqref{eq:boundedenergysol}. For each $\nu\in\mathbbm N$ and $s\in I$ we define the connection
\begin{eqnarray}\label{eq:mapsAnu}
A^{\nu}(s)\coloneqq A|_{[-\nu-1,-\nu+1]}(s-\nu)\in\A^{1,2;p}(P).
\end{eqnarray}
Note that these again satisfy equation \eqref{EYF}. We introduce the energy
\begin{eqnarray*} 
\mathcal E(\AA)\coloneqq\int_{-1}^1\|\dot A(s)-d_{A(s)}\Psi(s)\|_{L^2(\Sigma)}^2\,ds
\end{eqnarray*}
for connections $\AA=A+\Psi\,ds\in\A^{1,p}(\hat P_I)$. It is easily checked that $\mathcal E$ is invariant under gauge transformations in $\G^{2,p}(\hat P_I)$ and continuous as a map $\mathcal E\colon\A^{1,p}(\hat P_I)\to\R$. Considering $A^{\nu}$ as a connection on $\hat P_I$ it follows from our assumptions on $A$ that 
\begin{eqnarray}\label{eq:energyonintervalconvergtozero}
\lim_{\nu\to\infty}\E(A^{\nu})=0.
\end{eqnarray}

\begin{lem}\label{lem:approxcritptaftergauge}
Let constants $k\in\mathbbm N_0$ and $\kappa>0$ be given. Then there exists a smooth connection $\AA^{\infty}\in\A(\hat P_I)$, a positive integer $\nu_0=\nu_0(k,\kappa)$ and a sequence $g^{\nu}\in\G^{2,p}(\hat P_I)$ of gauge transformations such that for $\AA^{\nu}\coloneqq(g^{\nu})^{\ast}A^{\nu}$ and every $\nu\geq\nu_0$ the inequality
\begin{eqnarray}\label{eq:kappaW1pbound}
\|\AA^{\nu}-\AA^{\infty}\|_{W^{k,p}(I\times\Sigma)}<\kappa
\end{eqnarray}
is satisfied. Moreover, the connection $\AA^{\infty}$ can be chosen to be of the form
\begin{eqnarray}\label{eq:limitindepofs}
\AA^{\infty}(s)=A^{\infty}\qquad\textrm{for all}\quad s\in I,
\end{eqnarray}
where $A^{\infty}\in\A(P)$ is a Yang--Mills connection. Furthermore, the gauge transformations $g^{\nu}$ can be chosen to be independent of the time variable $s$.
\end{lem}

\begin{proof}
\setcounter{step}{0}
\begin{step}
There exists a number $\nu_0$ and for each $\nu\geq\nu_0$ a smooth Yang--Mills connection $\AA^{\infty,\nu}$ such that the statement holds true with $\AA^{\infty}$ replaced by $\AA^{\infty,\nu}$.
\end{step}
Let $\kappa>0$ be given and assume by contradiction that such a number $\nu_0$ does not exist. Then we can extract a subsequence of $(A^{\nu})$, which we again label by $\nu$, such that \eqref{eq:kappaW1pbound} is contradicted for every sequence $(g^{\nu})\subseteq\G^{2,p}(\hat P_I)$ of gauge transformations and every $\AA^{\infty}$ of the form \eqref{eq:limitindepofs} where $A^{\infty}$ is Yang--Mills. In the following we consider $A^{\nu}$ as before as a connection on $\hat P_I$. Note that by assumption it follows
\begin{eqnarray*}
\limsup_{\nu\to\infty}\YMV(A^{\nu}(0))\leq a<\infty.
\end{eqnarray*}
Hence Theorem \ref{thm:boundedness} applies. In combination with the compact embedding $W^{2,p}(I\times\Sigma)\hookrightarrow W^{1,p}(I\times\Sigma)$ this yields the existence of a sequence $(g^{\nu})\subseteq\G^{2,p}(\hat P_I)$ of gauge transformations and a connection $\AA_1=A_1+\Psi_1\,ds\in\A^{1,p}(\hat P_I)$ such that (after passing to a subsequence)
\begin{eqnarray}\label{eq:limapproxYM}
(g^{\nu})^{\ast}A^{\nu}\to\AA_1\qquad\textrm{as}\quad\nu\to\infty
\end{eqnarray}
in $\A^{1,p}(\hat P_I)$. Note that each $\AA_1^{\nu}\coloneqq(g^{\nu})^{\ast}A^{\nu}$ satisfies the $\G^{2,p}(\hat P_I)$-invariant gradient flow equation \eqref{EYF}. We claim that $A_1$ satisfies the Yang--Mills equation $d_{A_1}^{\ast}F_{A_1}=0$ on $I\times\Sigma$ (however, $A_1$ need not be stationary). Namely, from \eqref{eq:energyonintervalconvergtozero} and \eqref{eq:limapproxYM} it follows that $\mathcal E(\AA_1)=0$. Then \eqref{EYF} and \eqref{eq:limapproxYM} imply that $d_{A_1}^{\ast}F_{A_1}+\nabla\V(A_1)=0$ (considered as an equation in $L^2$) and therefore $A_1(s)$ is a critical point of $\YMV$ for almost every $s\in I$. From our choice of perturbation $\V$ and Proposition \ref{prop:samecritpoints} it follows that $d_{A_1(s)}^{\ast}F_{A_1(s)}=0$ and $\V(A(s))=0$, for almost every $s\in I$. The claim follows. As $\V|_U=0$ for a suitable $L^2(\Sigma)$ neighborhood of $A(s)$ (again by the choice of $\V$) we conclude from \eqref{eq:limapproxYM} that $\AA_1^{\nu}$ satisfies \eqref{EYF} with vanishing perturbation $\V=0$ for every $\nu\geq\nu_1$ sufficiently large. In this situation the bootstrap arguments in the proof of Theorem \ref{thm:boundedness} can straightforwardly be repeated and show that the convergence in \eqref{eq:limapproxYM} holds true in $\A^{k,p}(\hat P_I)$, for every  $k\in\mathbbm N$. Therefore the limit $\AA_1=A_1+\Psi_1\,ds$ is smooth. Let $h\in\G(\hat P)$ satisfy $h^{-1}\Psi_1h+h^{-1}\partial_sh=0$. Thus denoting $A_2^{\nu}+\Psi_2^{\nu}\,ds\coloneqq h^{\ast}(A_1^{\nu}+\Psi_1^{\nu}\,ds)$ and $A_2\coloneqq h^{\ast}A_1$ it follows that 
\begin{eqnarray}\label{eq:convergtozero}
\lim_{\nu\to\infty}(A_2^{\nu}+\Psi_2^{\nu}\,ds)=h^{\ast}A_1+(h^{-1}\Psi_1h+h^{-1}\partial_sh)\,ds=A_2
\end{eqnarray}
in $C^{\infty}(I\times\Sigma)$. Because $\E(h^{\ast}\AA_1)=\E(\AA_1)=0$ it follows that $\partial_sA_2=0$ and $A_2(s)\equiv A_2$ is a smooth Yang--Mills connection. Now we choose a sequence $(g_1^{\nu})\subseteq\G(\hat P_I)$ of smooth gauge transformation satisfying $(g_1^{\nu})^{-1}\Psi_2^{\nu}g_1^{\nu}+(g_1^{\nu})^{-1}\partial_sg_1^{\nu}=0$ for all $\nu$. From \eqref{eq:convergtozero} it follows that this choice is possible such that we have uniform convergence $g_1^{\nu}\to\mathbbm 1$ as $\nu\to\infty$. Hence the gauge transformed sequence of connections $\AA_3^{\nu}\coloneqq(g_1^{\nu})^{\ast}(A_2^{\nu}+\Psi_2^{\nu}\,ds)$ still satisfies \eqref{eq:convergtozero}. By the choice of gauge transformations $g_1^{\nu}$ it follows that each connection $\AA_3^{\nu}$ is in temporal gauge, i.e.~of the form $\AA_3^{\nu}=A_3^{\nu}$ and hence differs from $A^{\nu}$ by a gauge transformation which is constant in $s\in I$. Therefore, denoting $A^{\infty}\coloneqq A_2$ we have found a smooth Yang--Mills connection which contradicts our assumption. Hence there exists an integer $\nu_0$ such that for every $\nu\geq\nu_0$ inequality \eqref{eq:kappaW1pbound} is satisfied, for some gauge transformation $g^{\nu}$ and some smooth Yang--Mills connection $\AA^{\infty,\nu}$ of the form  \eqref{eq:limitindepofs}.
\begin{step}
The connections $\AA^{\infty,\nu}$ obtained in Step 1 can be chosen independently of $\nu\geq\nu_0$.
\end{step}
As shown in Step 1, $\AA^{\infty,\nu}(s)=A^{\infty,\nu}$ for all $s\in I$ where $A^{\infty,\nu}\in\A(P)$ is a Yang--Mills connection. To prove the claim it suffices to show that for all  sufficiently large $\nu\geq\nu_0$ the connections $A^{\infty,\nu}$ are contained in a single gauge orbit of $\G(P)$. Assume by contradiction that such a number $\nu_0$ does not exist. Then we can partition the set $N\coloneqq\{\nu\in\mathbbm N\mid\nu\geq\nu_0\}$ into two disjoint infinite subsets $N_1,N_2$ such that the two sequences $(A^{\lambda})_{\lambda\in N_j}$, $j=1,2$, do not have a common accumulation point up to gauge transformations. We choose an increasing sequence $(\mu_{\ell})\subseteq N$ of integers such that $\mu_{\ell}\in N_1$ and $\mu_{\ell}+1\in N_2$ holds for all $\ell\in\mathbbm N$. Then define for each $\ell\in\mathbbm N$ and $s\in[-2,2]$ the connection $A_1^{\ell}\in\A^{1,2;p}(P)$ by
\begin{eqnarray*}
A_1^{\ell}(s)=\begin{cases}
A^{\mu_{\ell}}(s+1)&\textrm{if}\quad-2\leq s\leq0,\\
A^{\mu_{\ell}+1}(s-1)&\textrm{if}\quad0\leq s\leq2.
\end{cases}
\end{eqnarray*}
Arguing as in Step 1 we can find a sequence $(g_{\ell})$ of gauge transformations such that (up to passing to a further subsequence) $g_{\ell}^{\ast}A_1^{\ell}$ converges uniformly to some $\AA_2^{\infty}\in\A(\hat P_{[-2,2]})$, where $\AA_2^{\infty}(s)\equiv A_2^{\infty}$ is a smooth Yang--Mills connection. This contradicts the assumption that a common accumulation point up to gauge transformations of the sequences $(A^{\lambda})_{\lambda\in N_j}$, $j=1,2$, does not exist. The claim now follows.
\end{proof}

The next step is to show that the connections $\AA^{\nu}$ obtained in Lemma \ref{lem:approxcritptaftergauge} converge exponentially in $L^2$ to $\AA^{\infty}$ as $\nu\to\infty$. As shown in the lemma, $\AA^{\infty}(s)\equiv A^{\infty}$ for some smooth Yang--Mills connection $A^{\infty}$. We note that $A^{\infty}\in\mathcal C_0$ for some smooth Banach submanifold $\mathcal C_0\subseteq\A^{2,p}(P)$ of Yang--Mills connections of energy at most $a$. A straightforward application of the implicit function theorem shows the existence of positive constants $\kappa_0=\kappa_0(A^{\infty})$ and $c_0=c_0(A^{\infty})$ such that for every $A\in\A^{2,p}(P)$ with $\|A-A^{\infty}\|_{W^{2,p}(\Sigma)}<\kappa_0$ there exists a connection $A_0\in\mathcal C_0$ with 
\begin{eqnarray}\label{eq:normalspace}
\|A-A_0\|_{W^{2,p}(\Sigma)}\leq c_0\|A-A^{\infty}\|_{W^{2,p}(\Sigma)}\qquad\textrm{and}\qquad A-A_0\in(T_{A_0}\mathcal C_0)^{\perp}.
\end{eqnarray}
Note that in particular $A-A_0$ is orthogonal to the gauge orbit through $A_0$ and hence satisfies $d_{A_0}^{\ast}(A-A_0)=0$. Subsequently we use standard elliptic estimates for the Hessian $H_{A_0}\colon W^{2,r}(\Sigma)\to L^r(\Sigma)$. This is a bounded operator for every $A_0\in\A^{2,p}(P)$ and $2\leq r<\infty$. Standard elliptic theory shows that its range is closed and every $\beta\in L^r(\Sigma)$ admits a unique $L^2(\Sigma)$-orthogonal decomposition
\begin{eqnarray}\label{eq:decgradientHessian1}
\beta	=\beta_0+\beta_1\in\im H_{A_0}\oplus\ker H_{A_0}.
\end{eqnarray}
Subsequently, we apply this with $r=p$ and $A\in\A^{2,p}(P)$ to $\beta=\beta(A)\coloneqq d_A^{\ast}F_A\in L^p(\Sigma)$. Assuming $\|A-A^{\infty}\|_{W^{2,p}(\Sigma)}<\kappa_0$, where $\kappa_0$ is the constant as above, we fix $A_0\in\mathcal C_0$ such that \eqref{eq:normalspace} holds. We also set $\alpha\coloneqq A-A_0$ and define 
\begin{eqnarray*}
R(\alpha)\coloneqq\frac{1}{2}d_{A_0}^{\ast}[\alpha\wedge\alpha]-[\ast\alpha\wedge\ast(d_{A_0}\alpha+\frac{1}{2}[\alpha\wedge\alpha])].
\end{eqnarray*}
Expanding the term $\beta=d_{A_0+\alpha}^{\ast}F_{A_0+\alpha}$ in $\alpha$ we obtain the identity
\begin{eqnarray}\label{eq:decompositiongradient}
\beta=H_{A_0}\alpha+R(\alpha).
\end{eqnarray}
Here we used that $A_0$ satisfies the Yang--Mills equation $d_{A_0}^{\ast}F_{A_0}=0$.

\begin{prop}\label{prop:compdriftterm}
Let $A^{\infty}\in\mathcal C_0$ and $\kappa_0>0$ be as in the previous paragraph. Then for every $\eps>0$ there exist positive constants $\kappa_1=\kappa_1(A^{\infty})\leq\kappa_0$ and $c_1=c_1(A^{\infty})$ with the following significance. For every $A\in\A^{2,p}(P)$ with $\|A-A^{\infty}\|_{W^{2,p}(\Sigma)}<\kappa_1$ the decomposition \eqref{eq:decgradientHessian1} of $\beta=d_A^{\ast}F_A$ satisfies the estimate
\begin{eqnarray*}
\|\beta_1\|_{L^2(\Sigma)}\leq\eps c_1\|d_A^{\ast}F_A\|_{L^2(\Sigma)}.
\end{eqnarray*}
\end{prop}

\begin{proof}
Let $A_0\in\mathcal C_0$ and $\alpha=A-A_0$ be as before. Note that (assuming $\kappa_1$ sufficiently small such that $\|[\alpha\wedge\alpha]\|_{L^2(\Sigma)}\leq\|\alpha\|_{L^2(\Sigma)}$, and using the Sobolev embedding $W^{2,p}(\Sigma)\hookrightarrow C^1(\Sigma)$) there exists a constant $c=c(\Sigma)>0$ with
\begin{eqnarray*}
\|R(\alpha)\|_{L^2(\Sigma)}\leq c\|\alpha\|_{C^1(\Sigma)}\|\alpha\|_{L^2(\Sigma)}\leq c\kappa_1\|\alpha\|_{L^2(\Sigma)}. 
\end{eqnarray*}
From \eqref{eq:decompositiongradient} and orthogonality of $\im H_{A_0}$ and $\ker H_{A_0}$ it follows that
\begin{eqnarray*}
\|\beta_1\|_{L^2(\Sigma)}^2=\langle\beta_1,H_{A_0}\alpha+R(\alpha)-\beta_0\rangle=\langle\beta_1,R(\alpha)\rangle\leq\|\beta_1\|_{L^2(\Sigma)}\|R(\alpha)\|_{L^2(\Sigma)},
\end{eqnarray*}
hence $\|\beta_1\|_{L^2(\Sigma)}\leq\|R(\alpha)\|_{L^2(\Sigma)}\leq c\kappa_1\|\alpha\|_{L^2(\Sigma)}$. Denoting by $\lambda=\lambda(A_0)>0$ the smallest (in absolute value) non-zero eigenvalue of $H_{A_0}$ it follows that $\|H_{A_0}\alpha\|_{L^2(\Sigma)}\geq\lambda\|\alpha\|_{L^2(\Sigma)}$. Here we used the second condition in \eqref{eq:normalspace} which by the Morse--Bott condition is equivalent to $\alpha\in(\ker H_{A_0})^{\perp}$. In fact, $\lambda$ can be chosen uniformly for all $A_0\in\mathcal C_0$ by compactness of $\mathcal C_0$ up to gauge transformations and the fact that the operator norm of $H_{A_0}\colon W^{2,2}(\Sigma)\to L^2(\Sigma)$, and hence $\lambda(A_0)$, depend continuously on $A_0\in\A^{2,p}(P)$. Then it follows (we drop the subscript $L^2(\Sigma)$ after $\|\cdot\|$),
\begin{multline*}
\frac{\|\beta_1\|}{\|\beta_0+\beta_1\|}\leq\frac{c\kappa_1\|\alpha\|}{\|H_{A_0}\alpha\| -\|R(\alpha)\|}\leq\frac{c\kappa_1\|\alpha\|}{\lambda\|\alpha\|-\|R(\alpha)\|}=\frac{c\kappa_1}{\lambda}+\frac{c\kappa_1\|R(\alpha)\|}{\lambda^2\|\alpha\|-\lambda\|R(\alpha)\|}\\
\leq\frac{c\kappa_1}{\lambda}+\frac{c\kappa_1\|R(\alpha)\|}{\lambda^2c^{-1}\kappa_1^{-1}\|R(\alpha)\|-\lambda\|R(\alpha)\|}=\frac{c\kappa_1}{\lambda}+\frac{c\kappa_1}{\lambda^2c^{-1}\kappa_1^{-1}-\lambda}.
\end{multline*}
Now choose $\kappa_1>0$ still smaller if necessary, such that $\frac{c\kappa_1}{\lambda}+\frac{c\kappa_1}{\lambda^2c^{-1}\kappa_1^{-1}-\lambda}<\eps$ is satisfied. The claim then follows. 
\end{proof}

\begin{lem}[$L^2$ exponential decay of the gradient]\label{exponentialdecayofthegradient}
Let $a\geq0$ be the energy level fixed before. There exists a constant $\lambda=\lambda(a)>0$ such that the following estimate is satisfied. Assume $A\in\A_{\loc}^{1,2;p}(P)$ is a solution of \eqref{EYF} as specified in the paragraph preceding Lemma \ref{lem:approxcritptaftergauge}. Then there exists $T_0\geq0$ such that for all $T\geq T_0$ and $s\leq-T$
\begin{eqnarray}\label{aprioriexpdecay}
\|\partial_sA(s)\|_{L^2(\Sigma)}\leq e^{\lambda(s+T)}\|\partial_sA(-T)\|_{L^2(\Sigma)}.
\end{eqnarray}
An analogous exponential decay estimate holds for all $s\geq T$.
\end{lem}

\begin{proof}
To prove exponential decay, we shall apply Lemma \ref{apriori0} to the map $s\mapsto\frac{1}{2}\|\dot A(s)\|_{L^2(\Sigma)}^2$. In a first step, a decay estimate of the same kind will be obtained for the maps $(g^{\nu})^{\ast}A^{\nu}$ with $A^{\nu}$ as defined in \eqref{eq:mapsAnu} and time-independent gauge transformations $g^{\nu}$ for which the conclusion of Lemma \ref{lem:approxcritptaftergauge} holds true. For a sufficiently small constant $\kappa>0$, which we shall fix in course of the proof, we let $\nu_0\in\mathbbm N$ be such that Lemma \ref{lem:approxcritptaftergauge} applies with $k=2$ and this constant $\kappa$. Hence for all $\nu\geq\nu_0$ there holds the inequality 
\begin{eqnarray}\label{eq:prelimestkappa}
\|(g^{\nu})^{\ast}A^{\nu}-A^{\infty}\|_{W^{2,p}(I\times\Sigma)}<\kappa.
\end{eqnarray}
Choosing $\nu_0$ still larger if necessary we conclude from the proof of Lemma \ref{lem:approxcritptaftergauge} that $A^{\nu}$ is a solution of the unperturbed equation $\partial_sA^{\nu}+d_{A^{\nu}}^{\ast}F_{A^{\nu}}=0$, for every $\nu\geq\nu_0$. For any fixed $\nu\geq\nu_0$ we temporarily denote $A\coloneqq(g^{\nu})^{\ast}A^{\nu}$. Differentiating the unperturbed gradient flow equation yields the identity 
\begin{eqnarray}\label{eq:timedifferentialA1}
\ddot A=-\frac{d}{ds}d_A^{\ast}F_A=-d_A^{\ast}d_A\dot A+\ast[\dot A\wedge\ast F_A]=-H_A\dot A.
\end{eqnarray}
We furthermore calculate
\begin{multline}\label{eq:timedifferentialA2}
\frac{d}{ds}(H_A\dot A)=H_A\ddot A+d_A^{\ast}[\dot A\wedge\dot A]-\ast[\dot A\wedge\ast d_A\dot A]+\ast[\ast d_A\dot A\wedge\dot A]\\
=H_A\ddot A+d_A^{\ast}[\dot A\wedge\dot A]-2\ast[\dot A\wedge\ast d_A\dot A].
\end{multline}
Combining \eqref{eq:timedifferentialA1} and \eqref{eq:timedifferentialA2} we obtain
\begin{eqnarray}\label{asymptconvex1}
\nonumber\frac{d^2}{ds^2}\frac{1}{2}\|\dot A\|_{L^2(\Sigma)}^2&=&\frac{d}{ds}\langle\ddot A,\dot A\rangle\\
\nonumber&=&\|\ddot A\|_{L^2(\Sigma)}^2-\langle\partial_s(H_A\dot A),\dot A\rangle\\
\nonumber&=&2\|H_A\dot A\|_{L^2(\Sigma)}^2-\langle\dot A,d_A^{\ast}[\dot A\wedge\dot A]\rangle+2\langle\dot A,\ast[\dot A\wedge\ast d_A\dot A]\rangle\\
&=&2\|H_A\dot A\|_{L^2(\Sigma)}^2-3\langle d_A\dot A,[\dot A\wedge\dot A]\rangle.
\end{eqnarray}
We use the orthogonal decomposition of $\dot A$ as in \eqref{eq:decompositiongradient} into $\dot A=\beta_0+\beta_1$ where $\beta_0\in\im H_{A_0}$ and $\beta_1\in\ker H_{A_0}$ for suitable $A_0\in\mathcal C_0$. Set $\alpha\coloneqq A-A_0$. The triangle inequality implies that $\|d_A\dot A\|_{L^2(\Sigma)}^2\leq2\|d_A\beta_0\|_{L^2(\Sigma)}^2+2\|d_A\beta_1\|_{L^2(\Sigma)}^2$. The last two terms are now estimated separately as follows. Using that $d_{A_0}^{\ast}d_{A_0}\beta_1+\ast[\ast F_{A_0}\wedge\beta_1]=0$ we obtain
\begin{eqnarray}\label{eq:estdAbeta1}
\nonumber\|d_A\beta_1\|_{L^2(\Sigma)}^2&\leq&2\|d_{A_0}\beta_1\|_{L^2(\Sigma)}^2+2\|[\alpha\wedge\beta_1]\|_{L^2(\Sigma)}^2\\
\nonumber&=&-2\langle\beta_1,\ast[\ast F_{A_0}\wedge\beta_1]\rangle+2\|[\alpha\wedge\beta_1]\|_{L^2(\Sigma)}^2\\
&\leq&c(A_0)\|\beta_1\|_{L^2(\Sigma)}^2+c\|\alpha\|_{L^{\infty}(\Sigma)}^2\|\beta_1\|_{L^2(\Sigma)}^2.
\end{eqnarray}
Furthermore, by standard elliptic theory and the fact that $\beta_0\in(\ker H_{A_0})^{\perp}$ it follows that
\begin{eqnarray}\label{ee:estbeta0deriv}
\|\beta_0\|_{L^2(\Sigma)}^2+\|d_A\beta_0\|_{L^2(\Sigma)}^2\leq c(A_0)\|H_{A_0}\beta_0\|_{L^2(\Sigma)}^2.
\end{eqnarray}
Using that $\|\alpha\|_{W^{2,p}(\Sigma)}<c_0\kappa\eqqcolon\kappa_2$ by \eqref{eq:prelimestkappa} and the first condition in \eqref{eq:normalspace}, and choosing $\kappa$ still smaller if necessary, we obtain for the last term in \eqref{asymptconvex1} the estimate
\begin{eqnarray}\label{eq:estremainderterm}
\nonumber\lefteqn{3\big|\langle d_A\dot A,[\dot A\wedge\dot A]\rangle\big|}\\
\nonumber&\leq&c\|\dot A\|_{L^{\infty}(\Sigma)}\big(\|\beta_0\|_{L^2(\Sigma)}^2+\|d_A\beta_0\|_{L^2(\Sigma)}^2+\|\beta_1\|_{L^2(\Sigma)}^2+\|d_A\beta_1\|_{L^2(\Sigma)}^2\big)\\
&\leq&c(A_0)\|\dot A\|_{L^{\infty}(\Sigma)}\big(\|H_A\beta_0\|_{L^2(\Sigma)}^2+\|\beta_1\|_{L^2(\Sigma)}^2\big).
\end{eqnarray}
Next we estimate $d_{A_0}^{\ast}\beta_1$ using the identity
\begin{eqnarray*}
0=d_A^{\ast}\beta=d_{A_0}^{\ast}(\beta_0+\beta_1)-\ast[\alpha\wedge\ast(\beta_0+\beta_1)]=d_{A_0}^{\ast}\beta_1-\ast[\alpha\wedge\ast(\beta_0+\beta_1)].
\end{eqnarray*}
Together with \eqref{eq:estdAbeta1}, \eqref{ee:estbeta0deriv}, and $\|\alpha\|_{L^{\infty}(\Sigma)}<\kappa_2$ this implies the estimate
\begin{eqnarray}\label{ee:estbeta1deriv}
\nonumber\|\beta_1\|_{W^{1,2}(\Sigma)}^2&\leq&c(A_0)(\|d_{A_0}\beta_1\|_{L^2(\Sigma)}^2+\|d_{A_0}^{\ast}\beta_1\|_{L^2(\Sigma)}^2+\|\beta_1\|_{L^2(\Sigma)}^2)\\
\nonumber&\leq&c(A_0)(\kappa_2^2\|\beta_0\|_{L^2(\Sigma)}^2+\|\beta_1\|_{L^2(\Sigma)}^2)\\
&\leq&c(A_0)(\kappa_2^2\|H_{A_0}\beta_0\|_{L^2(\Sigma)}^2+\|\beta_1\|_{L^2(\Sigma)}^2),
\end{eqnarray}
the first inequality being satisfied by Lemma \ref{lem:slicewiseLpsection}. Denoting by $\rho$ the square of the operator norm of $H_A-H_{A_0}\colon W^{1,2}(\Sigma)\to L^2(\Sigma)$ we can now estimate the term $\|H_A\dot A\|_{L^2(\Sigma)}^2$ appearing in \eqref{asymptconvex1} as
\begin{eqnarray}\label{eq:HdotAterm}
\nonumber2\|H_A\dot A\|_{L^2(\Sigma)}^2&\geq&\|H_A\beta_0\|_{L^2(\Sigma)}^2-2\|H_A\beta_1\|_{L^2(\Sigma)}^2\\
\nonumber&\geq&\|H_A\beta_0\|_{L^2(\Sigma)}^2-\|H_{A_0}\beta_1\|_{L^2(\Sigma)}^2-2\|(H_A-H_{A_0})\beta_1\|_{L^2(\Sigma)}^2\\
\nonumber&\geq&\|H_A\beta_0\|_{L^2(\Sigma)}^2-2\rho\|\beta_1\|_{W^{1,2}(\Sigma)}^2\\\
&\geq&(1-c(A_0)\rho\kappa_2^2)\|H_A\beta_0\|_{L^2(\Sigma)}^2-c(A_0)\rho\|\beta_1\|_{L^2(\Sigma)}^2.
\end{eqnarray}
In the second but last line we used $\beta_1\in\ker H_{A_0}$. The last estimate follows from \eqref{ee:estbeta1deriv}. Furthermore, because $\beta_0\in(\ker H_{A_0})^{\perp}$ there exists a constant $\lambda=\lambda(A_0)>0$ such that, after choosing $\|\alpha\|_{W^{2,p}(\Sigma)}<\kappa_2=c_0\kappa$ still smaller if necessary, we can estimate $\|H_A\beta_0\|_{L^2(\Sigma)}\geq\lambda\|\beta_0\|_{L^2(\Sigma)}$. Proposition \ref{prop:compdriftterm} implies the estimate $\|\beta_1\|_{L^2(\Sigma)}\leq\eps\|\dot A\|_{L^2(\Sigma)}$ for some $\eps>0$ which we will fix below. Combining \eqref{asymptconvex1}, \eqref{eq:estremainderterm}, and \eqref{eq:HdotAterm}, and denoting $\delta_1\coloneqq1-c(A_0)\rho\kappa_2^2-c(A_0)\|\dot A\|_{L^{\infty}(\Sigma)}$ (which is positive for $\kappa_2$ small enough) and $\delta_2\coloneqq c(A_0)(\rho+\|\dot A\|_{L^{\infty}(\Sigma)})$ this yields
\begin{eqnarray}\label{eq:estddsdotA}
\nonumber\frac{d^2}{ds^2}\frac{1}{2}\|\dot A\|_{L^2(\Sigma)}^2&\geq&\delta_1\|H_A\beta_0\|_{L^2(\Sigma)}^2-\delta_2\|\beta_1\|_{L^2(\Sigma)}^2\\
\nonumber&\geq&\delta_1\lambda^2\|\beta_0\|_{L^2(\Sigma)}^2-\delta_2\eps^2\|\dot A\|_{L^2(\Sigma)}^2\\
\nonumber&=&\delta_1\lambda^2\big(\|\dot A\|_{L^2(\Sigma)}^2-\|\beta_1\|_{L^2(\Sigma)}^2\big)-\delta_2\eps^2\|\dot A\|_{L^2(\Sigma)}^2\\
&\geq&\big(\delta_1\lambda^2-\delta_2\eps^2-\delta_1\lambda^2\eps^2\big)\|\dot A\|_{L^2(\Sigma)}^2.
\end{eqnarray}
Let $\eps>0$ be sufficiently small such that the factor $\delta_1\lambda^2-\delta_2\eps^2-\delta_1\lambda^2\eps^2$ is positive. By Proposition \ref{prop:compdriftterm} such a choice is possible after fixing the constant $\kappa>0$ still smaller if necessary. Note that the expression $\delta_1\lambda^2-\delta_2\eps^2-\delta_1\lambda^2\eps^2$ involves only constants which do not depend on $A$ but only on the critical manifold $\mathcal C_0$. We can replace this factor by a uniform one, depending only on the chosen energy level $a\geq0$, because the set of critical manifolds below level $a$ is compact up to gauge transformations. Thus we have shown that estimate \eqref{eq:estddsdotA} holds for all $A=(g^{\nu})^{\ast}A^{\nu}$ where $\nu\geq\nu_0$ and $\nu_0$ is the positive integer determined by $\kappa$ such that the conclusion of Lemma \ref{lem:approxcritptaftergauge} holds true. Note that estimate \eqref{eq:estddsdotA} is invariant under the time-independent gauge transformations $g^{\nu}$. Set $T_0\coloneqq\nu_0$. Then estimate \eqref{eq:estddsdotA} is satisfied for the original map $s\mapsto\frac{1}{2}\|\dot A(s)\|_{L^2(\Sigma)}^2$ where $s\leq-T_0$. Therefore Lemma \ref{apriori0} applies and shows \eqref{aprioriexpdecay}. The proof of an exponential decay result of the same type in the case $s\geq T$ follows analogously.
\end{proof} 

Before we turn to the proof of Theorem \ref{YMExponentialdecay} we need to establish a further estimate for the terms $\partial_s^k\dot A$, where $k\geq0$ is some integer. For this it seems necessary to deal with norms with respect to varying reference connections. Most of our argumentation in this passage follows \cite[Section 5]{SalWeh}. Let $\AA\in\A(\hat P_I)$ be a smooth reference connection. Throughout it will be chosen to be in temporal gauge, i.e.~of the form $\AA(s)=A(s)\in\A(P)$ for all $s\in I$. We include the reference connection in our notation and denote the respective Sobolev spaces of sections by $W_{\AA}^{k,r}(I\times\Sigma)$, where $1<r<\infty$ and $k\geq0$ is some integer, cf.~Section \ref{sec:prelim} for details. The resulting norms are invariant under gauge transformations $g\in\G(\hat P_I)$ in the sense that
\begin{eqnarray}\label{eq:gaugeinvariantnorm}
\|g^{-1}\alpha g\|_{W_{g^{\ast}\AA}^{k,r}(I\times\Sigma)}=\|\alpha\|_{W_{\AA}^{k,r}(I\times\Sigma)}
\end{eqnarray}
is satisfied for every $\alpha\in\Omega^k(\Sigma,\ad(\hat P_I))$. Similarly, we use the notation $C_{\AA}^{\ell}(I\times\Sigma)$ for the $C^{\ell}$ space defined with respect to the reference connection $\AA$. For the remainder of this section we denote
\begin{eqnarray}\label{eq:operatorlinparabolic}
(\mathcal D_{\AA}\alpha)(s)\coloneqq\frac{d}{ds}\alpha(s)+\Delta_{A(s)}\alpha(s)+\ast[\ast F_{A(s)}\wedge\alpha(s)]
\end{eqnarray}
for $\alpha\in\Omega^1(\Sigma,\ad(\hat P_I))$ and $s\in I$.

\begin{prop}\label{prop:gaugeinvarparabest}
Let $A\in\A_{\loc}^{1,2;p}(P)$ be as assumed in Theorem \ref{YMExponentialdecay} and $(A^{\nu})$ be the sequence defined in \eqref{eq:mapsAnu}. Let $\AA^{\infty}\in\A(\hat P_I)$ be the smooth limit connection as specified in Lemma \ref{lem:approxcritptaftergauge}. Fix a properly contained subinterval $I_1$ of $I$. Then for every integer $k\geq0$ there exist constants $c(\AA^{\infty},k)>0$ and $\nu_0\in\mathbbm N$ such that, for every $\nu\geq\nu_0$ and every time-dependent $\ad(P)$-valued $1$-form $\beta$,
\begin{eqnarray*}
\|\beta\|_{W_{A^{\nu}}^{k+1,2k+2;2}(I_1\times\Sigma)}\leq c(\AA^{\infty},k)\big(\|\mathcal D_{A^{\nu}}\beta\|_{W_{A^{\nu}}^{k,2k;2}(I\times\Sigma)}+\|\beta\|_{L^2(I\times\Sigma)}\big).
\end{eqnarray*}
(Here $W_A^{k,2k;2}(I\times\Sigma)$ denotes the Sobolev space of time-dependent $\ad(P)$-valued $1$-forms $\beta$ such that $\nabla_A^i\partial_s^j\beta$ is in $L^2(I\times\Sigma)$ for all integers $i,j\geq0$ with $2j+i\leq2k$.)
\end{prop}

\begin{proof}
There is a constant $c(\AA^{\infty},k)>0$ such that for every $\alpha\in\Omega^1(\Sigma,\ad(\hat P_I))$ it holds 
\begin{eqnarray*}
\|\alpha\|_{W_{\AA^{\infty}}^{k+1,2k+2;2}(I_1\times\Sigma)}\leq c(\AA^{\infty},k)\big(\|\mathcal D_{\AA^{\infty}}\alpha\|_{W_{\AA^{\infty}}^{k,2k;2}(I\times\Sigma)}+\|\alpha\|_{L^2(I\times\Sigma)}\big).
\end{eqnarray*}
This follows from standard estimate for linear parabolic operators with time-independent coefficients (cf.~e.g.~\cite[Theorem 2.2]{Web}) which here applies as by \eqref{eq:limitindepofs} $\AA^{\infty}$ does not depend on $s$. As can easily be checked, this estimate is stable under small variations of $\AA^{\infty}$ with respect to the $C^{\ell}(I\times\Sigma)$ norm, where $\ell=2k+1$. Therefore, after fixing $\kappa>0$ sufficiently small and applying Lemma \ref{lem:approxcritptaftergauge}, it persists to hold true with the new constant $2c(\AA^{\infty},k)$ and with $\AA^{\infty}$ replaced by $(g^{\nu})^{\ast}A^{\nu}$, for every $\nu\geq\nu_0(\kappa)$ sufficiently large. Here $(g^{\nu})\subseteq\G(\hat P_I)$ is a sequence of gauge transformations for which estimate \eqref{eq:kappaW1pbound} of Lemma \ref{lem:approxcritptaftergauge} is satisfied. Hence for every $\beta\in\Omega^1(\Sigma,\ad(\hat P_I))$ it follows, applying the above estimate with $\alpha\coloneqq(g^{\nu})^{-1}\beta g^{\nu}$, that
\begin{multline*}
\|(g^{\nu})^{-1}\beta g^{\nu}\|_{W_{(g^{\nu})^{\ast}A^{\nu}}^{k+1,2k+2;2}(I_1\times\Sigma)}\\
\leq2c(\AA^{\infty},k)\big(\|\mathcal D_{(g^{\nu})^{\ast}A^{\nu}}((g^{\nu})^{-1}\beta g^{\nu})\|_{W_{(g^{\nu})^{\ast}A^{\nu}}^{k,2k;2}(I\times\Sigma)}+\|(g^{\nu})^{-1}\beta g^{\nu}\|_{L^2(I\times\Sigma)}\big).
\end{multline*}
Now the claim follows, using \eqref{eq:gaugeinvariantnorm} together with the identity
\begin{eqnarray*}
\mathcal D_{(g^{\nu})^{\ast}A^{\nu}}((g^{\nu})^{-1}\beta g^{\nu})=(g^{\nu})^{-1}(\mathcal D_{A^{\nu}}\beta)g^{\nu},
\end{eqnarray*}
and gauge invariance of the $L^2$ norm.
\end{proof}

\begin{proof}{\bf(Theorem \ref{YMExponentialdecay})}
Let $A\in\A_{\loc}^{1,2;p}(P)$ be as assumed in Theorem \ref{YMExponentialdecay} and $(A^{\nu})$ be the sequence defined in \eqref{eq:mapsAnu}. Let $I_1\subseteq I$ be as in Proposition \ref{prop:gaugeinvarparabest}. Differentiating (with respect to $s$) the gradient flow equation $\partial_sA+d_A^{\ast}F_A=0$ satisfied by $A$ for $s\leq-T_0$ it follows that $\partial_s\dot A+d_A^{\ast}d_A\dot A+\ast[\ast F_A\wedge\dot A]=0$ and $d_A^{\ast}\dot A=0$. Hence $\mathcal D_{A^{\nu}}\dot A^{\nu}=0$ by definition \eqref{eq:operatorlinparabolic} of $\mathcal D_{A^{\nu}}$. Applying Proposition \ref{prop:gaugeinvarparabest} with $\beta=\dot A^{\nu}$ we obtain for every $k\geq0$ and $\nu\geq\nu_0=T$ (with $T\geq T_1$ sufficiently large, where $T_1$ is as stated in Theorem \ref{YMExponentialdecay}) the estimate 
\begin{eqnarray}\label{eq:estSobolevgrad}
\|\dot A^{\nu}\|_{W_{A^{\nu}}^{k,2k;2}(I_1\times\Sigma)}\leq c(\AA^{\infty},k)\|\dot A^{\nu}\|_{L^2(I\times\Sigma)}.
\end{eqnarray}
Fix an integer $\ell\geq0$ and let $k\geq0$ be sufficiently large such that there is a continuous embedding $W_{A^{\nu}}^{k,2k;2}(I_1\times\Sigma)\hookrightarrow C_{A^{\nu}}^{\ell}(I_1\times\Sigma)$. Integrating \eqref{aprioriexpdecay} over the interval $[-\nu-1,-\nu+1]$ we can estimate the right-hand side of \eqref{eq:estSobolevgrad} further and obtain
\begin{eqnarray}\label{eq:estSobolevgrad1}
\|\dot A^{\nu}\|_{C_{A^{\nu}}^{\ell}(I_1\times\Sigma)}\leq c(\AA^{\infty},\ell)Ce^{-\lambda(\nu-T)}
\end{eqnarray}
for constants $c(\AA^{\infty},\ell)>0$ and $C\coloneqq\frac{e^{\lambda}-e^{-\lambda}}{\lambda}$. We claim the existence of a smooth Yang--Mills connection $A^-$ and a constant $c_{\ell}$ for each integer $\ell\geq0$ such that
\begin{eqnarray}\label{eq:claimexpdecayCell}
\|A-A^-\|_{C^{\ell}([-\nu-1,-\nu+1]\times\Sigma)}\leq c_{\ell}e^{-\lambda(\nu-T)}
\end{eqnarray}
is satisfied for every $\nu\geq T$. (Note that we claim this to be satisfied with a norm taken with respect for some smooth reference connection independent of $A$). We prove the claim by induction on $\ell$. For $\ell=0$ it follows from \eqref{eq:estSobolevgrad1} that $\|\dot A\|_{C^0([-\nu-1,-\nu+1]\times\Sigma)}\leq c(\AA^{\infty},0)Ce^{-\lambda(\nu-T)}$ for every $\nu\geq T$, and hence the integral
\begin{eqnarray}\label{eq:deflimitconnection}
A^-\coloneqq A(-T)-\int_{-\infty}^{-T}\partial_sA(s)\,ds
\end{eqnarray}
converges in $C^0(\Sigma,T^{\ast}\Sigma\otimes\ad(P))$ and defines a continuous connection $A^-$. Furthermore, \eqref{eq:claimexpdecayCell} holds for $\ell=0$. Fix an integer $\ell\geq0$ and suppose estimate \eqref{eq:claimexpdecayCell} holds true with this $\ell$ and some $C^{\ell}$-regular connection $A^-$. From \eqref{eq:claimexpdecayCell} it follows that $A|_{(-\infty,T]}$ is bounded in $C^{\ell}((-\infty,T])$ (this space being defined with respect to some fixed smooth reference connection which we do not specify in our notation) and hence there exists a uniform constant  $\delta_{\ell}>0$ independent of $\nu$ such that
\begin{eqnarray*}
\|\alpha\|_{C^{\ell+1}(I\times\Sigma)}\leq\delta_{\ell}\|\alpha\|_{C_{A^{\nu}}^{\ell+1}(I\times\Sigma)}
\end{eqnarray*}
holds for all $\alpha\in\Omega^1(I\times\Sigma,T^{\ast}(I\times\Sigma)\otimes\ad(\hat P_I))$. The crucial point here is that the definition of $C_{A^{\nu}}^{\ell+1}$ involves derivatives of $A^{\nu}$ only up to order $\ell$. Hence with $\alpha\coloneqq\dot A$ and \eqref{eq:estSobolevgrad1} it follows that
\begin{eqnarray*}
\|\dot A\|_{C^{\ell+1}(I_1\times\Sigma)}\leq\delta_{\ell}c(\AA^{\infty},\ell+1)Ce^{-\lambda(\nu-T)}.
\end{eqnarray*}
This shows that the integral in \eqref{eq:deflimitconnection} converges in $C^{\ell+1}(\Sigma,T^{\ast}\Sigma\otimes\ad(P))$ to some connection $A^-$ is of class $C^{\ell+1}$. Furthermore, \eqref{eq:claimexpdecayCell} holds with $\ell$ replaced by $\ell+1$ and with constant $\delta_{\ell}c(\AA^{\infty},\ell+1)C$. In particular it follows that the connection $A^-$ is smooth. It is a Yang--Mills connection because \eqref{eq:claimexpdecayCell} implies that
\begin{eqnarray*}
d_{A^-}^{\ast}F_{A^-}=\lim_{s\to-\infty}d_{A(s)}^{\ast}F_{A(s)}=-\lim_{s\to-\infty}\partial_sA(s)=0.
\end{eqnarray*}
This proves the claim and hence the first part of the exponential decay estimate  asserted in Theorem \ref{YMExponentialdecay}. Forward exponential decay follows in a completely analogous manner. This completes the proof of Theorem \ref{YMExponentialdecay}.
\end{proof}

\section{Fredholm theory}\label{sec:Fredholm theorem}
\subsection{Yang--Mills Hessian} 
For $A\in\A(P)$ we let $\H_A$ denote the {\emph{augmented Yang--Mills Hessian}} defined by
\begin{multline*} 
\H_A\coloneqq\left(\begin{array}{cc}d_A^{\ast}d_A+\ast[\ast F_A\wedge\,\cdot\,]&-d_A\\-d_A^{\ast}&0\end{array}\right)\colon\\
\Omega^1(\Sigma,\ad(P))\oplus\Omega^0(\Sigma,\ad(P))\to\Omega^1(\Sigma,\ad(P))\oplus\Omega^0(\Sigma,\ad(P)).
\end{multline*}
In order to find a domain for $\H_A$ which makes the subsequent Fredholm theory work, we fix an irreducible smooth reference connection $A_0\in\A(P)$ and decompose the space $\Omega^1(\Sigma,\ad(P))$ of smooth $\ad(P)$-valued $1$-forms as the $L^2(\Sigma)$ orthogonal sum
\begin{multline*} 
\Omega^1(\Sigma,\ad(P))=\ker\big(d_A^{\ast}\colon\Omega^1(\Sigma,\ad(P))\to\Omega^0(\Sigma,\ad(P))\big)\\
\oplus\;\im\big(d_A\colon\Omega^0(\Sigma,\ad(P))\to\Omega^1(\Sigma,\ad(P))\big).
\end{multline*}
Let $W_0^{2,p}$ and $W_1^{1,p}$ denote the completions of the first component, respectively of the second component with respect to the $(k,p)$-Sobolev norm ($k=1,2$). We set $\mathcal W^p(\Sigma)\coloneqq W_0^{2,p}\oplus W_1^{1,p}$ and endow this space with the sum norm. It is independent of the choice of reference connection $A_0$ by the following proposition.

\begin{prop}
Let $A_0,A_1\in\A(P)$ be two irreducible smooth connections. Then the spaces $\mathcal W^p(\Sigma)$ and $\hat{\mathcal W}^p(\Sigma)$ defined above with respect to the reference connections $A_0$, respectively $A_1$ are isomorphic as Banach spaces.
\end{prop}

\begin{proof}
We have $\hat{\mathcal W}^p(\Sigma)=\hat W_0^{2,p}\oplus\hat W_1^{1,p}$ where the last two spaces are defined in analogy to $W_0^{2,p}$ and $W_1^{1,p}$ but with respect to the reference connection $A_1$. We define a bijective linear map $\Lambda\colon\mathcal W^p(\Sigma)\to\hat{\mathcal W}^p(\Sigma)$ as follows. Denote $\beta\coloneqq A_1-A_0$. For $\alpha_0\in W_0^{2,p}$ we set
\begin{eqnarray}\label{eq:mapLambda1}
\Lambda\colon\alpha_0\mapsto(\alpha_0-d_{A_1}\ph,d_{A_1}\ph)\in\hat W_0^{2,p}\oplus\hat W_1^{1,p},
\end{eqnarray}
where $\ph$ is the unique solution of
\begin{eqnarray}\label{eq:Laplaceph1}
\Delta_{A_1}\ph=-\ast[\beta\wedge\ast\alpha_0].
\end{eqnarray}
For $\alpha_1=d_{A_0}\ph_0\in W_1^{1,p}$ we define
\begin{eqnarray}\label{eq:mapLambda2}
\Lambda\colon\alpha_1\mapsto(d_{A_1}\gamma-[\beta\wedge\ph_0],d_{A_1}\ph)\in\hat W_0^{2,p}\oplus\hat W_1^{1,p},
\end{eqnarray}
where $\ph=\ph_0-\gamma$ and $\gamma$ is the unique solution of 
\begin{eqnarray}\label{eq:Laplaceph2}
\Delta_{A_1}\gamma=d_{A_1}^{\ast}[\beta\wedge\ph_0].
\end{eqnarray} 
The linear map $\Lambda$ is uniquely defined through \eqref{eq:Laplaceph1} and \eqref{eq:Laplaceph2}. Bijectivity of $\Lambda$ is easily checked. Boundedness of $\Lambda$ follows from ellipticity and bijectivity of the Laplacians $\Delta_{A_0}$ and $\Delta_{A_0}$. Namely, for $\ph$ in \eqref{eq:Laplaceph1} we have the bound
\begin{eqnarray*}
\|\ph\|_{W^{3,p}(\Sigma)}\leq c(A_0,A_1)\|\alpha_0\|_{W^{1,p}(\Sigma)}.
\end{eqnarray*}
Thus boundedness of the map in \eqref{eq:mapLambda1} follows. Similarly, the equation $d_{A_0}^{\ast}\alpha_1=\Delta_{A_0}\ph_0$ implies the estimate $\|\ph_0\|_{W^{2,p}(\Sigma)}\leq c(A_0)\|\alpha_1\|_{W^{1,p}(\Sigma)}$, and \eqref{eq:Laplaceph2} gives $\|\gamma\|_{W^{3,p}(\Sigma)}\leq c(A_0,A_1)\|\ph_0\|_{W^{2,p}(\Sigma)}$. Putting these together yields boundedness of the map in \eqref{eq:mapLambda2}. Hence the map $\Lambda\colon\mathcal W^p(\Sigma)\to\hat{\mathcal W}^p(\Sigma)$ is a Banach space isomorphism.
\end{proof}

We let $p>1$ and consider the augmented Yang--Mills Hessian as an operator
\begin{eqnarray*}
\H_A\colon\mathcal W^p(\Sigma)\oplus W^{1,p}(\Sigma,\ad(P))\to L^p(\Sigma,T^{\ast}\Sigma\otimes\ad(P))\oplus L^p(\Sigma,\ad(P)).
\end{eqnarray*}
In the case $p=2$ this is a densely defined symmetric operator on the Hilbert space $L^2(\Sigma,T^{\ast}\Sigma\otimes\ad(P))\oplus L^2(\Sigma,T^{\ast}\Sigma)$ with domain
\begin{eqnarray}\label{def:domainHA}
\dom\H_A\coloneqq\mathcal W^2(\Sigma)\oplus W^{1,2}(\Sigma,\ad(P)).
\end{eqnarray}
We show in Proposition \ref{YMselfad} below that it is self-adjoint. For the further discussion of the operator $\H_A$ it will be convenient to also decompose the first component $\beta$ of $\H_A(\alpha,\psi)^T$ as $\beta=\beta_0+\beta_1$, where $d_{A_0}^{\ast}\beta_0=0$ and $\beta_1=d_{A_0}(\omega-\psi)$ for some $0$-form $\omega$. A short calculation shows that for $\alpha=\alpha_0+\alpha_1=\alpha_0+d_{A_0}\varphi$ (with $d_{A_0}^{\ast}\alpha_0=0$) this $0$-form $\omega$ is determined as the unique solution of the equation
\begin{multline}\label{eqsplitting}
\Delta_{A_0}\omega=-[\ast F_A\wedge\ast d_A\alpha]+\ast[\theta\wedge\ast d_A^{\ast}d_A\alpha]+\ast d_{A_0}[\ast F_A\wedge\alpha]-d_{A_0}^{\ast}[\theta\wedge\psi].
\end{multline}
Here we denote $\theta\coloneqq A-A_0$. We furthermore define
\begin{align*}
K_A\alpha_0\coloneqq d_{A_0}^{\ast}[\theta\wedge\alpha_0]-\ast[\theta\wedge\ast d_A\alpha_0]+\ast[\ast F_A\wedge\alpha_0],\\
L_A\ph\coloneqq-d_A^{\ast}d_A[\theta\wedge\ph]+[d_A^{\ast}F_A\wedge\varphi]-\ast[\ast F_A\wedge[\theta\wedge\ph]].
\end{align*}

By direct calculation we obtain that with respect to the above decomposition of $\beta$ into $\beta=\beta_0+d_{A_0}(\omega-\psi)$ the augmented Hessian takes the form
\begin{eqnarray}\label{eqsplitting1}
\H_A\left(\begin{array}{c}\alpha_0\\\alpha_1\\\psi\end{array}\right)=\left(\begin{array}{c}\Delta_A\alpha_0+K_A\alpha_0+L_A\ph-[\theta\wedge\psi]-d_{A_0}\omega\\-d_{A_0}\psi+d_{A_0}\omega\\-d_{A_0}^{\ast}\alpha_1+\ast[\theta\wedge\ast\alpha]\end{array}\right),
\end{eqnarray}
where $\theta=A-A_0$, $\alpha_1=d_{A_0}\varphi$, and $\omega$ is the solution of \eqref{eqsplitting}. 

\begin{prop}\label{YMselfad}
For every $A\in\A(P)$, the operator $\H_A$ with domain $\dom\H_A$ as defined in \eqref{def:domainHA} is self-adjoint. It satisfies for all $(\alpha,\psi)\in\dom\,\H_A$ and $p>1$ the elliptic estimate
\begin{eqnarray}\label{YMHellipticestimate}
\|\alpha\|_{\mathcal W^p(\Sigma)}+\|\psi\|_{W^{1,p}(\Sigma)}\leq c\big(\|\H_A(\alpha,\psi)\|_{L^p(\Sigma)}+\|(\alpha,\psi)\|_{L^p(\Sigma)}\big)
\end{eqnarray}
with constant $c=c(A,p)$. If $A$ is a Yang--Mills connection, then the number of negative eigenvalues (counted with multiplicities) of $\H_A$ equals the Morse index of the Yang--Mills Hessian $H_A\mathcal{YM}$ as given by \eqref{eq:YMHess}.
\end{prop}

\begin{proof}
We show estimate \eqref{YMHellipticestimate}. All norms are with respect to the domain $\Sigma$, so we drop this in our notation. Then the equation $\alpha_1=d_{A_0}\ph$ and bijectivity of the operator $\Delta_{A_0}$ on $0$-forms imply the elliptic estimate
\begin{eqnarray}\label{eq:ellestvarphi}
\|\ph\|_{W^{1,p}}\leq c(A_0,p)\|d_{A_0}^{\ast}\alpha_1\|_{W^{-1,p}}\leq c(A_0,p)\|\alpha_1\|_{L^p}.
\end{eqnarray}
Similarly, we can estimate the solution $\omega$ of equation \eqref{eqsplitting} as
\begin{eqnarray}\label{eq:ellestvarphi1}
\|\omega\|_{L^p}\leq c(A,A_0,p)\big(\|\alpha\|_{L^p}+\|\ph\|_{L^p}\big)\leq c(A,A_0,p)\|\alpha\|_{L^p},
\end{eqnarray}
where in the second step we used \eqref{eq:ellestvarphi}. Let $(\beta_0,\beta_1,\gamma)\coloneqq\H_A(\alpha_0,\alpha_1,\psi)$. Ellipticity of the operator $\Delta_A$ applied to the first equation in \eqref{eqsplitting1} yields the estimate
\begin{eqnarray*}
\|\alpha_0\|_{W^{1,p}}&\leq&c(A,A_0,p)\big(\|\alpha_0\|_{L^p}+\|\beta_0\|_{W^{-1,p}}+\|K_A\alpha_0\|_{W^{-1,p}}\\
&&+\|L_A\ph\|_{W^{-1,p}}+\|\psi\|_{W^{-1,p}}+\|d_{A_0}\omega\|_{W^{-1,p}}\big)\\
&\leq&c(A,A_0,p)\big(\|\alpha_0\|_{L^p}+\|\beta_0\|_{L^p}+\|\ph\|_{W^{1,p}}+\|\psi\|_{L^p}+\|\omega\|_{L^p}\big)\\
&\leq&c(A,A_0,p)\big(\|\alpha\|_{L^p}+\|\beta_0\|_{L^p}+\|\psi\|_{L^p}\big).
\end{eqnarray*}
The last estimate follows from \eqref{eq:ellestvarphi} and \eqref{eq:ellestvarphi1}. Bootstrapping this estimate we obtain the bound
\begin{eqnarray*}
\|\alpha_0\|_{W^{2,p}}\leq c(A,A_0,p)\big(\|\alpha\|_{L^p}+\|\beta_0\|_{L^p}+\|\psi\|_{L^p}\big).
\end{eqnarray*}
This requires to bound $\|\omega\|_{W^{1,p}}$ for which the estimate 
\begin{eqnarray}\label{eq:ellestvarphi2}
\|\omega\|_{W^{1,p}}\leq c(A,A_0,p)\big(\|\alpha\|_{L^p}+\|\alpha_0\|_{W^{1,p}}+\|\ph\|_{W^{1,p}}\big).
\end{eqnarray}
holds. It is obtained as before from equation \eqref{eqsplitting}. The asserted estimate for $\alpha_1=d_{A_0}\ph$ follows similarly by elliptic regularity and the third equation in \eqref{eqsplitting1}. From the second equation in \eqref{eqsplitting1} we obtain again by elliptic regularity of $\Delta_{A_0}=d_{A_0}^{\ast}d_{A_0}$ that
\begin{eqnarray*}
\|\psi-\omega\|_{W^{1,p}}\leq c(A_0,p)\|\beta_1\|_{L^p}.
\end{eqnarray*}
Now the required estimate for $\psi$ follows, using \eqref{eq:ellestvarphi2} to bound $\|\omega\|_{W^{1,p}}$. Combining the estimates for $\alpha_0$, $\alpha_1$, and $\psi$, we obtain \eqref{YMHellipticestimate}. It implies self-adjointness in the case $p=2$. To prove the assertion on the index, let $A$ be a Yang--Mills connection and $(\alpha,\psi)^T$ be an eigenvector of $\H_A$ with corresponding eigenvalue $\lambda<0$. Let $\alpha=\alpha_0+\alpha_1$ be as above the Hodge decomposition of $\alpha$ with $d_{A}^{\ast}\alpha_0=0$ and $\alpha_1=d_{A}\ph$. Then the eigenvalue equation for $\H_A$ together with the first line of \eqref{eqsplitting1} yields $H_A\alpha_0=\lambda\alpha_0$. This uses that  $d_A\omega=0$ and $d_A^{\ast}F_A=0$ as $A$ is assumed to be Yang--Mills. Hence $\lambda$ is a negative eigenvalue of $H_A$. Conversely, if the eigenvalue equation $H_A\alpha=\lambda\alpha$ is satisfied for some $\lambda<0$ and $\alpha=\alpha_0+\alpha_1$, then necessarily $\alpha_1=0$ because $\alpha_1\in\ker H_A$. Thus $H_A\alpha_0=\lambda\alpha_0$, and this equation implies that $\lambda$  is also an eigenvalue of $\H_A$ with eigenvector $(\alpha_0,0)^T$.
\end{proof}

\subsection{Linearized operator}\label{subsect:linearizedoperator}
\setcounter{footnote}{0}
We next discuss the linearization of the perturbed $\G(\hat P)$-invariant Yang--Mills gradient flow equation \eqref{EYF}. Let $a>0$ be a regular value of $\YM$, and $p>1$ and $\delta>0$ be as in Section \ref{Banachmanifolds}. Throughout we fix an $a$-admissible perturbation $\V\in Y$ with $\|\V\|$ sufficiently small such that Proposition \ref{prop:samecritpoints} applies. We furthermore fix a pair $(\hat{\mathcal C}^-,\hat{\mathcal C}^+)\in\mathcal{CR}^a\times\mathcal{CR}^a$ of critical manifolds and denote $\mathcal C^{\pm}\coloneqq\hat{\mathcal C}^{\pm}/\G_0^{2,p}(P)$. Let $(A,\Psi)\in\hat{\mathcal M}(\hat{\mathcal C}^-,\hat{\mathcal C}^+)$ be a connecting Yang--Mills gradient flow line as defined in Section \ref{sect:modspaces}. Since all statements in this and the next section are formulated in a gauge-invariant way we may assume that $(A,\Psi)$ has the regularity properties as asserted in Proposition \ref{prop:regularity}, and in particular that $\Psi$ vanishes outside some finite interval $(-T,T)$. (Note that we may use Proposition 
\ref{prop:temporalgauge} to satisfy the assumptions made in Proposition \ref{prop:regularity}). We now define the Banach spaces
\begin{align*}
\Z^{\delta,p}\coloneqq\big(W_{\delta}^{1,p}(\R,L^p(\Sigma,T^{\ast}\Sigma\otimes\ad(P)))\cap L_{\delta}^p(\R,\mathcal W^p(\Sigma))\big)
\oplus W_{\delta}^{1,p}(\R\times\Sigma,\ad(P)),\\
\mathcal L^{\delta,p}\coloneqq L_{\delta}^p(\mathbbm R\times\Sigma,T^{\ast}\Sigma\otimes\ad(P))\oplus L_{\delta}^p(\mathbbm R\times\Sigma,\ad(P)).
\end{align*}
We let $H_A\V$ denote the Hessian of the map $\V$ as discussed in Proposition \ref{formuladifferential}. For $(\alpha,\psi)\in\Z^{\delta,p}$ we formally set $H_A\V(\alpha,\psi)^T\coloneqq(H_A\V\alpha,0)$. Furthermore, for $\Psi$ as before we define $M_{\Psi}(\alpha,\psi)^T\coloneqq([\Psi\wedge\alpha],-[\Psi\wedge\psi])$.

\begin{defn}\label{def:hordifferential}\upshape
The horizontal differential  at $(A,\Psi)$ of the section $\F$ as in \eqref{sectionF} is the linear operator
\begin{eqnarray}\label{eq:hordifferential}
\mathcal D_{(A,\Psi)}=\frac{d}{ds}+\mathcal H_{A}+M_{\Psi}+H_A\V\colon\Z^{\delta,p}\to\mathcal L^{\delta,p},
\end{eqnarray}
and the equation $\mathcal D_{(A,\Psi)}(\alpha,\psi)^T=0$ is called \emph{linearized perturbed Yang--Mills gradient flow equation}.
\end{defn}

We show in Section \ref{sect:Fredholmthm} that $\mathcal D_{(A,\Psi)}$ is a Fredholm operator and determine its index.

\begin{rem}\label{remBanachtangent}
\upshape
\begin{compactenum}[(i)]
\item
From the definition of $\F$ as a section $\F\colon\B\to\E$ (cf.~\eqref{sectionF}) it follows that its linearization $d\F(A,\Psi)$ acts on (gauge equivalence classes) of pairs $(\alpha,\psi)$ where $\alpha(s)$ converges exponentially to some $\alpha^{\pm}\in T_{A^{\pm}}\hat{\mathcal C}^{\pm}$ as $s\to\pm\infty$ for a given pair $\hat{\mathcal C}^{\pm}$ of critical manifolds. Hence with $\mathcal C^{\pm}\coloneqq\hat{\mathcal C}^{\pm}/\G_0^{2,p}(P)$, $d\F(A,\Psi)$ is properly considered as an operator on the space $\Z^{\delta,p}\oplus\R^{\dim\mathcal C^-}\oplus\R^{\dim\mathcal C^-}$. This is in contrast to the definition of $\mathcal D_{(A,\Psi)}$ in \ref{def:hordifferential}. However, it is easy to see that $d\F(A,\Psi)$ is Fredholm if and only if this property holds for $\mathcal D_{(A,\Psi)}$, and that the Fredholm indices are related via the formula
\begin{eqnarray*}
\ind d\F(A,\Psi)=\ind\mathcal D_{(A,\Psi)}+\dim\mathcal C^-+\dim\mathcal C^+.
\end{eqnarray*}
To see this, we view $d\F(A,\Psi)$ as a compact perturbation of the operator $\mathcal D_{(A,\Psi)}$, the latter being extended trivially to $\Z^{\delta,p}\oplus\R^{\dim\mathcal C^-}\oplus\R^{\dim\mathcal C^-}$.
\item
It suffices to discuss the Fredholm theory for the operator $\frac{d}{ds}+\H_A$, disregarding the terms $M_{\Psi}$ and $H_A\V$ in \eqref{eq:hordifferential}. This does not change the Fredholm property because both terms contribute only a compact perturbation. This follows from the fact that the support of both $\Psi$ and $\V$ is contained in some compact subset $[-T,T]\times\Sigma$ together with the Rellich compactness theorem and Proposition \ref{prop:Hessianbounded}.
\end{compactenum}
\end{rem}

From now on we write $\mathcal D_{(A,\Psi)}$ instead of $\mathcal D_A$, this notation being justified by the previous remark. Because the Hessians $\H_{A^{\pm}}$ will in general (i.e.~if $\dim\mathcal C^{\pm}\neq0$) have non-trivial zero eigenspaces, we cannot apply directly standard theorems on the spectral flow to prove the Fredholm theorem stated below. As an intermediate step we therefore use the Banach space isomorphisms 
\begin{eqnarray*}
\nu_1\colon\Z^{\delta,p}\to\Z^{0,p}\eqqcolon\Z^p\qquad\textrm{and}\qquad\nu_2\colon\mathcal L^{\delta,p}\to\mathcal L^{0,p}\eqqcolon\mathcal L^p
\end{eqnarray*}
given by multiplication with the weight function $e^{\delta\beta(s)s}$, where $\beta$ denotes the cut-off function introduced at the beginning of Section \ref{Banachmanifolds}. Then the assertion of Theorem \ref{thm:Fredholmtheorem} is equivalent to the analogous one for the operator
\begin{eqnarray*}
\mathcal D_A^{\delta}\coloneqq\nu_2\circ\mathcal D_A\circ\nu_1^{-1}\colon\quad\Z^p\to\mathcal L^p,
\end{eqnarray*}
which we shall prove instead. Note that the operator $\mathcal D_A^{\delta}$ takes the form 
\begin{eqnarray*}
\mathcal D_A^{\delta}=\frac{d}{ds}+\mathcal H_A-(\beta+\beta's)\delta,
\end{eqnarray*}
and hence, by our choice of $\beta$, the operator family $s\mapsto\mathcal H_A-(\beta+\beta's)\delta$ converges to the operators $\mathcal H_{A(s)}\mp\delta$ as $s\to\pm\infty$. These limit operators are invertible for a suitable choice of $\delta>0$ (which will be fixed in the next section). It follows that the spectral flow of this operator family is given by the right-hand side of \eqref{eq:indexD} below. In the following we denote $\H_A^{\delta}\coloneqq\H_A-(\beta+\beta's)\delta$.

\subsection{Fredholm theorem}\label{sect:Fredholmthm}
Let $\mathcal C^{\pm}=\hat{\mathcal C}^{\pm}/\G_0^{2,p}(P)$ be the pair of critical manifolds fixed in the previous section. We define the constant $\delta_0(\mathcal C^-,\mathcal C^+)$ to be the infimum of the set 
\begin{eqnarray*}
\{|\lambda|\in\R\mid\lambda\neq0\;\textrm{and}\;\lambda\;\textrm{is eigenvalue of}\;\mathcal H_A\;\textrm{for some}\;A\in \hat{\mathcal C}^-\cup\hat{\mathcal C}^+\}.
\end{eqnarray*}
We remark that $\delta_0(\mathcal C^-,\mathcal C^+)$ is positive as follows from compactness and non-degeneracy of the critical manifolds $\mathcal C^{\pm}$. In the following we assume that $0<\delta<\delta_0(\mathcal C^-,\mathcal C^+)$. For a critical point $A\in\hat{\mathcal C}^-\cup\hat{\mathcal C}^+$ we let $\ind A$ denote its Morse index, i.e.~the number of negative eigenvalues of $\mathcal H_{A^{\pm}}$, counted with multiplicities.

\begin{thm}[Fredholm theorem]\label{thm:Fredholmtheorem}
Let $(A,\Psi)\in\hat{\mathcal M}(\hat{\mathcal C}^-,\hat{\mathcal C}^+)$ be such that the asymptotic conditions 
\begin{eqnarray*} 
\lim_{s\to\pm\infty}A(s)=A^{\pm} 
\end{eqnarray*}
are satisfied for Yang--Mills connections $A^{\pm}\in\hat{\mathcal C}^{\pm}$ (the limits being understood in the sense of \eqref{eq:regularityatends}). Then the operator $\mathcal D_{(A,\Psi)}=\frac{d}{ds}+\mathcal H_{A}+M_{\Psi}+H_A\V\colon\Z^{\delta,p}\to\mathcal L^{\delta,p}$ as in Definition \ref{def:hordifferential} is a Fredholm operator of index
\begin{eqnarray}\label{eq:indexD}
\ind\mathcal D_{(A,\Psi)}=\ind A^--\ind A^+-\dim\mathcal C^+.
\end{eqnarray}  
\end{thm}

We first show Theorem \ref{thm:Fredholmtheorem} in the case $p=2$, where it follows from well-known results on the spectral flow for families of self-adjoint operators in Hilbert space, cf.~\cite{RobSal}. The case of general Sobolev exponents $p>1$ will afterwards be reduced to the Hilbert space case. The here relevant Hilbert space is 
\begin{eqnarray*}
H\coloneqq L^2(\Sigma,T^{\ast}\Sigma\otimes\ad(P))\oplus L^2(\Sigma,T^{\ast}\Sigma)
\end{eqnarray*}
on which for each $s\in\R$ the operator $\H_{A(s)}^{\delta}$ is self-adjoint with domain $W\coloneqq\dom\H_{A(s)}^{\delta}=\mathcal W^2(\Sigma)\oplus W^{1,2}(\Sigma,\ad(P))$. 

\begin{proof}{\bf{(Theorem \ref{thm:Fredholmtheorem} in the case $p=2$)}}
The result follows from \cite[Theorem A]{RobSal}. To apply this result we need to check that the following properties (i-v) are satisfied. (i) The inclusion $W\hookrightarrow H$ of Hilbert spaces is compact with dense image. This holds true by definition of the space $W$ and the Rellich--Kontrachov compactness theorem. (ii) The operator $\H_{A(s)}^{\delta}\colon H\to H$ is unbounded and self-adjoint with dense domain $W$. This is satisfied by Proposition \ref{YMselfad}. (iii) The norm of $W$ is equivalent to the graph norm of $\H_{A}^{\delta}(s)$ for every $s\in\R$. This holds by the elliptic estimate \eqref{YMHellipticestimate}. (iv) The map $\R\to\mathcal L(W,H)\colon s\mapsto\H_{A(s)}^{\delta}$ is continuously differentiable with respect to the weak operator topology. For this we need to verify that for every $\xi=(\alpha,\psi)\in W$ and $\eta=(\eta_1,\eta_2)\in H$ the map $s\mapsto\langle\H_{A(s)}^{\delta}\xi,\eta\rangle$ is of class $C^1(\R,\R)$. This amounts to check this property for the two maps 
\begin{align*}
g_1\colon s\mapsto\langle d_{A(s)}^{\ast}d_{A(s)}\alpha+\ast[\ast F_{A(s)}\wedge\alpha]-(\beta+\beta's)\delta\alpha-d_{A(s)}\psi,\eta_1\rangle,\\
g_2\colon s\mapsto\langle-d_{A(s)}^{\ast}\alpha-(\beta+\beta's)\delta\psi,\eta_2\rangle.
\end{align*}
Now
\begin{multline*}
\dot g_1(s)=\langle d_{A(s)}^{\ast}[\dot A(s)\wedge\alpha]-\ast[\dot A(s)\wedge\ast d_{A(s)}\alpha]+\ast[\ast d_{A(s)}\dot A(s)\wedge\alpha]\\
-\partial_s(\beta+\beta's)\delta\alpha-[\dot A(s)\wedge\psi],\eta_1\rangle.
\end{multline*}
We conclude continuity of $\dot g(s)$ at an arbitrary point $s_0\in\R$ from the following smoothness properties of the path $s\mapsto A(s)$. Let $I$ be some finite open interval containing $s_0$. By our regularity assumptions on $A$, $\dot A\in W^{1,2;p}(I\times\Sigma)$ for some $p>3$. Sobolev embedding thus implies that $\dot A\in C^0(I\times\Sigma)$. Interpolation moreover yields $\nabla_{A_0}\dot A\in C^0(I,L^p(\Sigma))$ for any continuous reference connection $A_0$. Now let $s_1\in I$. Then continuity of the map $s\mapsto\langle d_{A(s)}^{\ast}[\dot A(s)\wedge\alpha],\eta_1\rangle$ follows with $A_0\coloneqq A(s_0)$, $A_1\coloneqq A(s_1)$, and $\beta\coloneqq A_1-A_0$ from the estimate
\begin{eqnarray*}
|\langle d_{A_1}^{\ast}[\dot A_1\wedge\alpha]-d_{A_0}^{\ast}[\dot A_0\wedge\alpha], \eta_1\rangle|&\leq&|\langle d_{A_0}^{\ast}[(\dot A_1-\dot A_0)\wedge\alpha]-\ast[\beta\wedge\ast[\dot A_1\wedge\alpha]],\eta_1\rangle|\\
&\leq&c\|\eta_1\|_{L^2(\Sigma)}\|\nabla_{A_0}\alpha\|_{L^2(\Sigma)}\|\dot A_1-\dot A_0\|_{L^{\infty}(\Sigma)}\\
&&+c\|\eta_1\|_{L^2(\Sigma)}\|\alpha\|_{L^q(\Sigma)}\|\nabla_{A_0}(\dot A_1-\dot A_0)\|_{L^p(\Sigma)}\\
&&+c\|\eta_1\|_{L^2(\Sigma)}\|[\dot A_1\wedge\alpha]\|_{L^2(\Sigma)}\|\beta\|_{L^{\infty}(\Sigma)},
\end{eqnarray*}
where we let $p,q>1$ be such that $p^{-1}+q^{-1}=\frac{1}{2}$. Note that $\|\alpha\|_{L^q(\Sigma)}$ and $\|\nabla_{A_0}\alpha\|_{L^2(\Sigma)}$ are finite because of $\alpha\in\mathcal W^2\subseteq W^{1,2}(\Sigma)$ and continuity of the embedding $W^{1,2}(\Sigma)\hookrightarrow L^q(\Sigma)$ for every $1<q<\infty$. Continuity of the map $s\mapsto\langle\ast[\ast d_{A(s)}\dot A(s)\wedge\alpha],\eta_1\rangle$ follows similarly. The remaining terms in $\dot g_1$ and $\dot g_2$ can easily be estimated. (v) Let $A^{\pm}$ be the limiting connections as assumed in the theorem. Then the operators $\H_{A^{\pm}}^{\delta}\in\mathcal L(W,H)$ are invertible and are the limits of $\H_{A(s)}^{\delta}$ in the norm topology as $s\to\pm\infty$. Invertibility follows by the choice of the weight $\delta$. The exponential decay Theorem \ref{YMExponentialdecay} gives uniform convergence $A(s)\to A^{\pm}$, hence in particular norm convergence $\H_{A(s)}^{\delta}\to \H_{A^{\pm}}^{\delta}$ as $s\to\pm\infty$. Hence the assumptions of \cite[Theorem A]{RobSal} are satisfied by the operator family $s\mapsto\H_{A(s)}^{\delta}$. This proves Theorem \ref{thm:Fredholmtheorem} in the case $p=2$.
\end{proof}

The proof of Theorem \ref{thm:Fredholmtheorem} in the general case reduces by standard arguments to the case $p=2$.

\begin{proof}{\bf{(Theorem \ref{thm:Fredholmtheorem} in the case $1<p<\infty$)}}
We outline the proof. Full details can be found in \cite{Swoboda1}. Combining via a standard cut-off function argument the estimate \eqref{apriori2A} with bijectivity of the operator $\mathcal D_A^{\delta}$ for stationary paths $A(s)\equiv A^{\pm}$ we obtain the estimate
\begin{eqnarray*}
\|\xi\|_{\Z^p}\leq c(A)\big(\|\mathcal D_A^{\delta}\xi\|_{\L^p}+\|\xi\|_{L^p(I\times\Sigma)}\big)
\end{eqnarray*}
for a constant $c(A)$ and some compact interval $I\subseteq\R$. Hence it follows from the abstract closed range lemma (cf.~e.g.~\cite{Web}) that the operator $\mathcal D_A^{\delta}$ has finite-dimensional kernel and closed range. Similarly, one can show that $\coker\mathcal D_A^{\delta}$ is also finite-dimensional, and that the dimensions of the kernel and cokernel do not depend on $p$. This proves Theorem \ref{thm:Fredholmtheorem} in the general case.
\end{proof}

\section{Compactness}{\label{sec:Compactness}}
Let $a>0$ be a regular value of $\YM$. As in Section \ref{sect:modspaces} we fix an $a$-admissible perturbation $\V\in Y$ with $\|\V\|$ sufficiently small such that Proposition \ref{prop:samecritpoints} applies. We furthermore fix a pair $(\hat{\mathcal C}^-,\hat{\mathcal C}^+)\in\mathcal{CR}^a\times\mathcal{CR}^a$ of critical manifolds and denote $\mathcal C^{\pm}\coloneqq\hat{\mathcal C}^{\pm}/\G_0^{2,p}(P)$. We continue to use the notation introduced in Section \ref{sec:YMgradientflow} and denote by $\hat P_I\coloneqq I\times P$ the trivial extension of the principle $G$-bundle $P$ to the base manifold $I\times\Sigma$, where $I$ is some interval. As before we set $\hat P\coloneqq\hat P_{\R}$. Throughout we identify the pair 
\begin{eqnarray*}
(A,\Psi)\in C^{\infty}(\R,\A(P))\times C^{\infty}(\R,\Omega^0(\Sigma,\ad(P)))
\end{eqnarray*}
with the connection $\AA=A+\Psi\,ds\in\A(\hat P)$. As such, its curvature is given by 
\begin{eqnarray}\label{eq:formula3curv}
F_{\AA}=F_A+(d_A\Psi-\partial_sA)\wedge ds.
\end{eqnarray}
We use the symbols $\hat\ast$, $\hat d_{\mathbbm A}$, etc.~for the Hodge and differential operators acting on $\Omega^{\ast}(\mathbbm R\times\Sigma,\ad(\hat P_I))$. In particular, $\hat d_{\AA}$ and $\hat d_{\AA}^{\ast}=-\hat\ast\hat d_{\AA}\hat\ast$ act on $\ad(\hat P_I)$-valued $1$-forms $\alpha+\psi\,ds$ as
\begin{align}
\label{formulaextforms}\hat d_{\AA}(\alpha+\psi\,ds)=d_A\alpha+\big(d_A\psi+\partial_s\alpha-[\Psi,\alpha]\big)\wedge ds,\\
\label{formulaextforms1}\hat d_{\AA}^{\ast}(\alpha+\psi\,ds)=d_A^{\ast}\alpha-\partial_s\psi-[\Psi,\psi].
\end{align}
The Laplace operator $\hat\Delta_{\AA}$ on $0$-forms $\psi\in\Omega^0(\mathbbm R\times\Sigma,\ad(\hat P_I))$ is given by
\begin{eqnarray*} 
\hat\Delta_{\AA}\psi=\hat d_{\AA}^{\ast}\hat d_{\AA}\psi=\big(\Delta_A-\partial_s^2\big)\psi-\partial_s[\Psi,\psi]-\big[\Psi,\partial_s\psi+[\Psi,\psi]\big].
\end{eqnarray*}
A connection $\AA\in\A(\hat P)$ is said to be in local slice with respect to a reference connection $\AA_0\in\A(\hat P)$ if it satisfies the condition $\hat d_{\AA_0}^{\ast}(\AA-\AA_0)=0$. Writing $\AA=A+\Psi\,ds$ and $\AA_0=A_0+\Psi_0\,ds$ it follows from \eqref{formulaextforms1} that this condition is equivalent to
\begin{eqnarray}\label{locslicestep4}
d_{A_0}^{\ast}(A-A_0)-\partial_s(\Psi-\Psi_0)-[\Psi_0,\Psi]=0.
\end{eqnarray}

The aim of this section is to prove the following compactness theorem.  

\begin{thm}[Compactness]\label{thm:compactness}
Let $\AA^{\nu}=A^{\nu}+\Psi^{\nu}\,ds$, $\nu\in\mathbbm N$, be a sequence of solutions to the perturbed Yang--Mills gradient flow equation
\begin{eqnarray}\label{eqn:compactness1}
\partial_sA+d_A^{\ast}F_A-d_A\Psi+\nabla\V(A)=0
\end{eqnarray}
satisfying $(A^{\nu},\Psi^{\nu})\in\hat{\mathcal M}(\hat{\mathcal C}^-,\hat{\mathcal C}^+)$ for all $\nu\in\mathbbm N$. Here $\hat{\mathcal M}(\hat{\mathcal C}^-,\hat{\mathcal C}^+)$ is as defined in Section \ref{sect:modspaces} and $\hat{\mathcal C}^{\pm}$ denotes the pair of critical manifolds fixed before. Then there exists a sequence $(g^{\nu})\subseteq\mathcal G_{\loc}^{2,p}(\hat P)$ of gauge transformations such that a subsequence of $(g^{\nu})^{\ast}\AA^{\nu}\coloneqq A_1^{\nu}+\Psi_1^{\nu}\,ds$ converges to a solution $\AA^{\ast}=A^{\ast}+\Psi^{\ast}\,ds$ of \eqref{eqn:compactness1} in the sense that 
\begin{eqnarray}\label{eq:convergencesubsequ}
A_1^{\nu}\to A^{\ast}\quad\textrm{in}\quad\A^{1,2;p}(\hat P_I)\qquad\textrm{and}\qquad\Psi_1^{\nu}\to\Psi^{\ast}\quad\textrm{in}\quad W^{1,p}(I\times\Sigma)
\end{eqnarray}
as $\nu\to\infty$, for every compact interval $I$.
\end{thm}

\begin{proof}
The proof follows from Theorem \ref{thm:boundedness} below. This theorem asserts for every compact interval $I$ the existence of a sequence of gauge transformations $g_I^{\nu}\in\G^{2,p}(\hat P_I)$ such that a subsequence of $(g_I^{\nu})^{\ast}\AA^{\nu}\eqqcolon A_1^{\nu}+\Psi_1^{\nu}\,ds$ satisfies a uniform bound of the form
\begin{eqnarray}\label{eq:unifboundscompproof}
\|A_1^{\nu}\|_{W^{1,2;p}(I\times\Sigma)}+\|\dot A_1^{\nu}\|_{W^{1,2;p}(I\times\Sigma)}+\|\Psi_1\|_{W^{2,p}(I\times\Sigma)}\leq C(I).
\end{eqnarray}
Let $A_0\in\A(P)$ be a smooth reference connection. From \eqref{eq:unifboundscompproof} it follows that the sequences $(\partial_sA_1^{\nu})$ and $(\nabla_{A_0}^2A_1^{\nu})$ are bounded in $W^{1,p}(I\times\Sigma)$. The existence of a convergent subsequence of $(A_1^{\nu})$, and similarly of $(\Psi_1^{\nu})$ as asserted in \eqref{eq:convergencesubsequ} then follows from Rellich's theorem. Now denote $I_n\coloneqq[-n,n]$ for $n\in\mathbbm N$. As $\R\times\Sigma$ is exhausted by the compact sets $I_n\times\Sigma$ it follows from standard patching arguments as e.g.~in \cite[Theorem 4.16]{Frauenfelder2} and \cite[Proposition 3.6]{Wehrheim} that it exists a sequence $(g^{\nu})\subseteq\mathcal G_{\loc}^{2,p}(\hat P)$ such that the above convergence on $I_n\times\Sigma$ for every $n\in\mathbbm N$ holds with $g_{I_n}^{\nu}\in\G^{2,p}(\hat P_{I_n})$ replaced by $g^{\nu}|_{I_n}$. With this sequence $(g^{\nu})$ the theorem follows.
\end{proof}

\begin{thm}\label{thm:boundedness} 
Let $I=[b,c]$ be a compact interval and $\AA^{\nu}=A^{\nu}+\Psi^{\nu}\,ds$, $\nu\in\mathbbm N$, be a sequence of connections on $\hat P_I$ where $A^{\nu}\in\A^{1,2;p}(P)$ and $\Psi^{\nu}\in W^{1,p}(I\times\Sigma)$. Assume that  
\begin{eqnarray*}
\limsup_{\nu\to\infty}\YMV(A^{\nu}(b))\leq a
\end{eqnarray*}
is satisfied for the constant $a\geq0$ fixed initially. Then there exists a constant $C(I)$, a sequence $(g^{\nu})\subseteq\mathcal G^{2,p}(\hat P_I)$ of gauge transformations, and a smooth connection $\mathbbm A^{\infty}=A^{\infty}+\Psi^{\infty}\,ds\in\A(\hat P_I)$ such that (after extraction of a subsequence which we again label by $\nu$)
\begin{eqnarray*}
\beta^{\nu}\coloneqq(g^{\nu})^{\ast}A^{\nu}-A^{\infty},\qquad\psi^{\nu}\coloneqq(g^{\nu})^{\ast}\Psi^{\nu}-\Psi^{\infty}
\end{eqnarray*} 
satisfies the uniform bound
\begin{multline*}
\|\beta^{\nu}\|_{L^p(I\times\Sigma)}+\|\partial_s\beta^{\nu}\|_{W^{1,2;p}(I\times\Sigma)}+\|\nabla_{A^{\infty}}\beta^{\nu}\|_{W^{1,2;p}(I\times\Sigma)}\\
+\|\psi^{\nu}\|_{W^{2,p}(I\times\Sigma)}+\|\partial_s\psi^{\nu}\|_{W^{2,p}(I\times\Sigma)}\leq C(I)
\end{multline*}
for all $\nu\in\mathbbm N$.
\end{thm}

\begin{proof} 
The proof, which we divide into several steps, is based on Uhlenbeck's weak compactness theorem and the existence of local slices.
\setcounter{step}{0}
\begin{step}
Let $1<p<4$. There exists a constant $C(p)$ such that the curvature bound $\|F_{\AA^{\nu}}\|_{L^p(I\times\Sigma)}\leq C(p)$ is satisfied for all $\nu\in\mathbbm N$.
\end{step}
Since the estimate is invariant under gauge transformations in $\G(\hat P_I)$ it suffices to prove it for $\Psi^{\nu}=0$. Then, as follows from \eqref{eq:formula3curv} and \eqref{eqn:compactness1}, the curvature is given by 
\begin{eqnarray*}
F_{\AA^{\nu}}=F_{A^{\nu}}+(d_{A^{\nu}}^{\ast}F_{A^{\nu}}+\nabla\V(A^{\nu}))\,ds.
\end{eqnarray*}
Uniform $L^p$ bounds for the terms $F_{A^{\nu}}$ and $d_{A^{\nu}}^{\ast}F_{A^{\nu}}$ hold by Lemmata \ref{lem:univcurvatureestimate0} and \ref{lem:univcurvatureestimate1}. With $\V=\sum_{\ell=1}^{\infty}\lambda_{\ell}\V_{\ell}\in Y$, a uniform estimate for $\nabla\V(A^{\nu})$ is provided by condition (iii) in Section \ref{Bspaceperturbations} from which it follows that
\begin{eqnarray*}
\big\|\sum_{\ell=1}^{\infty}\lambda_{\ell}\nabla\V_{\ell}(A^{\nu})\big\|_{L^p(I\times\Sigma)}&\leq&\sum_{\ell=1}^{\infty}|\lambda_{\ell}|\cdot\|\nabla\V_{\ell}(A^{\nu})\|_{L^p(I\times\Sigma)}\\
&\leq&\sum_{\ell=1}^{\infty}|\lambda_{\ell}|\cdot\Big(\int_IC_{\ell}^p\big(1+\|F_{A^{\nu}(s)}\|_{L^3(\Sigma)}\big)^p\,ds\Big)^{\frac{1}{p}}\\
&\leq&2^p\sum_{\ell=1}^{\infty}C_{\ell}|\lambda_{\ell}|\cdot|I|+2^p\sum_{\ell=1}^{\infty}C_{\ell}|\lambda_{\ell}|\cdot\|F_{A^{\nu}}\|_{L^p(I,L^3(\Sigma))}\\
&=&2^p(|I|+\|F_{A^{\nu}}\|_{L^p(I,L^3(\Sigma))})\|\V\|.
\end{eqnarray*}
Again by Lemma \ref{lem:univcurvatureestimate0}, the term $\|F_{A^{\nu}}\|_{L^p(I,L^3(\Sigma))}$ is uniformly bounded.
  
\begin{step}
Let $3<p<4$ and choose $\eps>0$. There exists a sequence $g^{\nu}\in\G^{2,p}(\hat P_I)$ of gauge transformations and a smooth reference connection $\AA^{\infty}=A^{\infty}+\Psi^{\infty}\,ds$ such that (up to extraction of a subsequence) the sequence $(g^{\nu})^{\ast}\AA^{\nu}$ satisfies the following three conditions.  
\begin{compactenum}[(i)]
\item
Each connection $(g^{\nu})^{\ast}\AA^{\nu}$ satisfies the local slice condition \eqref{locslicestep4} with respect to the reference connection $\AA^{\infty}$. 
\item
The difference $\beta^{\nu}+\psi^{\nu}\,ds\coloneqq(g^{\nu})^{\ast}\AA^{\nu}-\AA^{\infty}$ is uniformly bounded in $W^{1,p}(I\times\Sigma)$.
\item 
The sequence $\beta^{\nu}+\psi^{\nu}\,ds$ satisfies the uniform bound 
\begin{eqnarray*}
\|\beta^{\nu}\|_{C^0(I\times\Sigma)}+\|\psi^{\nu}\|_{C^0(I\times\Sigma)}<\eps.
\end{eqnarray*}
\end{compactenum}
\end{step}
The sequence $\AA^{\nu}$ satisfies a uniform $L^p$ curvature bound by Step 1. Hence Uhlenbeck's weak compactness theorem (cf.~\cite[Theorem 7.1]{Wehrheim}) yields a sequence $g^{\nu}\in\G^{2,p}(\hat P_I)$ of gauge transformations such that a subsequence of $(g^{\nu})^{\ast}\AA^{\nu}$ converges weakly in $W^{1,p}(I\times\Sigma)$ to some limit connection $\AA'$. This sequence is in particular bounded in $W^{1,p}(I\times\Sigma)$ and contains (by compactness of the embedding $W^{1,p}(I\times\Sigma)\hookrightarrow C^0(I\times\Sigma)$ for $p>3$) a subsequence which converges in $C^0(I\times\Sigma)$ to $\AA'$. We label this subsequence again by $\nu$. Now let $C,\delta>0$ be the constants in the statement of the local slice Theorem \ref{thm:locslicethm} with parameters $p>3$ and $q=\infty$. We then replace $\AA'$ by a smooth reference connection $\AA^{\infty}=A^{\infty}+\Psi^{\infty}\,ds$ such that $\|\AA'-\AA^{\infty}\|_{C^0(\Sigma)}<\min\{\frac{\delta}{2},\frac{\eps}{C}\}$. It then follows from Theorem \ref{thm:locslicethm} that for every large enough $\nu$ the connection $(g^{\nu})^{\ast}\AA^{\nu}$ can be put in local slice with respect to $\AA^{\infty}$. Therefore condition (i) is satisfied. Moreover, the same theorem asserts that this can be done preserving the uniform bound in $W^{1,p}(I\times\Sigma)$ and the uniform bound (with constant $\eps$) in $C^0(I\times\Sigma)$. Thus also conditions (ii) and (iii) are satisfied.

\begin{step}
The sequences $(\psi)^{\nu}$ and  $(\beta)^{\nu}$ are uniformly bounded in $W^{1,2;p}(I\times\Sigma)$, respectively in $W^{2,p}(I\times\Sigma)$, for every $p<\infty$.
\end{step}
After applying a smooth gauge transformation to the sequence $(\AA^{\nu})$, we may assume that the assertions of Step 2 continue to hold with $\Psi^{\infty}=0$. For convenience we drop the index $\nu$ in the subsequent calculations. Expanding $d_A$, $d_A^{\ast}$ and $F_A$ as
\begin{eqnarray*}
d_A=d_{A^{\infty}}+[\beta\wedge\,\cdot\,],\quad d_A^{\ast}=d_{A^{\infty}}^{\ast}-\ast[\beta\wedge\ast\,\cdot\,],\quad F_A=F_{A^{\infty}}+d_{A^{\infty}}\beta+\frac{1}{2}[\beta\wedge\beta],
\end{eqnarray*}
equation \eqref{eqn:compactness1} reads
\begin{multline}\label{proofEYF1}
0=\partial_sA^{\infty}+\partial_s\beta+d_{A^{\infty}}^{\ast}F_{A^{\infty}}-\ast\big[\beta\wedge\ast\big(F_{A^{\infty}}+d_{A^{\infty}}\beta+\frac{1}{2}[\beta\wedge\beta]\big)\big]\\
+d_{A^{\infty}}^{\ast}d_{A^{\infty}}\beta+\frac{1}{2}d_{A^{\infty}}^{\ast}[\beta\wedge\beta]-d_{A^{\infty}}\psi-[\beta\wedge\psi]+\nabla\V(A).
\end{multline}
We combine this equation with the local slice condition \eqref{locslicestep4} to obtain for $\beta$ the parabolic PDE
\begin{multline}\label{proofalphaparabolic}
\partial_s\beta+\Delta_{A^{\infty}}\beta=-\partial_sA^{\infty}-d_{A^{\infty}}^{\ast}F_{A^{\infty}}-\frac{1}{2}d_{A^{\infty}}^{\ast}[\beta\wedge\beta]+d_{A^{\infty}}\partial_s\psi\\
+\ast\big[\beta\wedge\ast\big(F_{A^{\infty}}+d_{A^{\infty}}\beta+\frac{1}{2}[\beta\wedge\beta]\big)\big]+d_{A^{\infty}}\psi+[\beta\wedge\psi]-\nabla\V(A).
\end{multline}
Applying $d_{A^{\infty}}^{\ast}$ to both sides of equation \eqref{proofEYF1}, substituting
\begin{eqnarray*}
d_{A^{\infty}}^{\ast}\partial_s\beta=\partial_s^2\psi+\ast[\partial_sA^{\infty}\wedge\ast\beta]
\end{eqnarray*}
according to \eqref{locslicestep4}, and using that
\begin{eqnarray*}
d_{A^{\infty}}^{\ast}\nabla\V(A)=d_A^{\ast}\nabla\V(A)+\ast[\beta\wedge\ast\nabla\V(A)]=\ast[\beta\wedge\ast\nabla\V(A)]
\end{eqnarray*}
yields for $\psi$ the elliptic PDE
\begin{multline}\label{proofellipticpsi}
\hat\Delta_{A^{\infty}}\psi=d_{A^{\infty}}^{\ast}\partial_sA^{\infty}+\ast[\partial_sA^{\infty}\wedge\ast\beta]+\ast\big[\beta\wedge d_{A^{\infty}}\ast\big(F_{A^{\infty}}+\frac{1}{2}[\beta\wedge\beta]\big)\big]\\
+\ast[\beta\wedge\ast d_{A^{\infty}}^{\ast}d_{A^{\infty}}\beta]-\frac{1}{2}[\ast d_{A^{\infty}}\beta\wedge\ast[\beta\wedge\beta]]-\frac{1}{2}[\ast F_{A^{\infty}}\wedge\ast[\beta\wedge\beta]]\\
-d_{A^{\infty}}^{\ast}[\beta\wedge\psi]+\ast[\beta\wedge\ast\nabla\V(A)].
\end{multline}
Let $p>1$ arbitrary. We fix a compact interval $I_1$ such that $I\subseteq\interior I_1$ and a smooth cut-off function $\ph$ such that $\ph|_{I\times\Sigma}\equiv1$ and $\supp\ph\subseteq\interior I_1\times\Sigma$. To be able to obtain estimates on the whole domain $I\times\Sigma$ we replace $\beta$ and $\psi$ by $\hat\beta\coloneqq\ph\beta$ and $\hat\psi\coloneqq\ph\psi$. Then \eqref{proofalphaparabolic} and \eqref{proofellipticpsi} yield equations of the same type for $\hat\beta$ and $\hat\psi$, with a number of additional terms involving $\ph$ and derivatives of $\beta$ and $\psi$ of order at most one in the spacial variables (which are uniformly bounded in $L^p(I_1\times\Sigma)$ by Step 2). Then the arguments below give the desired uniform bound for $\hat\beta$ and $\hat\psi$ on the domain $I_1\times\Sigma$. By choice of the cut-off function $\ph$ this implies the claim for $\beta$ and $\psi$. To keep the exposition short (and avoid to write down further terms involving $\ph$) we only carry out the main argument for the original $\beta$ and $\psi$. From \eqref{proofalphaparabolic} it follows by standard parabolic regularity theory that, for a positive constant $c=c(A^{\infty},I,p)$,
\begin{multline}\label{est:parabolicbeta}
c^{-1}\|\beta\|_{W^{1,2;p}}\leq1+\|\beta\|_{L^p}+\|\{\beta,[\beta\wedge\beta]\}\|_{L^p}+\|\{\nabla_{A^{\infty}}\beta,\beta\}\|_{L^p}\\
+\|\{\beta,\psi\}\|_{L^p}+\|d_{A^{\infty}}\psi\|_{L^p}+\|d_{A^{\infty}}\partial_s\psi\|_{L^p}+\|\nabla\V(A)\|_{L^p}.
\end{multline}
From \eqref{proofellipticpsi} and elliptic regularity we obtain for a constant $c=c(I,p)$ the estimate
\begin{multline}\label{est:ellipticpsi}
c^{-1}\|\psi\|_{W^{2,p}}\leq1+\|\psi\|_{L^p}+\|\beta\|_{L^p}+\|\{\beta,\beta\}\|_{L^p}+\|\nabla_{A^{\infty}}\beta\|_{L^p}+\|\nabla_{A^{\infty}}[\beta\wedge\psi]\|_{L^p}\\
+\|d_{A^{\infty}}\{\beta,[\beta,\beta]\}\|_{L^p}+\|\{\beta,d_{A^{\infty}}^{\ast}d_{A^{\infty}}\beta\}\|_{L^p}+\|[\beta\wedge\ast\nabla\V(A)]\|_{L^p}.
\end{multline}
Now let $3<p<4$. By Step 2 there holds a uniform bound for $\|\beta\|_{C^0}$ and $\|\beta\|_{W^{1,p}}$. The term $\|\nabla\V(A)\|_{L^p}$ is uniformly bounded as shown in Step 1. It thus follows that each term on the right-hand side of \eqref{est:parabolicbeta}, except the term $\|d_{A^{\infty}}\partial_s\psi\|_{L^p}$, is uniformly bounded. It is estimated using \eqref{est:ellipticpsi}. Note that the expression
\begin{eqnarray*}
\|\{\beta,d_{A^{\infty}}^{\ast}d_{A^{\infty}}\beta\}\|_{L^p}\leq c\|\beta\|_{C^0}\|d_{A^{\infty}}^{\ast}d_{A^{\infty}}\beta\|_{L^p}
\end{eqnarray*}
appearing in \eqref{est:ellipticpsi} becomes absorbed by the left-hand side of \eqref{est:parabolicbeta} after fixing $\eps$ in condition (iii) of Step 2 sufficiently small. Hence it follows that the sequence $(\psi^{\nu})$ is uniformly bounded in $W^{2,p}(I\times\Sigma)$ and $(\beta^{\nu})$ is uniformly bounded in $W^{1,2;p}(I\times\Sigma)$, for every $p<4$. We now can iterate the argumentation so far to obtain the same bounds for every $p<\infty$. To be precise, Sobolev embedding yields that $\nabla_{A^{\infty}}\psi^{\nu}$ admits a uniform bound in $C^0(I\times\Sigma)$. From standard interpolation for anisotropic Sobolev spaces and the Sobolev embedding theorem we obtain a uniform bound for $\nabla_{A^{\infty}}\beta^{\nu}$ in $L^{p_1}(I\times\Sigma)$ for every $p_1<\frac{9p}{9-p}$. From this it follows that the previous elliptic and parabolic estimates apply with $p$ replaced by $\frac{3p}{2}$. Repeating this argument a finite number of times, we inductively obtain uniform bounds  for $(\psi)^{\nu}$ in $W^{1,2;p}(I\times\Sigma)$ and for $(\beta)^{\nu}$ in $W^{2,p}(I\times\Sigma)$, for every $p<\infty$.

\begin{step}
We prove the theorem.
\end{step}
Similarly to Step 1, a uniform estimate for $\partial_s\nabla\V(A^{\nu})$ is provided by condition (iv) of Section \ref{Bspaceperturbations} from which it follows for $p\geq3$ that
\begin{eqnarray}\label{eq:estpartialnablaV}
\nonumber\lefteqn{\big\|\sum_{\ell=1}^{\infty}\lambda_{\ell}\partial_s\nabla\V_{\ell}(A^{\nu})\big\|_{L^p(I\times\Sigma)}}\\
\nonumber&\leq&\sum_{\ell=1}^{\infty}|\lambda_{\ell}|\cdot\|H_A\V_{\ell}(\dot A^{\nu})\|_{L^p(I\times\Sigma)}\\
\nonumber&\leq&\sum_{\ell=1}^{\infty}|\lambda_{\ell}|\cdot\Big(\int_IC_{\ell}^p\big(1+\|F_{A^{\nu}(s)}\|_{L^3(\Sigma)}\big)^p\|\dot A^{\nu}\|_{L^p(\Sigma)}^p\,ds\Big)^{\frac{1}{p}}\\
\nonumber&\leq&2^p\sum_{\ell=1}^{\infty}C_{\ell}|\lambda_{\ell}|\cdot\|\dot A^{\nu}\|_{L^p(I\times\Sigma)}+2^p\sum_{\ell=1}^{\infty}C_{\ell}|\lambda_{\ell}|\cdot\|F_{A^{\nu}}\|_{L^p(I\times\Sigma)}\|\dot A^{\nu}\|_{L^p(I\times\Sigma)}\\
&=&2^p\big(1+\|F_{A^{\nu}}\|_{L^p(I\times\Sigma)}\big)\|\dot A^{\nu}\|_{L^p(I\times\Sigma)}\|\V\|.
\end{eqnarray}
Differentiating \eqref{proofalphaparabolic} with respect to $s$ and denoting
\begin{eqnarray*}
\dot\Delta_{A^{\infty}}\coloneqq-\ast[\dot A^{\infty}\wedge\ast d_{A^{\infty}}\,\cdot\,]+d_{A^{\infty}}^{\ast}[\dot A^{\infty}\wedge\,\cdot\,]+[\dot A^{\infty}\wedge d_{A^{\infty}}^{\ast}\,\cdot\,]-d_{A^{\infty}}\ast[\dot A^{\infty}\wedge\ast\,\cdot\,]
\end{eqnarray*}
we obtain for $\dot\beta$ the parabolic PDE
\begin{multline}\label{proofdotalphaparabolic}
\partial_s\dot\beta+\Delta_{A^{\infty}}\dot\beta=-\dot\Delta_{A^{\infty}}\beta-\ddot A^{\infty}-\partial_sd_{A^{\infty}}^{\ast}F_{A^{\infty}}+\frac{1}{2}\ast[\dot A^{\infty}\wedge\ast[\beta\wedge\beta]]\\
-d_{A^{\infty}}^{\ast}[\dot\beta\wedge\beta]+\partial_sd_{A^{\infty}}\partial_s\psi+\ast\big[\dot\beta\wedge\ast\big(F_{A^{\infty}}+d_{A^{\infty}}\beta+\frac{1}{2}[\beta\wedge\beta]\big)\big]\\
+\ast\big[\beta\wedge\ast\partial_s\big(F_{A^{\infty}}+d_{A^{\infty}}\beta+\frac{1}{2}[\beta\wedge\beta]\big)\big]+\partial_sd_{A^{\infty}}\psi+\partial_s[\beta\wedge\psi]-\partial_s\nabla\V(A).
\end{multline}
Differentiating \eqref{proofellipticpsi} with respect to $s$ we obtain for $\dot\psi$ the elliptic PDE
\begin{multline}\label{proofellipticdotpsi}
\hat\Delta_{A^{\infty}}\dot\psi=-\dot\Delta_{A^{\infty}}\psi+\partial_sd_{A^{\infty}}^{\ast}\partial_sA^{\infty}+\ast\partial_s[\partial_sA^{\infty}\wedge\ast\beta]+\ast[\dot\beta\wedge\ast d_{A^{\infty}}^{\ast}d_{A^{\infty}}\beta]\\
+\ast\partial_s\big[\beta\wedge d_{A^{\infty}}\ast\big(F_{A^{\infty}}+\frac{1}{2}[\beta\wedge\beta]\big)\big]+\ast[\beta\wedge\ast\partial_sd_{A^{\infty}}^{\ast}d_{A^{\infty}}\beta]-\frac{1}{2}\partial_s[\ast d_{A^{\infty}}\beta\wedge\ast[\beta\wedge\beta]]\\
-\frac{1}{2}\partial_s[\ast F_{A^{\infty}}\wedge\ast[\beta\wedge\beta]]
-\partial_sd_{A^{\infty}}^{\ast}[\beta\wedge\psi]+\ast[\dot\beta\wedge\ast\nabla\V(A)]+\ast[\beta\wedge\ast\partial_s\nabla\V(A)].
\end{multline}
We now use equations \eqref{proofdotalphaparabolic} and \eqref{proofellipticdotpsi} and argue as in Step 3 to obtain the asserted uniform bounds for $\|\dot\beta\|_{W^{1,2;p}(I\times\Sigma)}$ and $\|\dot\psi\|_{W^{2,p}(I\times\Sigma)}$. Namely, with the exception of the terms
\begin{eqnarray}\label{missingterms}
\partial_sd_{A^{\infty}}\dot\psi,\qquad d_{A^{\infty}}^{\ast}[\dot\beta\wedge\beta],\qquad[\beta\wedge\ast\partial_sd_{A^{\infty}}\beta],\qquad[\ast\partial_sd_{A^{\infty}}\beta\wedge\ast[\beta\wedge\beta]], 
\end{eqnarray}
the right-hand side of \eqref{proofdotalphaparabolic} is uniformly bounded in $L^p(I\times\Sigma)$ for some sufficiently small $1<r\leq p$. This follows from the uniform $W^{1,2;p}$ and $W^{2,p}$ bounds obtained in Step 3 for $\beta$, respectively $\psi$, and estimate \eqref{eq:estpartialnablaV}. The term $\partial_sd_{A^{\infty}}\dot\psi$ can be estimated using \eqref{proofellipticdotpsi}. The term $d_{A^{\infty}}^{\ast}[\dot\beta\wedge\beta]$ is bounded in $L^r$ (for $1<r<\frac{p}{2}$) as
\begin{eqnarray*}
\|d_{A^{\infty}}^{\ast}[\dot\beta\wedge\beta]\|_{L^r(I\times\Sigma)}\leq c\big(\|\{\dot\beta,\nabla_{A^{\infty}}\beta\}\|_{L^r(I\times\Sigma)}+\|\beta\|_{C^0(I\times\Sigma)}\|\nabla_{A^{\infty}}\dot\beta\|_{L^r(I\times\Sigma)}\big),
\end{eqnarray*}
the latter expression being absorbed in the left-hand side of the parabolic estimate for $\dot\beta$ after fixing $\eps$ in condition (iii) of Step 2 sufficiently small. The last two terms in \eqref{missingterms} are estimated in the same way. Similarly, the right-hand side of \eqref{proofellipticdotpsi} is uniformly bounded in $L^r(I\times\Sigma)$ with the exception of the terms $[\beta\wedge\partial_sd_{A^{\infty}}\ast[\beta\wedge\beta]]$, $[\beta\wedge\partial_sd_{A^{\infty}}^{\ast}d_{A^{\infty}}\beta]$, and $\partial_sd_{A^{\infty}}^{\ast}[\beta\wedge\psi]$. The first two of these can be absorbed as before. For the last one we estimate
\begin{multline}\label{eq:smallsecondderivbeta}
\|\partial_sd_{A^{\infty}}^{\ast}[\beta\wedge\psi]\|_{L^r(I\times\Sigma)}\leq c\big(\|\{\nabla_{A^{\infty}}\beta,\dot\psi\}+\{\dot\beta,\nabla_{A^{\infty}}\psi\}\\
+\{\beta,\partial_s\nabla_{A^{\infty}}\psi\}+\|\psi\|_{C^0(I\times\Sigma)}\|\partial_s\nabla_{A^{\infty}}\beta\|_{L^r(I\times\Sigma)}\big).
\end{multline}
After fixing $\eps$ in condition (iii) of Step 2 still smaller if necessary (and replacing $\partial_s\nabla_{A^{\infty}}\beta$ by $\nabla_{A^{\infty}}\dot\beta+[\dot A^{\infty}\wedge\beta]$), the last term in \eqref{eq:smallsecondderivbeta} can be absorbed in the left-hand side of the parabolic estimate for $\dot\beta$. We therefore arrive at uniform bounds for $\|\dot\beta\|_{W^{1,2;r}(I\times\Sigma)}$ and $\|\dot\psi\|_{W^{2,r}(I\times\Sigma)}$. Bootstrap arguments as in Step 3 the yield the desired uniform bounds for all $r\leq p$. It remains to show that also $\nabla_{A^{\infty}}\beta^{\nu}$ is uniformly bounded in $W^{1,2;p}(I\times\Sigma)$. With the estimates already proved this is implied by the following standard argument. Namely, using the Bochner--Weitzenb\"ock formula \eqref{BWformula} and the commutator identity \eqref{commutatorid} we obtain 
\begin{multline*}
\big(\frac{d}{ds}+\Delta_{A^{\infty}}\big)\nabla_{A^{\infty}}\beta^{\nu}=\nabla_{A^{\infty}}\big(\frac{d}{ds}+\Delta_{A^{\infty}}\big)\beta^{\nu}+\{\beta^{\nu},\nabla_{A^{\infty}}\beta^{\nu}\}\\
+\{F_{A^{\infty}},\nabla_{A^{\infty}}\beta^{\nu}+\beta^{\nu}\}+\{R_{\Sigma},\nabla_{A^{\infty}}\beta^{\nu}+\beta^{\nu}\}.
\end{multline*}
The right-hand side of this equation is uniformly bounded in $L^p(I\times\Sigma)$ as follows from what we have already shown and the parabolic equation \eqref{proofalphaparabolic} satisfied by $\beta^{\nu}$ (note that applying $\nabla_{A^{\infty}}$ to the right-hand side of \eqref{proofalphaparabolic} yields an expression uniformly bounded in $L^p(I\times\Sigma)$). Now parabolic regularity for the operator $\frac{d}{ds}+\Delta_{A^{\infty}}$ shows that $\nabla_{A^{\infty}}\beta^{\nu}$ is uniformly bounded in $W^{1,2;p}(I\times\Sigma)$, as desired. This completes the proof of the theorem. 
\end{proof}

\section{Transversality}\label{sect:transversality}

\subsection{Universal moduli space}
Throughout we fix a regular value $a>0$ of $\YM$. For the definition of the Banach space $Y$ we refer to Section \ref{Bspaceperturbations}. In this section we use the notation $Y^a$ for the closed subspace of $a$-admissible perturbations as introduced in Definition \ref{def:regperturbation}. We furthermore fix a pair of {\emph{disjoint}} critical manifolds $\hat{\mathcal C}^{\pm}\in\mathcal{CR}^a$. We set $\mathcal C^{\pm}\coloneqq\frac{\hat{\mathcal C}^{\pm}}{\G_0^{2,p}(P)}$ and consider the smooth Banach space bundle
\begin{eqnarray*}
\E=\E(\mathcal C^-,\mathcal C^+,\delta,p)\to\B(\mathcal C^-,\mathcal C^+,\delta,p)\times Y^a,
\end{eqnarray*}
cf.~Section \ref{Banachmanifolds} for definitions. We define the smooth section $\F$ of $\E$ by 
\begin{eqnarray}\label{eq:universalsection}
\mathcal F\colon[(A,\Psi,\mathcal V)]\mapsto[\partial_sA+d_A^{\ast}F_A-d_A\Psi+\nabla\V(A)],
\end{eqnarray}
and call its zero set $\mathcal M^{\univ}(\mathcal C^-,\mathcal C^+)\coloneqq\F^{-1}(0)$ the \emph{universal moduli space}. Thus the perturbation $\mathcal V$ which had been kept fixed so far is now allowed to vary over the Banach space $Y^a$.\\
\noindent\\
Let $u=[(A,\Psi,\V)]\in\mathcal M^{\univ}(\mathcal C^-,\mathcal C^+)$. In view of Proposition \ref{prop:temporalgauge} we may assume, after applying a suitable gauge transformation, that $\Psi=0$. As in \eqref{eq:universalsection}, we let $\hat{\mathcal D}_{(A,\V)}\coloneqq d_u\mathcal F$ denote the horizontal differential at $(A,0,\V)$ of the section $\F$. The discussion of the operator $\hat{\mathcal D}_{(A,\V)}$ parallels the one in Section \ref{subsect:linearizedoperator}. We put 
\begin{eqnarray*}
\hat{\mathcal D}_{(A,\V)}\colon\Z^{\delta,p}\times Y^a\to\mathcal L^{\delta,p},\quad(\alpha,\psi,v)\mapsto\mathcal D_{A}(\alpha,\psi)+\nabla v(A),
\end{eqnarray*}
with $\mathcal D_A$ being defined in \eqref{eq:hordifferential} and Banach spaces $\Z^{\delta,p}$ and $\mathcal L^{\delta,p}$ as in Section \ref{subsect:linearizedoperator}. Note that $\hat{\mathcal D}_{(A,\V)}$ is the sum of the Fredholm operator $\mathcal D_{A}$ and the bounded operator $v\mapsto\nabla v(A)$, and therefore has closed range. The Fredholm property of $\mathcal D_{A}$ has been shown in Theorem \ref{thm:Fredholmtheorem}. With $v=\sum_{\ell=1}^{\infty}\lambda_{\ell}\V_{\ell}\in Y^a$, the assertion on boundedness follows from the estimate
\begin{eqnarray*}
\|\nabla v(A)\|_{L^p(\R\times\Sigma)}=\|\nabla v(A)\|_{L^p([-T,T]\times\Sigma)}\leq2^p(2T+\|F_{A^{\nu}}\|_{L^p([-T,T],L^3(\Sigma))})\|v\|.
\end{eqnarray*}
The first identity holds for some constant $T=T(A)<\infty$ because $A(s)$ is contained in the support of $v$ only for some finite time interval (by assumption, $\supp v$ is contained in the complement of some $L^2$ neighborhood of $\hat{\mathcal C}^-\cup\hat{\mathcal C}^+$). The last inequality was shown in Step 1 of the proof of Theorem \ref{thm:boundedness}. 

\begin{thm}[Transversality]\label{thm:linsurjetivederef}
The horizontal differential $d_u\F$ of the map $\F$ is surjective for every $u\in\mathcal M^{\univ}(\mathcal C^-,\mathcal C^+)$.
\end{thm}

\begin{proof} 
As argued before, the operator $\hat{\mathcal D}_{(A,\V)}$ has closed range. By Proposition \ref{prop:denserange} below this range is dense in $\mathcal L^{\delta,p}$. These two properties together imply that $\hat{\mathcal D}_{(A,\V)}$ and hence $d_u\F$ is surjective.
\end{proof}

We temporarily use the notation $\mathcal M(\mathcal C^-,\mathcal C^+;\mathcal V)$ for the moduli space as in \eqref{modulispace}, defined with respect to the fixed perturbation $\V\in Y^a$. From Theorem \ref{thm:linsurjetivederef} and the implicit function theorem it follows that the universal moduli space $\mathcal M^{\univ}(\mathcal C^-,\mathcal C^+)$ is a smooth Banach manifold. Let $\pi\colon\mathcal M^{\univ}(\mathcal C^-,\mathcal C^+)\to Y^a$ denote the projection to the second factor. Because $\mathcal D_A$ is Fredholm (Theorem \ref{thm:Fredholmtheorem}), $d_u\F$ is surjective (Theorem \ref{thm:linsurjetivederef}), and the map $v\mapsto\nabla v(A)$ is bounded (cf.~the preceding paragraph) we have the following two facts. Both follow from part (ii) of \cite[Proposition 3.3]{Weber0}. First, the map $\pi$ is a smooth Fredholm map whose index is given by the Fredholm index of $\mathcal D_A$. Second, the closed subspace $\ker d_u\F$ admits a topological complement in $\Z^{\delta,p}\times Y^a$. Hence we may apply to $\pi$ the Sard--Smale theorem for Fredholm maps between Banach manifolds, cf.~the book \cite[Theorem 3.6.15]{Abraham}, from which it follows that the set of regular values 
\begin{eqnarray*}
\mathcal R\coloneqq\big\{\mathcal V\in Y^a\,\big|\,d_u\pi\,\textrm{is surjective for all}\,u\in\mathcal M(\mathcal C^-,\mathcal C^+;\mathcal V)\big\}\subseteq Y^a
\end{eqnarray*}
is residual in $Y^a$. Hence in particular, there exists a regular value $\mathcal V^{\reg}\in\mathcal R$ in every arbitrarily small ball $B_{\eps}(0)$ (with respect to the norm on $Y^a$) around zero. For every such $\mathcal V^{\reg}$, the moduli space $\mathcal M(\mathcal C^-,\mathcal C^+;\mathcal V^{\reg})$ is a submanifold of $\mathcal M^{\univ}(\mathcal C^-,\mathcal C^+)$ of dimension equal to $\ind\mathcal D_A$.

\subsection{Surjectivity of linearized operators}\label{sec:Surjectivitylinoperators}

It remains to state and prove Proposition \ref{prop:denserange} used in the proof of Theorem \ref{thm:linsurjetivederef}. In the following let $1<q<\infty$ be the dual Sobolev exponent to $p$ and denote
\begin{eqnarray}\label{eq:compnentseta}
\eta=(\eta_1,\eta_2)\in(\mathcal L^{\delta,p})^{\ast}=L_{-\delta}^q(\mathbbm R\times\Sigma,T^{\ast}\Sigma\otimes\ad(P))\oplus L_{-\delta}^q(\mathbbm R\times\Sigma,\ad(P)).
\end{eqnarray}
Let $H_A=d_A^{\ast}d_A+\ast[\ast F_A\wedge\,\cdot\,]+H_A\V$ and set $D_A\coloneqq\frac{d}{ds}+H_A$ and $D_{A}^{\ast}\coloneqq-\frac{d}{ds}+H_A$.

\begin{prop}\label{prop:denserange}
The image of the operator $\hat{\mathcal D}_{(A,\V)}\colon\Z^{\delta,p}\times Y^a\to\mathcal L^{\delta,p}$ is dense in $\mathcal L^{\delta,p}$, for every $(A,0,\V)\in\mathcal M^{\univ}(\mathcal C^-,\mathcal C^+)$.
\end{prop}

\begin{proof}
Density of the range is equivalent to triviality of its annihilator. This means that, given $\eta=(\eta_1,\eta_2)$ as in \eqref{eq:compnentseta} with
\begin{eqnarray}\label{eq:definitionannihilator}
\langle\hat{\mathcal D}_{A,\V}(\alpha,\psi,v),\eta\rangle_{\mathbbm R\times\Sigma}=0\qquad\textrm{for all}\quad(\alpha,\psi,v)\in\Z^{\delta,p}\times Y^a,
\end{eqnarray}
then $\eta=0$. Condition \eqref{eq:definitionannihilator} is equivalent to 
\begin{eqnarray}\label{eq:denserange}
\langle\mathcal D_A(\alpha,\psi),(\eta_1,\eta_2)\rangle_{\mathbbm R\times\Sigma}=0\qquad\textrm{and}\qquad\langle\nabla v(A),\eta_1\rangle_{\mathbbm R\times\Sigma}=0
\end{eqnarray}
for all $(\alpha,\psi,v)\in\Z^{\delta,p}\times Y^a$. Assume by contradiction that there exists $0\neq\eta\in\mathcal L^{-\delta,q}$ which satisfies both conditions in \eqref{eq:denserange}. Then choosing $\alpha=0$ the first of these conditions implies that $\langle\dot\psi,\eta_2\rangle_{\mathbbm R\times\Sigma}=0$ holds for all $\psi$ as before. Hence $\eta_2=0$. Choosing $\psi=0$ the first condition in \eqref{eq:denserange} reduces to $\langle D_A\alpha,\eta_1\rangle_{\mathbbm R\times\Sigma}=0$ for all $\alpha$, from which it follows that $\mathcal D_{A}^{\ast}\eta_1=0$. Hence Proposition \ref{prop:modelpert} below applies and yields the existence of a model perturbation $\V_0$ (not necessarily contained in the Banach space $Y^a$) such that $\langle\nabla\V_0(A),\eta_1\rangle_{\mathbbm R\times\Sigma}>0$. By construction of $\V_0$ from the data $A_0\in\A(P)$, $\eta_0\in\Omega^1(\Sigma,T^{\ast}\Sigma\otimes\ad(P))$, and $\eps>0$ there exists a close enough perturbation $v\in Y^a$ such that $\langle\nabla v(A),\eta_1\rangle_{\mathbbm R\times\Sigma}>0$. This contradicts the second equation in \eqref{eq:denserange}. Hence $\eta=0$, completing the proof of the proposition.
\end{proof}

It remains to prove Proposition \ref{prop:modelpert} below. For this we need the following auxiliary results. For the remainder of this section we rename $\eta_1\in L_{-\delta}^q(\mathbbm R\times\Sigma,T^{\ast}\Sigma\otimes\ad(P))$ to $\eta$.

\begin{prop}[Slicewise orthogonality]\label{slicewiseorth}
Fix $[(A,0,\V)]\in\mathcal M^{\univ}(\mathcal C^-,\mathcal C^+)$ and let $\eta$ satisfy $D_A^{\ast}\eta=0$ on $\R\times\Sigma$. Then for all $s\in\R$ there holds the relation $\langle\dot A(s),\eta(s)\rangle=0$.
\end{prop}

\begin{proof}
Set $\beta(s)\coloneqq\langle\dot A(s),\eta(s)\rangle$. Using the equation $\partial_sA+d_A^{\ast}F_A+\nabla\V(A)=0$ one easily checks that $\dot A$ satisfies $D_A\dot A=0$. Hence by symmetry of $H_A$ it follows that
\begin{eqnarray*}
\dot\beta=\langle\dot A,\dot\eta\rangle+\langle\ddot A,\eta\rangle=\langle\dot A,H_A\eta\rangle+\langle-H_A\dot A,\eta\rangle=0.
\end{eqnarray*}
Since $\lim_{s\to-\infty}\langle\dot A(s),\eta(s)\rangle=0$ we conclude that $\beta$ vanishes identically.
\end{proof}

In the following two propositions we let $A_0\in\A^{0,p}(P)$ be a (not necessarily irreducible) connection and denote by
\begin{eqnarray*}
\mathfrak m_{A_0,\eps}\colon\big(\S_{A_0}(\varepsilon)\times\G^{1,p}(P)\big)/\Stab{A_0}\to\A^{0,p}(P)
\end{eqnarray*}
the map introduced in Theorem \ref{locslicethm}. It is well-defined for every sufficiently small $0<\eps<\eps(A_0)$. We also recall the definition of the closed $L^2$ neighborhoods $U_{\mathcal C^{\pm}}$ around the critical manifolds $\hat{\mathcal C}^{\pm}$ at the end of Section \ref{sect:criticalmanifolds}.

\begin{prop}[No return]\label{noreturn}
Let $[(A,0,\V)]\in\mathcal M^{\univ}(\mathcal C^-,\mathcal C^+)$ and denote $A^{\pm}\coloneqq\lim_{s\to\pm\infty}A(s)\in\hat{\mathcal C}^{\pm}$. Then for every $s_0\in\R$ such that $A_0\coloneqq A(s_0)\notin U_{\mathcal C^-}\cup U_{\mathcal C^+}$, and every $\delta>0$ there is a constant $0<\eps<\eps(A_0)$ with
\begin{eqnarray*}
A(s)\in\im\mathfrak m_{A_0,\eps}\qquad\Longrightarrow\qquad s\in(s_0-\delta,s_0+\delta).
\end{eqnarray*}
\end{prop}

\begin{proof}
Let $\kappa>0$ be such that 
\begin{eqnarray}\label{eq:noreturn}
\dist_{L^2(\Sigma)}(A_0,\hat{\mathcal C}^-\cup\hat{\mathcal C}^+)>\kappa.
\end{eqnarray}
By the choice of $A_0$ and definition of the neighborhoods $U_{\mathcal C^{\pm}}$ the existence of such a constant $\kappa$ follows. Assume by contradiction that there exists a sequence $(\eps_i)$ of positive numbers with $\varepsilon_i\to0$ as $i\to\infty$, and a sequence $(s_i)\subseteq\R$ such that for all $i$
\begin{eqnarray}\label{eq:noreturn0}
A(s_i)\in\im\mathfrak m_{A_0,\eps_i}
\end{eqnarray}
but $s_i\notin(s_0-\delta,s_0+\delta)$. Hence by definition of the map $\mathfrak m(A_0,\eps_i)$ there exist gauge transformations $g_i\in\G^{1,p}(P)$ such that $\alpha(s_i)\coloneqq g_i^{\ast}A(s_i)-A_0$ satisfies
\begin{eqnarray}\label{eq:noreturn1}
\|\alpha(s_i)\|_{L^2(\Sigma)}<\eps_i
\end{eqnarray}
for all $i$. Assume first that the sequence $(s_i)$ is unbounded. Hence we can choose a subsequence (without changing notation) such that $s_i$ converges to $-\infty$ or to $+\infty$. It follows that (for one sign $+$ or $-$)
\begin{eqnarray*}
A(s_i)\stackrel{L^2(\Sigma)}{\longrightarrow}A^{\pm}\in\hat{\mathcal C}^{\pm}\qquad\textrm{as}\quad i\to\pm\infty.
\end{eqnarray*}
For $i$ sufficiently large such that $\eps_i<\kappa$ it follows that \eqref{eq:noreturn} and \eqref{eq:noreturn0} cannot be satisfied simultaneously. Otherwise \eqref{eq:noreturn1} would imply for some $A_1^{\pm}$ in the gauge orbit of $A^{\pm}$ the inequality $\|A_1^{\pm}-A_0\|_{L^2(\Sigma)}\leq\eps_i$, in contradiction to \eqref{eq:noreturn}. This contradiction shows that the sequence $(s_i)$ has an accumulation point $s_{\ast}\notin(s_0-\delta,s_0+\delta)$. So there exists a subsequence $(s_i)$ with $\lim_{i\to\infty}s_i=s_{\ast}$. Because the gradient flow line $s\mapsto A(s)$ is continuous as a map $\R\to L^p(\Sigma)$ it follows that $\lim_{i\to\infty}A(s_i)=A(s_{\ast})$ in $L^p(\Sigma)$. Assumption \eqref{eq:noreturn0} shows that $A(s_{\ast})\in\im\mathfrak m_{A_0,\eps_i}$ for all $i$ which is possible only if $A(s_{\ast})$ is contained in the gauge orbit of $A_0$. So $\YMV(A(s_{\ast}))=\YMV(A_0)$. As the map $s\mapsto\YMV(A(s))$ is strictly monotone decreasing it follows that $s_{\ast}=s_0$, which contradicts $s_{\ast}\notin(s_0-\delta,s_0+\delta)$. Hence the assumption was wrong and the claim follows.
\end{proof}

From now on we assume that the reference connection $A_0\in\A(P)$ is smooth and irreducible and define for $0<\eps<\eps(A_0)$ the map 
\begin{eqnarray*}
\alpha_{A_0,\eps}\colon\A^{0,p}(\Sigma)\to L^p(\Sigma,T^{\ast}(\Sigma)\otimes\ad(P))
\end{eqnarray*}
as in \eqref{def:alphai} such that $\supp\alpha_{A_0,\eps}\subseteq\im\mathfrak m_{A_0,\eps}$. The following remarks concerning $\alpha_{A_0,\eps}$ will be of importance for the next proposition. Let $s\mapsto A(s)$ and $s\mapsto\eta(s)$ be paths of connections, respectively of $\ad(P)$-valued $1$-forms as in Proposition \ref{prop:modelpert} below. For $s\in\R$ we denote
\begin{eqnarray*}
\alpha_{A_0,\eps}'(s)\coloneqq d\alpha_{A_0,\eps}(A(s))\eta(s),\qquad\dot\alpha_{A_0,\eps}(s)\coloneqq d\alpha_{A_0,\eps}(A(s))\dot A(s).
\end{eqnarray*}
Assume now that for some $s_0\in\R$ we have that $A_0=A(s_0)$. Then $\alpha_{A_0,\eps}'(s_0)=\eta(s_0)\eqqcolon\eta_0$ and $\dot\alpha_{A_0,\eps}(s_0)=\dot A(s_0)$ as follows from Proposition \ref{formuladifferential} together with the fact that $\alpha_{A_0,\eps}(s_0)=0$. Thus by continuous differentiability of the map $A\colon\R\to L^2(\Sigma)$ and continuity of the map $\eta\colon\R\to L^2(\Sigma)$ there exists a constant $\delta>0$ with the following significance. For all $s\in(s_0-\delta,s_0+\delta)$ we have that
\begin{compactenum}[(A)]
\item
$\|\eta(s)\|_{L^2(\Sigma)}\leq2\|\eta_0\|_{L^2(\Sigma)}$,
\item 
$\langle\alpha_{A_0,\eps}'(s),\eta_0\rangle\geq\frac{1}{2}\|\eta_0\|_{L^2(\Sigma)}^2>0$,
\item 
and with $\mu\coloneqq\|\dot A(s_0)\|_{L^2(\Sigma)}>0$ that
\begin{eqnarray*}
\frac{1}{2}\mu\leq\frac{\|\alpha_{A_0,\eps}(A(s))\|_{L^2(\Sigma)}}{|s-s_0|}\leq\frac{3}{2}\mu.
\end{eqnarray*}
\end{compactenum}
Similar continuity arguments show that properties (A-C) continue to hold if the irreducible connection $A_0$ is only assumed to be close in $L^p(\Sigma)$ to $A(s_0)$ (even if $A(s_0)$ itself is not irreducible).

\begin{prop}[Model perturbation]\label{prop:modelpert}
Let $A_0=A(s_0)$ be as in Proposition \ref{noreturn}. Assume $\eta$ satisfies $D_A^{\ast}\eta=0$ and $\eta_0\coloneqq\eta(s_0)\neq0$. Then there exists $0<\eps_0<\eps(A_0)$ and a gauge-invariant smooth map $\V_0\colon\A(P)\to\mathbbm R$ such that the following three properties are satisfied,
\begin{compactenum}[(i)]
\item
$\supp\V_0\subseteq\im\mathfrak m_{A_0,\eps_0}$,
\item
$\langle\nabla\V_0(A_0),\eta_0\rangle_{\Sigma}=\|\eta_0\|_{L^2(\Sigma)}^2$,
\item
$\langle\nabla\V_0(A),\eta\rangle_{\mathbbm R\times\Sigma}>0$.
\end{compactenum}
\end{prop}

\begin{proof} 
We construct the map $\V_0$ in a first step under the assumption that $A_0$ be irreducible. The general case is similar and will be treated afterwards. Let $\delta>0$ be such that the above conditions (A-C) are satisfied. Fix the constant $0<\eps_0<\eps(A_0)$ such that the conclusion of Proposition \ref{noreturn} applies. Now let $\rho\colon\mathbbm R\to[0,1]$ be the smooth cut-off function which was part of the data for the construction of the perturbations $\V_{\ell}$ in Section \ref{Bspaceperturbations}. For $0<\eps<\frac{\eps_0}{2}$ we define $\rho_{\eps}(r)\coloneqq\rho(\varepsilon^{-2}r)$. Note that $\|\rho_{\varepsilon}'\|_{L^{\infty}(\R)}<\varepsilon^{-2}$ and $\supp\rho_{\eps}\subseteq[4\eps^2,4\eps^2]$. Define
\begin{eqnarray*}
\V_0(A)\coloneqq\rho_{\varepsilon}(\|\alpha_{A_0,\eps_0}(A)\|_{L^2(\Sigma)}^2)\langle\alpha_{A_0,\eps_0}(A),\eta_0\rangle.
\end{eqnarray*}
The perturbation $\V_0$ satisfies condition (i) by construction. After choosing $\eps$ still smaller if necessary, Proposition \ref{prop:perturbsmooth} applies and shows smoothness of $\V_0$. Furthermore, it follows with $\alpha'(s_0)=\eta_0$ and by the fact that $\rho_{\eps}$ is constant in an open neighborhood around $0$ that $d\V_0(A_0)\eta_0=\|\eta_0\|_{L^2(\Sigma)}^2$, so that condition (ii) is satisfied. It remains to show property (iii). We fix constants $\sigma_1,\sigma_2,s_1,s_2$ with $\sigma_1<s_1<s_0<s_2<\sigma_2$ as follows. Let $s_2$ be such that $\|\alpha_{A_0,\eps_0}(A(s_2))\|_{L^2(\Sigma)}=\eps$ and $\|\alpha_{A_0,\eps_0}(A(s))\|_{L^2(\Sigma)}<\eps$ for all $s\in(s_0,s_2)$, and analogously for $s_1$. Let $\sigma_2$ be such that $\|\alpha_{A_0,\eps_0}(A(\sigma_2))\|_{L^2(\Sigma)}=2\eps$ and for all $s>\sigma_2$ the following holds. Either $A(s)\notin\im\mathfrak m_{A_0,\eps_0}$ or otherwise $\|\alpha_{A_0,\eps_0}(A(s))\|_{L^2(\Sigma)}>2\eps$. The number $\sigma_1$ is defined analogously. In the following calculations we often drop the arguments $A$ in $\alpha_{A_0,\eps_0}(A)$ and $s$ in $\alpha_{A_0,\eps_0}'(s)$ and $\dot\alpha_{A_0,\eps_0}(s)$, for brevity. It follows from (i) and Proposition \ref{noreturn} that
\begin{multline}\label{eq:innerprodgradeta}
\langle\nabla\V_0(A),\eta\rangle_{\mathbbm R\times\Sigma}\\
=\int_{s_0-\delta}^{s_0+\delta}d\V_0(A(s))\eta(s)\,ds=\int_{s_0-\delta}^{s_0+\delta}\rho_{\varepsilon}(\|\alpha_{A_0,\eps_0}\|_{L^2(\Sigma)}^2)\langle\alpha_{A_0,\eps_0}',\eta_0\rangle\,ds\\
+2\int_{s_0-\delta}^{s_0+\delta}\rho_{\varepsilon}'(\|\alpha_{A_0,\eps_0}\|_{L^2(\Sigma)}^2)\langle\alpha_{A_0,\eps_0},\alpha_{A_0,\eps_0}'\rangle\langle\alpha_{A_0,\eps_0},\eta_0\rangle\,ds.
\end{multline}
We estimate the last two terms separately. For the first one we obtain 
\begin{eqnarray*}
\lefteqn{\int_{s_0-\delta}^{s_0+\delta}\rho_{\varepsilon}(\|\alpha_{A_0,\eps_0}\|_{L^2(\Sigma)}^2)\langle\alpha_{A_0,\eps_0}',\eta_0\rangle\,ds}\\
&\geq&\int_{s_1}^{s_2}1\cdot\langle\alpha_{A_0,\eps_0}',\eta_0\rangle\,ds\\
&\geq&\frac{1}{2}(s_2-s_1)\|\eta_0\|_{L^2(\Sigma)}^2\\
&\geq&\frac{1}{3\mu}\big(\|\alpha_{A_0,\eps_0}(A(s_1))\|_{L^2(\Sigma)}+\|\alpha_{A_0,\eps_0}(A(s_2))\|_{L^2(\Sigma)}\big)\|\eta_0\|_{L^2(\Sigma)}^2\\
&=&\frac{2}{3\mu}\|\eta_0\|_{L^2(\Sigma)}^2\varepsilon.
\end{eqnarray*}
The first and second inequality follow from property (B), while the third one is by property (C). We define functions $f,g\colon\R\to\R$ by
\begin{eqnarray*}
f(s)\coloneqq\langle\alpha_{A_0,\eps_0}(A(s)),\alpha_{A_0,\eps_0}'(s)\rangle\qquad\textrm{and}\qquad g(s)\coloneqq\langle\alpha_{A_0,\eps_0}(A(s)),\eta_0\rangle.
\end{eqnarray*}
As $\alpha_{A_0,\eps_0}(s_0)=0$ it follows that $f(s_0)=g(s_0)=0$. By Proposition \ref{formuladifferential} we have that $\dot\alpha_{A_0,\eps_0}(s_0)=\dot A(s_0)$ and $\alpha_{A_0,\eps_0}'(s_0)=\eta_0$. Using Proposition \ref{slicewiseorth} it follows that
\begin{eqnarray*}
\dot f(s_0)=\langle\dot\alpha_{A_0,\eps_0}(s_0),\alpha_{A_0,\eps_0}'(s_0)\rangle+\langle\alpha_{A_0,\eps_0}(A(s_0)),\partial_s\alpha_{A_0,\eps_0}'(s_0)\rangle=\langle\dot A(s_0),\eta_0\rangle=0,
\end{eqnarray*}
and similarly that
\begin{eqnarray*}
\dot g(s_0)=\langle\dot\alpha_{A_0,\eps_0}(s_0),\eta_0\rangle=\langle\dot A(s_0),\eta_0\rangle=0.
\end{eqnarray*}
Hence there exists a constant $C=C(A_0,\eta_0)$ such that for all $s\in(s_0-\delta,s_0+\delta)$
\begin{eqnarray*}
|f(s)|\leq C(s-s_0)^2\qquad\textrm{and}\qquad|g(s)|\leq C(s-s_0)^2.
\end{eqnarray*}
The second term in \eqref{eq:innerprodgradeta} is now estimated as follows.
\begin{eqnarray*}
\lefteqn{2\int_{s_0-\delta}^{s_0+\delta}\rho_{\varepsilon}'(\|\alpha_{A_0,\eps_0}\|_{L^2(\Sigma)}^2)\langle\alpha_{A_0,\eps_0},\alpha_{A_0,\eps_0}'\rangle\langle\alpha_{A_0,\eps_0},\eta_0\rangle\,ds}\\
&=&2\int_{\sigma_1}^{\sigma_2}\rho_{\varepsilon}'(\|\alpha_{A_0,\eps_0}\|_{L^2(\Sigma)}^2)\langle\alpha_{A_0,\eps_0},\alpha_{A_0,\eps_0}'\rangle\langle\alpha_{A_0,\eps_0},\eta_0\rangle\,ds\\
&\geq&-2\int_{\sigma_1}^{\sigma_2}\|\rho_{\eps}'\|_{L^{\infty}(\R)}|\langle\alpha_{A_0,\eps_0},\alpha_{A_0,\eps_0}'\rangle|\cdot|\langle\alpha_{A_0,\eps_0},\eta_0\rangle|\,ds\\
&\geq&-2\varepsilon^{-2}C^2\int_{\sigma_1}^{\sigma_2}(s-s_0)^4\,ds\\
&=&-\frac{2}{5}\varepsilon^{-2}C^2\big(|\sigma_1-s_0|^5+|\sigma_2-s_0|^5\big)\\
&\geq&-\frac{2}{5}\varepsilon^{-2}C^2\big(\frac{2}{\mu}\big)^5\big(\|\alpha_{A_0,\eps_0}(A(\sigma_1))\|_{L^2(\Sigma)}^5+\|\alpha_{A_0,\eps_0}(A(\sigma_2))\|_{L^2(\Sigma)}^5\big)\\
&=&-\frac{128}{5\mu^5}C^2\varepsilon^3.
\end{eqnarray*}
The last inequality follows from property (C). Combining these estimates we find that
\begin{eqnarray}\label{eq:inprodpositive}
\langle\nabla\V_0(A),\eta\rangle_{\mathbbm R\times\Sigma}\geq\frac{2}{3\mu}\|\eta_0\|_{L^2(\Sigma)}^2\varepsilon-\frac{128}{5\mu^5}C^2\varepsilon^3
\end{eqnarray}
Choosing $\eps>0$ still smaller if necessary (which does not affect the argumentation so far), the last expression becomes strictly positive. This shows property (iii) and completes the proof in the case $A_0$ irreducible. In the case where $A_0$ is reducible the map $\alpha_{A_0,\eps_0}$ is not well-defined and we have to modify our argumentation slightly. Because the set of irreducible smooth connections is dense in $\A^{0,p}(\Sigma)$ we can choose for every $\eps>0$ an irreducible connection $A(\eps)$ with $\|A(\eps)-A_0\|_{L^p(\Sigma)}<\eps^{10}$. Then we replace in the definition of $\V_0$ the map $\alpha_{A_0,\eps_0}$ by $\alpha_{A(\eps),\eps_0}$. For $\eps>0$ sufficiently small properties (A-C) are still satisfied with $\alpha_{A(\eps),\eps_0}$ in place of $\alpha_{A_0,\eps_0}$. Hence we can repeat the previous estimates of the terms on the right-hand side of \eqref{eq:innerprodgradeta}. This introduces some further expressions depending on $A(\eps)-A_0$ which can be estimated against a sufficiently high power of $\eps$. The conclusion that the right-hand side in \eqref{eq:inprodpositive} is positive for sufficiently small $\eps>0$ then follows as before. 
\end{proof}

\section{Yang--Mills Morse homology}\label{YMMorsehomology}
\subsection{Morse--Bott theory}

We briefly recall Frauenfelder's \emph{cascade construction} of Morse homology for Morse functions with degenerate critical points satisfying the Morse--Bott condition (cf.~\cite[Appendix C]{Frauenfelder1}). Let $(M,g)$ be a Riemannian (Banach) manifold. A smooth function $f\colon M\to\mathbbm R$ is called {\emph{Morse--Bott}} if the set $\crit(f)\subseteq M$ of its critical points is a finite-dimensional submanifold of $M$, and if for each $x\in\crit(f)$ the {\emph{Morse--Bott condition}} $T_x\crit(f)=\ker\Hess_xf$ is satisfied. As an additional datum, we fix a Morse function $h\colon\crit(f)\to\mathbbm R$ which satisfies the Morse--Smale condition, i.e.~the stable and unstable manifolds $W_h^s(x)$ and $W_h^u(y)$ of any two critical points $x,y\in\crit(h)$ intersect transversally. We assign to a critical point $x\in\crit(h)\subseteq\crit(f)$ the index
\begin{eqnarray*} 
\Ind(x)\coloneqq\textrm{ind}_f(x)+\textrm{ind}_h(x).
\end{eqnarray*}
\begin{defn}\upshape
Let $x^-,x^+\in\crit(h)$ and $m\in\mathbbm N$. A {\emph{flow line from}} $x^-$ {\emph{to}} $x^+$ {\emph{with}} $m$ {\emph{cascades}} is a tuple $(\mathtt{x},T)\coloneqq(x_1,\ldots,x_m,t_1,\ldots,t_{m-1})$ with $x_j\in C^{\infty}(\mathbbm R,M)$ and $t_j\in\mathbbm R^+$ such that the following conditions are satisfied.
\begin{compactenum}[(i)]
\item
Each $x_j$ is a nonconstant solution of the gradient flow equation $\partial_sx_j+\nabla f(x_j)=0$.
\item
For each $1\leq j\leq m-1$ there exists a solution $y_j\in C^{\infty}(\mathbbm R,\crit(f))$ of the gradient flow equation $\partial_sy_j+\nabla h(y_j)=0$ such that $\lim_{s\to\infty}x_j(s)=y_j(0)$ and $\lim_{s\to-\infty}x_{j+1}(s)=y_j(t_j)$.
\item
There exist points $p^-\in W_h^u(x^-)\subseteq\crit(f)$ and $p^+\in W_h^s(x^+)\subseteq\crit(f)$ such that $\lim_{s\to-\infty}x_1(s)=p^-$ and $\lim_{s\to\infty}x_m(s)=p^+$.
\end{compactenum}
A {\emph{flow line with $m=0$ cascades}} is an ordinary Morse flow line of $h$ on $\crit(f)$ from $x^-$ to $x^+$.
\end{defn}
Denote by $\mathcal M_m(x^-,x^+)$ the set of flow lines from $x^-$ to $x^+$ with $m\geq0$ cascades (modulo the action of the group $\mathbbm R^m$ by time-shifts on tuples $(x_1,\ldots,x_m)$). We call
\begin{eqnarray}\label{modspacecascades}
\mathcal M(x^-,x^+)\coloneqq\bigcup_{m\in\mathbbm N_0}\mathcal M_m(x^-,x^+)
\end{eqnarray}
the set of {\emph{flow lines with cascades from}} $x^-$ {\emph{to}} $x^+$. In analogy to usual Morse theory (where the Morse function is required to have only isolated non-degenerate critical points), a sequence of flow lines with cascades may converge to a limit configuration which is a connected chain of such flow lines with cascades. This limiting behaviour is captured in the following definition. 

\begin{defn}\upshape
Let $x^-,x^+\in\crit(h)$. A {\emph{broken flow line with cascades from $x^-$ to $x^+$}} is a tuple $\mathtt{v}=(v_1,\ldots,v_{\ell})$ where each $v_j$, $j=1,\ldots,\ell$, consists of a flow line with cascades from $x^{(j-1)}$ to $x^{(j)}\in\crit(h)$ such that $x^{(0)}=x^-$ and $x^{(\ell)}=x^+$.
\end{defn}

We let $CM_{\ast}(M,f,h)$ denote the chain complex generated (as a $\mathbbm Z_2$ vector space) by the critical points of $h$ and graded by the index $\Ind$. On generators of $CM_{\ast}(M,f,h)$ we set
\begin{eqnarray*}
\partial_kx\coloneqq\sum_{\Ind(x')=k-1}n(x,x')x'.
\end{eqnarray*}
Here $n(x,x')\in\mathbbm Z_2$ denotes the number, counted modulo $2$, of elements in the zero dimensional moduli space $\mathcal M(x,x')$. By linear extension to $CM_{\ast}(M,f,h)$ this formally defines a boundary operator $\partial_k\colon CM_{k}(M,f,h)\to CM_{k-1}(M,f,h)$. The aim of the next section is to show that Morse--Bott theory with cascades applies to the Yang--Mills functional on the Banach manifold $\A^{1,p}(P)/\G_0^{2,p}(P)$ and gives rise to a well-defined boundary operator and Morse--Bott homology groups
\begin{eqnarray*}
HM_k\left(\frac{\A(P)}{\G_0(P)},\YMV,h\right)\coloneqq\frac{\ker\partial_{k}^{\V}}{\im\partial_{k+1}^{\V}}\qquad(k\in\mathbbm N_0).
\end{eqnarray*}

\subsection{Yang--Mills Morse complex}\label{sect:YMMorsecomplex}

We fix a regular value $a\geq0$ of $\YM$ and an $a$-admissible perturbation $\V\in Y$ with $\|\V\|<\delta$ sufficiently small such that Proposition \ref{prop:samecritpoints} applies. Let
\begin{eqnarray*}
\p(a)\coloneqq\frac{\{A\in\A^{1,p}(P)\mid d_A^{\ast}F_A=0\;\textrm{and}\;\YM(A)\leq a\}}{\G_0^{2,p}(P)}
\end{eqnarray*}
denote the set of based gauge equivalence classes of (weak) Yang--Mills connections of energy at most $a$. Let $h\colon\p(a)\to\R$ be a smooth Morse function. We let
\begin{eqnarray*}
CM_{\ast}^a(\A(P)/\G_0(P),\YM,h)
\end{eqnarray*}
denote the complex generated as a $\mathbbm Z_2$ vector space by the set $\crit(h)\subseteq\p(a)$ of critical points of $h$. For $x^-,x^+\in\crit(h)$ we call the set $\mathcal M(x^-,x^+)$ as in \eqref{modspacecascades} the moduli space of Yang--Mills gradient flow lines with cascades from $x^-$ to $x^+$.

\begin{lem}\label{lem:cascadmodspace}
For generic, $a$-admissible perturbation $\V\in Y$ as above, Morse function $h\colon\p(a)\to\R$, and all $x^-,x^+\in\crit(h)$, the set $\mathcal M(x^-,x^+)$ is a smooth manifold (with boundary) of dimension
\begin{eqnarray*}
\dim\mathcal M(x^-,x^+)=\Ind(x^-)-\Ind(x^+)-1.
\end{eqnarray*}
\end{lem}

\begin{proof}
The proof that $\mathcal M(x^-,x^+)$ is a smooth manifold for generic perturbations $\V$ follows the standard routine by writing $\mathcal M(x^-,x^+)$ as the zero set of a Fredholm section $\tilde\F$ of a suitable Banach space bundle, and then applying the implicit function theorem. This Fredholm problem can be reduced to the one studied in Section \ref{sec:Fredholm theorem}. Namely, as shown in \cite[Lemma C.12]{Frauenfelder1}, the set $\mathcal M(x^-,x^+)$ arises as a submanifold of the product of moduli spaces $\mathcal M(\mathcal C_i^-,\mathcal C_j^+)$ for suitable pairs $\mathcal C_i^-,\mathcal C_j^+\in\mathcal{CR}$ of critical manifolds. This way, the Fredholm property of the linearized section $d\tilde\F$ and the formula for its index follow from Theorem \ref{thm:Fredholmtheorem}. Similarly, transversality of $\tilde\F$ is a consequence of Theorem \ref{thm:linsurjetivederef} and does not require any new arguments.
\end{proof}

For $k\in\mathbbm N$ we define the Morse boundary operator
\begin{eqnarray*}
\partial_k^{\V}\colon CM_k^a(\A(P)/\G_0(P),\YM,h)\to CM_{k-1}^a(\A(P)/\G_0(P),\YM,h)
\end{eqnarray*}
to be the linear extension of the map
\begin{eqnarray*}
\partial_k^{\V}x\coloneqq\sum_{x'\in\crit(h)\atop\Ind(x')=k-1}n(x,x')x',
\end{eqnarray*}
where $x\in\crit(h)\subseteq\p(a)$ is a critical point of index $\Ind(x)=k$. The numbers $n(x,x')$ are given by counting modulo $2$ the oriented flow lines with cascades (with respect to $\YMV$ and $h$) from $x$ to $x'$, i.e.~
\begin{eqnarray*}
n(x,x')\coloneqq\#\mathcal M(x^-,x^+)\qquad(\mod 2).
\end{eqnarray*}

\begin{lem}\label{lem:cascadechainmap}
The numbers $n(x,x')$ are well-defined and $\partial_{\ast}^{\V}$ has the property of a chain map, i.e., it satisfies $\partial_{\ast}^{\V}\circ\partial_{\ast+1}^{\V}=0$.
\end{lem}

\begin{proof}
It follows from Theorem \ref{thm:compactness} by repeating the arguments of \cite[Theorems C.10]{Frauenfelder1} that for every $x^-,x^+\in\crit(h)$ the moduli space $\mathcal M(x^-,x^+)$ is compact up to convergence to broken flow lines with cascades. For $\Ind(x')=\Ind(x)-1$ this means that it is a finite set and $n(x,x')$ is well-defined. The chain map property follows from standard arguments making use of Theorem \ref{YMExponentialdecay} on exponential decay of Yang--Mills gradient flow lines.
\end{proof}

With these preparations, we can finally proof the main result.

\begin{proof}{\bf{(Theorem \ref{thm:mainresult})}}
That the pair $\big(CM_{\ast}^a(\A(P)/\G_0(P),\YM,h),\partial_{\ast}^{\V}\big)$ is a chain complex for generic, $a$-admissible perturbation $\V\in Y$ as above and Morse function $h\colon\p(a)\to\R$ follows from Lemmata \ref{lem:cascadmodspace} and \ref{lem:cascadechainmap}. Hence Yang--Mills Morse homology exists and is well-defined. From homotopy arguments standardly used in Morse theory theory (cf.~\cite{Floer1,Sal2,Schw}) it follows that these homology groups do not depend on the choice of the perturbation $\V$ and the Morse function $h$.  
\end{proof}

\appendix

\section{Perturbations}\label{sect:Properties of the perturbations}

Subsequently we list some relevant properties of the perturbations $\V_{\ell}$ introduced in Section \ref{Bspaceperturbations}. Let $A\in\A(P)$ be an irreducible connection and $\alpha\in L^{\infty}(\Sigma,T^{\ast}\Sigma\otimes\ad(P))$ be a $1$-form such that $d_A^{\ast}\alpha=0$. We introduce the following operators.
\begin{eqnarray*}
&&L_{A,\alpha}\colon\Omega^0(\Sigma,\ad(P))\to\Omega^0(\Sigma,\ad(P)),\quad\lambda\mapsto\Delta_A\lambda+\ast[\ast\alpha\wedge d_A\lambda],\\
&&R_{A,\alpha}\coloneqq L_{A,\alpha}^{-1}\colon\Omega^0(\Sigma,\ad(P))\to\Omega^0(\Sigma,\ad(P)),\\
&&M_{\alpha}\colon\Omega^1(\Sigma,\ad(P))\to\Omega^0(\Sigma,\ad(P)),\quad\xi\mapsto\ast[\alpha\wedge\ast\xi],\\
&&T_{A,\alpha}\coloneqq R_{A,\alpha}\circ M_{\alpha}\colon\Omega^1(\Sigma,\ad(P))\to\Omega^0(\Sigma,\ad(P)).
\end{eqnarray*}

\begin{prop}\label{prop:propertiesL}
The operator 
\begin{eqnarray*}
L_{A,\alpha}\colon L^2(\Sigma,\ad(P))\to L^2(\Sigma,\ad(P))
\end{eqnarray*}
is a densely defined self-adjoint operator with domain $W^{2,2}(\Sigma,\ad(P))$. Furthermore, there exists a constant $c(A)$ such that for every $\alpha$ with $\|\alpha\|_{L^2(\Sigma)}<c(A)$ the inverse $R_{A,\alpha}$ exists as a bounded operator 
\begin{eqnarray*}
R_{A,\alpha}\colon L^2(\Sigma,\ad(P))\to W^{2,2}(\Sigma,\ad(P)).
\end{eqnarray*}
\end{prop}

\begin{proof}
For every $\lambda,\mu\in L^2(\Sigma,\ad(P))$ it follows that
\begin{eqnarray*}
\langle\ast[\ast\alpha\wedge d_A\lambda],\mu\rangle=\langle d_A\lambda,[\alpha\wedge\mu]\rangle=\langle\lambda,-\ast d_A[\ast\alpha\wedge\mu]\rangle=\langle\lambda,\ast[\ast\alpha\wedge d_A\mu]\rangle,
\end{eqnarray*}
using $d_A^{\ast}\alpha=0$ in the last step. This implies symmetry of the operator $L_{A,\alpha}$. As the Laplace operator $\Delta_A$ is self-adjoint with domain $W^{2,2}(\Sigma,\ad(P))$ the same holds true for $L_{A,\alpha}$ by the Kato--Rellich theorem (cf.~\cite{ReSi}) because the perturbation $\ast[\ast\alpha\wedge d_A\lambda]$ is of relative bound zero. Assuming bijectivity of $L_{A,\alpha}$, boundedness of the operator 
\begin{eqnarray*}
R_{A,\alpha}\colon L^2(\Sigma,\ad(P))\to W^{2,2}(\Sigma,\ad(P))
\end{eqnarray*}
follows from elliptic regularity. It remains to show that $L_{A,\alpha}$ is bijective. For this we first consider for $p>2$ the bounded operator 
\begin{eqnarray*}
L_{A,\alpha}^p\colon W^{1,p}(\Sigma,\ad(P))\to W^{-1,p}(\Sigma,\ad(P)),\quad\lambda\mapsto L_{A,\alpha}\lambda.
\end{eqnarray*}
The assumption that $A$ is irreducible implies that $\Delta_A=L_{A,0}^p$ is injective and therefore (by symmetry) bijective. Bijectivity is preserved under small perturbations with respect to the operator norm, and thus $L_{A,\alpha}^p$ is bijective for $\|\alpha\|_{L^2(\Sigma)}\leq c(A)$ sufficiently small because
\begin{eqnarray*}
\|\ast[\ast\alpha\wedge d_A\lambda]\|_{W^{-1,p}(\Sigma)}\leq c\|\alpha\|_{L^2(\Sigma)}\|d_A\lambda\|_{L^p(\Sigma)}\leq c\|\alpha\|_{L^2(\Sigma)}\|\lambda\|_{W^{1,p}(\Sigma)}.
\end{eqnarray*}
The first estimate follows from Proposition \ref{prop:prodestimLp}. Now let $\lambda\in\ker L_{A,\alpha}$. Then the Sobolev embedding $W^{2,2}(\Sigma)\hookrightarrow W^{1,p}(\Sigma)$ for any $p<\infty$ implies that $\lambda\in\ker L_{A,\alpha}^p$, and therefore $\lambda=0$. Hence $L_{A,\alpha}$ is injective and (by self-adjointness) bijective. This completes the proof of the proposition. 
\end{proof}

For $\ell\in\mathbbm N$ let $\V_{\ell}\in Y$ be the perturbation defined in \eqref{modelperturbation}. As $\ell$ is kept fix throughout the subsequent propositions, we temporarily denote $\mathfrak m\coloneqq\mathfrak m_i$, $A_0\coloneqq A_i$,  $\eta\coloneqq\eta_{ij}$, and $\rho\coloneqq\rho_k$. For given $A\in\A^{0,p}(P)$ we also denote at several places $\alpha\coloneqq\alpha_i(A)$.

\begin{prop}\label{formuladifferential}
The map $\V\coloneqq\V_{\ell}\colon\A^{0,p}(P)\to\mathbbm R$ has the following properties. Let $A\in\im\mathfrak m$ and denote by $g\in\G^{1,p}(P)$ the unique gauge transformation such that \eqref{condlocslicepert} is satisfied. 
\begin{compactenum}[(i)]
\item
The differential and $L^2$ gradient of $\V$ at $A$ are given by
\begin{align*}
d\V(A)\xi=2\rho'(\|\alpha\|_{L^2(\Sigma)}^2)\langle\alpha,\hat\xi\rangle\langle\alpha,\eta\rangle+\rho(\|\alpha\|_{L^2(\Sigma)}^2)\langle\hat\xi-d_{g^{\ast}A}T_{A_0,\alpha}\hat\xi,\eta\rangle,\\
g^{-1}\nabla\V(A)g=2\rho'(\|\alpha\|_{L^2(\Sigma)}^2)\langle\alpha,\eta\rangle\alpha+\rho(\|\alpha\|_{L^2(\Sigma)}^2)(\eta+T_{A_0,\alpha}^{\ast}(\ast[\alpha\wedge\ast\eta])),
\end{align*}
with $\hat\xi\coloneqq g^{-1}\xi g$. Here we assume that $\xi\in\Omega^1(\Sigma,\ad(P))$ satisfies $d_A^{\ast}\xi=0$.
\item
Let $\beta\in\Omega^1(\Sigma,\ad(P))$ such that $d_{A}^{\ast}\beta=0$ and set $\gamma\coloneqq\hat\beta-d_{g^{\ast}A}T_{A_0,\alpha}\hat\beta$ with $\hat\beta\coloneqq g^{-1}\beta g$. The Hessian of $\V$ at $A$ is the map 
\begin{eqnarray*}
H_A\V\colon\Omega^1(\Sigma,\ad(P))\to\Omega^1(\Sigma,\ad(P))
\end{eqnarray*}
determined by the formula
\begin{eqnarray*}
\lefteqn{g^{-1}(H_A\V\beta)g-[g^{-1}\nabla\V(A)g,T_{A_0,\alpha}(g^{-1}\beta g)]}\\
&=&\rho(\|\alpha\|_{L^2(\Sigma)}^2)\left(S_{A_0,\alpha,\gamma}^{\ast}(\ast[\alpha\wedge\ast\eta])+T_{A_0,\alpha}^{\ast}(\ast[\gamma\wedge\ast\eta])\right)\\
&&+2\rho'(\|\alpha\|_{L^2(\Sigma)}^2)\left(\langle\alpha,\gamma\rangle T_{A_0,\alpha}^{\ast}(\ast[\alpha\wedge\ast\eta])+\langle\alpha,\gamma\rangle\eta+\langle\eta,\gamma\rangle\alpha+\langle\alpha,\eta\rangle\gamma\right)\\
&&+4\rho''(\|\alpha\|_{L^2(\Sigma)}^2)\langle\alpha,\gamma\rangle\langle\alpha,\eta\rangle\alpha.
\end{eqnarray*}
Here we denote 
\begin{eqnarray*}
S_{A_0,\alpha,\gamma}\coloneqq R_{A_0,\alpha}\circ M_{\gamma}\circ({\mathbbm 1}-d_{A_0}\circ R_{A_0,\alpha})\colon\Omega^1(\Sigma,\ad(P))\to\Omega^0(\Sigma,\ad(P)).
\end{eqnarray*}
\end{compactenum}
\end{prop}

\begin{proof}
\begin{compactenum}[(i)]
\item
Let $A(t)=A+t\xi$. Assume $A(t)$, $\alpha(t)=\alpha(A(t))$ and $g(t)$ satisfy condition \eqref{condlocslicepert} for all $t\in(-\varepsilon,\varepsilon)$ with $\varepsilon>0$ sufficiently small. We denote $g\coloneqq g(0)$ and set $\dot\alpha\coloneqq\left.\frac{d}{dt}\right|_{t=0}\alpha(t)$ and $\lambda\coloneqq g^{-1}\left.\frac{d}{dt}\right|_{t=0}g(t)$. Differentiating the equation $d_{A_0}^{\ast}(g^{\ast}A-A_0)=0$ at $t=0$ and using that $d_A^{\ast}\xi=0$ we obtain
\begin{eqnarray*}
0&=&d_{A_0}^{\ast}(g^{-1}\xi g+d_{g^{\ast}A}\lambda)\\
&=&g^{-1}(d_{(g^{-1})^{\ast}A_0}^{\ast}\xi)g+d_{A_0}^{\ast}d_{A_0}\lambda+d_{A_0}^{\ast}[\alpha\wedge\lambda]\\
&=&g^{-1}(d_{A-g\alpha g^{-1}}^{\ast}\xi)g+\Delta_{A_0}\lambda+d_{A_0}^{\ast}[\alpha\wedge\lambda]\\
&=&g^{-1}\ast[g\alpha g^{-1}\wedge\ast\xi]g+\Delta_{A_0}\lambda+d_{A_0}^{\ast}[\alpha\wedge\lambda]\\
&=&\ast[\alpha\wedge\ast g^{-1}\xi g]+\Delta_{A_0}\lambda+d_{A_0}^{\ast}[\alpha\wedge\lambda]\\
&=&M_{\alpha}\hat\xi+L_{A_0,\alpha}\lambda.
\end{eqnarray*}
Hence $\lambda=-T_{A_0,\alpha}\hat\xi$ by definition of $T_{A_0,\alpha}$, and
\begin{eqnarray}\label{eq:differentialalpha}
\dot\alpha=\left.\frac{d}{dt}\right|_{t=0}\big(g^{\ast}(t)A(t)-A_0\big)=\hat\xi-d_{g^{\ast}A}T_{A_0,\alpha}\hat\xi.
\end{eqnarray}
From this we obtain
\begin{eqnarray*}
\lefteqn{d\V(A)\xi=\left.\frac{d}{dt}\right|_{t=0}\rho(\|\alpha(t)\|_{L^2(\Sigma)}^2)\langle\alpha(t),\eta\rangle}\\
&=&2\rho'(\|\alpha\|_{L^2(\Sigma)}^2)\langle\alpha,\dot\alpha\rangle\langle\alpha,\eta\rangle+\rho(\|\alpha\|_{L^2(\Sigma)}^2)\langle\dot\alpha,\eta\rangle\\
&=&2\rho'(\|\alpha\|_{L^2(\Sigma)}^2)\langle\alpha,\hat\xi-d_{g^{\ast}A}T_{A_0,\alpha}\hat\xi\rangle\langle\alpha,\eta\rangle+\rho(\|\alpha\|_{L^2(\Sigma)}^2)\langle\hat\xi-d_{g^{\ast}A}T_{A_0,\alpha}\hat\xi,\eta\rangle\\
&=&2\rho'(\|\alpha\|_{L^2(\Sigma)}^2)\langle\alpha,\hat\xi\rangle\langle\alpha,\eta\rangle+\rho(\|\alpha\|_{L^2(\Sigma)}^2)\langle\hat\xi-d_{g^{\ast}A}T_{A_0,\alpha}\hat\xi,\eta\rangle.
\end{eqnarray*}
In the last line we used that $d_{g^{\ast}A}^{\ast}\alpha=0$. The formula for $\nabla\V(A)$ follows from this by taking adjoints and using that $d_{g^{\ast}A}^{\ast}\eta=d_{A_0}^{\ast}\eta-\ast[\alpha\wedge\ast\eta]=-\ast[\alpha\wedge\ast\eta]$.
\item 
The formula follows from differentiating the expression for $g^{-1}\nabla\V(A)g$ in (i) and formula \eqref{eq:differentialalpha}. The operator $S_{A_0,\alpha,\gamma}$ arises from differentiating
\begin{multline*}
\left.\frac{d}{dt}\right|_{t=0}T_{A_0,\alpha}=\left.\frac{d}{dt}\right|_{t=0}L_{A_0,\alpha}^{-1}\circ M_{\alpha}=-L_{A_0,\alpha}^{-1}\dot L_{A_0,\alpha}L_{A_0,\alpha}^{-1} M_{\alpha}+R_{A_0,\alpha}\dot M_{\alpha}\\
=-R_{A_0,\alpha}\dot L_{A_0,\alpha}R_{A_0,\alpha}+R_{A_0,\alpha} M_{\gamma}=R_{A_0,\alpha} M_{\gamma}(\mathbbm 1-d_{A_0}\circ R_{A_0,\alpha}),
\end{multline*}
using that $\dot L_{A_0,\alpha}=\ast[\ast\gamma\wedge d_{A_0}\,\cdot\,]=-M_{\gamma}\circ d_{A_0}$ and $\dot M_{\alpha}=M_{\gamma}$.
\end{compactenum}
\end{proof}

\begin{prop}\label{prop:prodestimLp}
Let $p>2$. There exists a constant $c(p)$ such the estimate
\begin{eqnarray*} 
\|uv\|_{W^{-1,p}(\Sigma)}\leq c(p)\|u\|_{L^2(\Sigma)}\|v\|_{L^p(\Sigma)}
\end{eqnarray*}
is satisfied for all $u\in L^2(\Sigma)$, $v\in L^p(\Sigma)$.
\end{prop}

\begin{proof}
Let $q<2$ denote the dual Sobolev exponent of $p$. Let $r\coloneqq\frac{2p}{2+p}<2$ and $s\coloneqq\frac{2q}{2-q}>2$, i.e.~$\frac{1}{r}+\frac{1}{s}=1$. Then the Sobolev embedding $W^{1,q}(\Sigma)\hookrightarrow L^s(\Sigma)$ implies the dual embedding $L^r(\Sigma)\hookrightarrow W^{-1,p}(\Sigma)$. Hence for some constant $c(p)$ it follows that
\begin{eqnarray*}
\|uv\|_{W^{-1,p}(\Sigma)}\leq c(p)\|uv\|_{L^r(\Sigma)},
\end{eqnarray*}
and H\"older's inequality (with exponents $\ell=\frac{2}{r}>1$ and $\ell'=\frac{2}{2-r}>1$) then implies that
\begin{eqnarray*}
\|uv\|_{L^r(\Sigma)}^r\leq\big(\int_{\Sigma}|u|^2\big)^{\frac{r}{2}}\big(\int_{\Sigma}|v|^{r\ell'}\big)^{\frac{1}{\ell'}}=\|u\|_{L^2(\Sigma)}^r\|v\|_{L^p(\Sigma)}^r,
\end{eqnarray*}
as claimed.
\end{proof}
 
We continue to use the notational conventions fixed before Proposition \ref{formuladifferential}.
 
\begin{prop}\label{prop:L2estgrad}
Let $A_0\in\A(P)$ and $p>2$. There exist constants $c(A_0)$, $c(A_0,p)$ and $\delta(A_0,p)$ such that the estimates 
\begin{align}
\label{eq:L2estgrad1}\|\alpha(A)\|_{W^{1,p}(\Sigma)}\leq c(A_0,p)\big(1+\|F_A\|_{L^p(\Sigma)}\big),\\
\label{eq:L2estgrad1A}\|\nabla\V(A)\|_{C^0(\Sigma)}\leq c(A_0)\big(1+\|F_A\|_{L^3(\Sigma)}\big),\\
\label{eq:L2estgrad2}\|d_A\nabla\V(A)\|_{L^p(\Sigma)}\leq c(A_0,p)\big(1+\|F_A\|_{L^p(\Sigma)}+\|\alpha(A)\|_{L^{2p}(\Sigma)}^2\big)
\end{align}
are satisfied for all $A\in\A^{0,p}(P)$ with $\|\alpha(A)\|_{L^2(\Sigma)}<\delta(A_0,p)$. 
\end{prop}

\begin{proof}
Throughout we denote $\alpha\coloneqq\alpha(A)$. To prove \eqref{eq:L2estgrad1} we claim that it suffices to consider $A\in\im\mathfrak m$ such $A$ is contained in the local slice with respect to $A_0$, i.e.~satisfying $d_{A_0}^{\ast}\alpha=0$ with $\alpha=A-A_0$. To see this note that $\alpha(A)=0$ for $A\notin\im\mathfrak m$, and for every $A\in\im\mathfrak m$ there exists by Proposition \ref{prop:condlocslicepert} a gauge transformation $g\in\G^{1,p}(P)$ such that $g^{\ast}A$ is in local slice with respect to $A_0$. As both sides of \eqref{eq:L2estgrad1} are invariant under the action of $\G^{1,p}(P)$ by gauge transformations the claim follows. Hence we may assume $A=A_0+\alpha$ and $d_{A_0}^{\ast}\alpha=0$. This implies $d_{A_0}\alpha=F_A-F_{A_0}-\frac{1}{2}[\alpha\wedge\alpha]$ and  
\begin{eqnarray*} 
\Delta_{A_0}\alpha=d_{A_0}^{\ast}(F_A-F_{A_0})-\frac{1}{2}d_{A_0}^{\ast}[\alpha\wedge\alpha].
\end{eqnarray*}
Elliptic regularity of the operator $\Delta_{A_0}\colon W^{1,p}(\Sigma)\to W^{-1,p}(\Sigma)$ yields for a constant $c(A_0,p)$ the estimate
\begin{eqnarray*}
\lefteqn{\|\alpha\|_{W^{1,p}(\Sigma)}}\\
&\leq&c(A_0,p)\big(\|d_{A_0}^{\ast}F_{A_0}\|_{W^{-1,p}(\Sigma)}+\|d_{A_0}^{\ast}F_A\|_{W^{-1,p}(\Sigma)}\\
&&+\|d_{A_0}^{\ast}[\alpha\wedge\alpha]\|_{W^{-1,p}(\Sigma)}+\|\alpha\|_{W^{-1,p}(\Sigma)}\big)\\
&\leq&c(A_0,p)\big(1+\|F_A\|_{L^p(\Sigma)}+\|\{\nabla_{A_0}\alpha,\alpha\}\|_{W^{-1,p}(\Sigma)}+\|\alpha\|_{L^2(\Sigma)}\big)\\
&\leq&c(A_0,p)\big(1+\|F_A\|_{L^p(\Sigma)}+\|\alpha\|_{L^2(\Sigma)}\|\nabla_{A_0}\alpha\|_{L^p(\Sigma)}+\|\alpha\|_{L^2(\Sigma)}\big).
\end{eqnarray*}
To obtain the last inequality we applied Proposition \ref{prop:prodestimLp}. Now fix the constant $\delta(A_0,p)>0$ sufficiently small such that 
\begin{eqnarray*}
c(A_0,p)\|\alpha\|_{L^2(\Sigma)}\|\nabla_{A_0}\alpha\|_{L^p(\Sigma)}\leq\|\alpha\|_{W^{1,p}(\Sigma)}
\end{eqnarray*}
holds for all $\alpha$ with $\|\alpha\|_{L^2(\Sigma)}<\delta(A_0,p)$, to conclude \eqref{eq:L2estgrad1}. We next show \eqref{eq:L2estgrad1A}. In view of the formula for $\nabla\V(A)$ given in Proposition \ref{formuladifferential} this amounts to estimate the $C^0$ norms of the terms $\alpha$ and $T_{A_0,\alpha}^{\ast}(\ast[\alpha\wedge\ast\eta])$. Using the Sobolev embedding $W^{1,3}(\Sigma)\hookrightarrow C^0(\Sigma)$ we obtain from \eqref{eq:L2estgrad1} and some constant $c(A_0)$ that
\begin{eqnarray}\label{eq:C0boundalphaFA}
\|\alpha\|_{C^0(\Sigma)}\leq c(A_0)(1+\|F_A\|_{L^3(\Sigma)}).
\end{eqnarray}
Similarly, the $C^0(\Sigma)$ norm of  
\begin{eqnarray*} 
T_{A_0,\alpha}^{\ast}(\ast[\alpha\wedge\ast\eta])=M_{\alpha}^{\ast}R_{A_0,\alpha}(\ast[\alpha\wedge\ast\eta])=[\alpha\wedge[R_{A_0,\alpha}(\ast[\alpha\wedge\ast\eta])]]
\end{eqnarray*}
can be estimated against 
\begin{multline*}
\|\alpha\|_{C^0(\Sigma)}\|R_{A_0,\alpha}(\ast[\alpha\wedge\ast\eta])\|_{C^0(\Sigma)}\leq\|\alpha\|_{C^0(\Sigma)}\|R_{A_0,\alpha}(\ast[\alpha\wedge\ast\eta])\|_{W^{2,2}(\Sigma)}\\
\leq\|\alpha\|_{C^0(\Sigma)}\|[\alpha\wedge\ast\eta]\|_{L^2(\Sigma)},
\end{multline*}
where in the last step we used Proposition \ref{prop:propertiesL}. Now $\|\alpha\|_{L^2(\Sigma)}\leq\delta(A_0,p)$ by assumption and hence the $C^0$ norm of $T_{A_0,\alpha}^{\ast}(\ast[\alpha\wedge\ast\eta])$ is dominated by some multiple of $\delta(A_0,p)\|\alpha\|_{C^0(\Sigma)}$. Together with \eqref{eq:C0boundalphaFA} this shows \eqref{eq:L2estgrad1A}. To prove \eqref{eq:L2estgrad2} we denote $\gamma\coloneqq R_{A_0,\alpha}(\ast[\alpha\wedge\ast\eta])$ and proceed as before. Again we may assume that $A=A_0+\alpha$ and $d_{A_0}^{\ast}\alpha=0$ by gauge invariance of both sides of \eqref{eq:L2estgrad2}. From the formula for $\nabla\V(A)$ stated in Proposition \ref{formuladifferential} we see that it suffices to estimate the $L^p(\Sigma)$ norms of the terms $d_A\alpha$ and 
\begin{eqnarray}\label{eq:L2estdAgrad}
d_A[\alpha\wedge\gamma]=[d_A\alpha\wedge\gamma]-[\alpha\wedge d_{A_0}\gamma]-[\alpha\wedge[\alpha\wedge\gamma]].
\end{eqnarray}
The asserted bound for $d_A\alpha=d_{A_0}\alpha+[\alpha\wedge\alpha]$ follows from \eqref{eq:L2estgrad1}. Proposition \ref{prop:propertiesL} yields a bound for $\|\gamma\|_{W^{2,2}(\Sigma)}$ in terms of $\|\alpha\|_{L^2(\Sigma)}$. Then the continuous embeddings $W^{2,2}(\Sigma)\hookrightarrow C^0(\Sigma)$ and $W^{1,2}(\Sigma)\hookrightarrow L^p(\Sigma)$ show that the $L^p(\Sigma)$ norms of the last two terms in \eqref{eq:L2estdAgrad} are bounded by a multiple of $\|\alpha\|_{L^2(\Sigma)}(\|\alpha\|_{C^0(\Sigma)}+\|\alpha\|_{L^{2p}(\Sigma)}^2)$. Applying \eqref{eq:L2estgrad1} to further estimate $\|\alpha\|_{C^0(\Sigma)}$ we obtain \eqref{eq:L2estgrad2}.
\end{proof}

\begin{prop}\label{prop:Hessianbounded}
Let $\V\coloneqq\V_{\ell}\colon\A(P)\to\mathbbm R$ be the perturbation as in Proposition \ref{formuladifferential}. Then for all $p>1$ and $A\in\A(P)$ such that $\alpha(A)\in L^{\infty}(\Sigma)$ the Hessian $H_A\V$ is a bounded linear operator
\begin{eqnarray*}
H_A\V:L^p(\Sigma,T^{\ast}\Sigma\otimes\ad(P))\to L^p(\Sigma,T^{\ast}\Sigma\otimes\ad(P)).
\end{eqnarray*}
Its operator norm satisfies $\|H_A\V\|\leq c(1+\|\alpha(A)\|_{L^{\infty}(\Sigma)})$ for some constant $c=c(p,\ell)$ independent of $A$.
\end{prop}

\begin{proof}
Consider the formula for $H_A\V$ obtained in Proposition \ref{formuladifferential}.
We first check boundedness of the operator $T_{A_0,\alpha}=R_{A_0,\alpha}\circ M_{\alpha}$. Its operator norm satisfies the required bound as follows from boundedness of $M_\alpha\colon L^p\to L^p$, Proposition \ref{prop:propertiesL}, and boundedness of the embedding $W^{2,2}(\Sigma)\hookrightarrow L^p(\Sigma)$. The same argument applies to the map $T_{A_0,\alpha}^{\ast}=M_{\alpha}^{\ast}\circ R_{A_0,\alpha}$, and similarly to the map $\beta\mapsto\gamma=\beta-d_{g^{\ast}A}T_{A_0,\alpha}\beta$. For the latter we use that the map $d_{g^{\ast}A}=d_{A_0}+[\alpha\wedge\,\cdot\,]\colon W^{2,2}(\Sigma)\to W^{1,2}(\Sigma)$ and the embedding $W^{1,2}(\Sigma)\hookrightarrow L^p(\Sigma)$ are bounded (with constant depending on $\|\alpha(A)\|_{L^{\infty}(\Sigma)}$ and $A_0$). Bounds for the map $\gamma\mapsto S_{A_0,\alpha,\gamma}^{\ast}(\ast[\alpha\wedge\ast\eta])$ and maps of the type $\gamma\mapsto\langle\eta,\gamma\rangle\alpha$ follow similarly. Likewise, the norm of the operator $\beta\mapsto[g^{-1}\nabla\V(A)g,T_{A_0,\alpha}(g^{-1}\beta g)]$ can be estimated as before using that $\|\nabla\V(A)\|_{L^{\infty}(\Sigma)}$ satisfies a bound in terms of $\|\alpha\|_{L^{\infty}(\Sigma)}$ (as follows from the formula for $\nabla\V(A)$ in Proposition \ref{formuladifferential}).
\end{proof}

\begin{prop}\label{prop:estimatespertVell}
For every $\ell\in\mathbbm N$ there exists a constant $C_{\ell}>0$ such that the estimates (i-iv) of Section \ref{Bspaceperturbations} concerning the perturbation $\V_{\ell}$ are satisfied.
\end{prop}

\begin{proof}
(i) follows from the definition \eqref{modelperturbation} of $\V_{\ell}$, the estimate $|\langle\alpha_i(A),\eta_{ij}\rangle|\leq\|\alpha_i(A)\|_{L^2(\Sigma)}\|\eta_{ij}\|_{L^2(\Sigma)}$, and our assumption $\|\alpha_i(A)\|_{L^2(\Sigma)}\leq\delta(A_i,p)$. To show (ii) we note, arguing as in the proof \eqref{eq:L2estgrad1A}, that $\|\nabla\V_{\ell}(A)\|_{L^2(\Sigma)}$ can be bounded in terms of $\|\alpha_i(A)\|_{L^2(\Sigma)}$. Hence (ii) follows from (i). Inequality (iii) follows from \eqref{eq:L2estgrad1A}. To obtain (iv) we note that $\|\alpha_i(A)\|_{L^{\infty}(\Sigma)}\leq c(A_i)(1+\|F_{A}\|_{L^3(\Sigma)})$ as follows from \eqref{eq:L2estgrad1} and the Sobolev embedding $W^{1,3}(\Sigma)\hookrightarrow L^{\infty}(\Sigma)$. Now combine this estimate with Proposition \ref{prop:Hessianbounded}.
\end{proof}

\section{Perturbed Yang--Mills gradient flow}

Throughout this section we fix a compact interval $I=[a,b]$. Let $\V=\sum_{\ell}\lambda_{\ell}\V_{\ell}\in Y$ be a perturbation, $Y$ denoting the universal space of perturbations as introduced in \eqref{def:univspacepert}. In this section we derive a priori estimates for solutions of the perturbed Yang--Mills gradient flow equation
\begin{eqnarray}\label{pertYMF}
\partial_sA+d_A^{\ast}F_A+\nabla\V(A)=0.
\end{eqnarray}
These are in particular solutions of \eqref{EYF} with $\Psi=0$. Conversely, any solution of \eqref{EYF} is by Proposition \ref{prop:temporalgauge} gauge equivalent under $\G_{\delta}^{2,p}(\hat P)$ to a solution of \eqref{pertYMF}, so for many purposes it is sufficient to have estimates only for these.  

\begin{prop}\label{FAuniversalbound}
Let $A$ be a solution of \eqref{pertYMF} on $I\times\Sigma$. Then for all $s\in I=[a,b]$ there holds the estimate
\begin{eqnarray*}
\|F_{A(s)}\|_{L^2(\Sigma)}\leq\|F_{A(a)}\|_{L^2(\Sigma)}+4\|\V\|.
\end{eqnarray*}
\end{prop}

\begin{proof}
The energy $\YMV(A)=\frac{1}{2}\int_{\Sigma}|F_A|^2\dvol(\Sigma)+\V(A)$ is monotone decreasing along flow lines, hence
\begin{eqnarray*}
\frac{1}{2}\|F_{A(s)}\|_{L^2(\Sigma)}^2&\leq&\YMV(A(s))+|\V(A(s))|\\
&\leq&\YMV(A(a))+\sup_{A\in\A(P)}|\V(A)|\\
&\leq&\frac{1}{2}\int_{\Sigma}|F_{A(a)}|^2\dvol(\Sigma)+2\sup_{A\in\A(P)}|\V(A)|\\
&\leq&\frac{1}{2}\int_{\Sigma}|F_{A(a)}|^2\dvol(\Sigma)+2\|\V\|,
\end{eqnarray*}
where in the last line we used property (i) of Section \ref{Bspaceperturbations} and the definition of $\|\V\|$.  
\end{proof}

\begin{prop}\label{FALpuniversalbound}
For every $2<p<4$ there exists a constant $C(p,|I|,\|\V\|)$ such that
\begin{eqnarray*}
\|F_A\|_{L^p(I\times\Sigma)}\leq C(p,|I|,\|\V\|)\big(1+\YMV(A(a))^{\frac{1}{2}}\big)
\end{eqnarray*}
for every $A$ satisfying \eqref{pertYMF} on $I\times\Sigma$. 
\end{prop}

\begin{proof}
We use H\"older's inequality with exponents $r=\frac{2}{p-2}$ and $s=\frac{2}{4-p}$ to obtain the estimate
\begin{eqnarray*}
\int_I\int_{\Sigma}|F_A|^p&\leq&\int_I\big(\int_{\Sigma}|F_A|^2\big)^{\frac{1}{r}}\big(\int_{\Sigma}|F_A|^{\frac{4}{4-p}}\big)^{\frac{1}{s}}\\
&\leq&\sup_{s\in I}\|F_{A(s)}\|_{L^2(\Sigma)}^{\frac{2}{r}}\int_I\|F_A\|_{L^{\frac{4}{4-p}}(\Sigma)}^2\\
&\leq&c(p)\|F_A\|_{L^{\infty}(I,L^2(\Sigma))}^{p-2}\int_I\big(\|F_A\|_{L^2(\Sigma)}^2+\|\nabla_AF_A\|_{L^2(\Sigma)}^2\big)\\
&\leq&c(p)|I|\cdot\|F_A\|_{L^{\infty}(I,L^2(\Sigma))}^p+c(p)\|F_A\|_{L^{\infty}(I,L^2(\Sigma))}^{p-2}\int_I\|\nabla_AF_A\|_{L^2(\Sigma)}^2.
\end{eqnarray*}
The third line is by the Sobolev embedding $W^{1,2}(\Sigma)\hookrightarrow L^{\frac{4}{4-p}}(\Sigma)$. Thanks to Proposition \ref{FAuniversalbound} we can  estimate the integral in the last line as 
\begin{eqnarray*}
\int_I\|\nabla_AF_A\|_{L^2(\Sigma)}^2&\leq&2\int_I\|d_A^{\ast}F_A+\nabla\V(A)\|_{L^2(\Sigma)}^2+2\int_I\|\nabla\V(A)\|_{L^2(\Sigma)}^2\\
&\leq&2\YMV(A(a))+2|I|\sup_{A\in\A(P)}\|\nabla\V(A)\|_{L^2(\Sigma)}^2\\
&\leq&2\YMV(A(a))+2|I|\cdot\|\V\|^2.
\end{eqnarray*}
The last inequality follows from the definition of $\|\V\|$ and condition (ii) in Section \ref{Bspaceperturbations}, with
\begin{multline*}
\sup_{A\in\A(P)}\|\nabla\V(A)\|_{L^2(\Sigma)}=\sup_{A\in\A(P)}\big\|\sum_{\ell=1}^{\infty}\lambda_{\ell}\V_{\ell}\big\|_{L^2(\Sigma)}\leq\sum_{\ell=1}^{\infty}C_{\ell}|\lambda_{\ell}|=\|\V\|.
\end{multline*}
Putting the previous estimates together, the claim follows.
\end{proof}

Let $\Delta_{\Sigma}=-\ast d\ast d$ denote the (positive semidefinite) Hodge Laplacian on functions on the Riemannian manifold $(\Sigma,g)$ and let $L_{\Sigma}\coloneqq\partial_s+\Delta_{\Sigma}$ be the corresponding heat operator. We also recall the Bochner--Weitzenb\"ock formula \eqref{BWformula}, relating the covariant Hodge Laplacian $\Delta_A$ and the Bochner Laplacian $\nabla_A^{\ast}\nabla_A$ on forms in $\Omega^k(\Sigma,\ad(P))$. For a form $\ph\in\Omega^k(\Sigma,\ad(P))$ we have the identity
\begin{eqnarray}\label{Laplaceformula}
\Delta_{\Sigma}\frac{1}{2}|\ph|^2=-|\nabla_A\ph|^2+\langle\nabla_A^{\ast}\nabla_A\ph,\ph\rangle.
\end{eqnarray}
We shall also make use of the commutator identity
\begin{eqnarray}\label{commutatorid}
[\nabla_A,\nabla_A^{\ast}\nabla_A]\ph=\{\ph,\nabla_A\ph\},
\end{eqnarray}
cf.~\cite[p.~17]{Donaldson}.

\begin{prop}\label{prop:curvevol}
Assume that $A$ satisfies \eqref{pertYMF} on $I\times\Sigma$. Consider (for $p\geq2$) the function $u_p\colon I\times\Sigma\to\mathbbm R$ defined by $u_p(s,z)\coloneqq\frac{1}{p}|\ast F_{A(s)}(z)|^p$. Denote $u\coloneqq u_2$. Then the following holds,
\begin{eqnarray*}
L_{\Sigma}u&=&-|d_A\ast F_A|^2-\langle\ast F_A,\ast d_A\nabla\V(A)\rangle,\\
L_{\Sigma}u_p&=&|\ast F_A|^{p-2}\left(-|d_A\ast F_A|^2-\langle\ast F_A,\ast d_A\nabla\V(A)\rangle\right)\\
&&-\ast(p-2)|\ast F_A|^{p-4}\langle\ast F_A,d_A\ast F_A\rangle\wedge\langle\ast F_A,\ast d_A\ast F_A\rangle.
\end{eqnarray*}
Moreover, the map $\ast\langle\ast F_A,d_A\ast F_A\rangle\wedge\langle\ast F_A,\ast d_A\ast F_A\rangle\colon I\times\Sigma\to\mathbbm R$ is non-negative.
\end{prop}

\begin{proof}
We calculate using \eqref{pertYMF}, 
\begin{eqnarray*}
\frac{d}{ds}\frac{1}{2}\langle\ast F_A,\ast F_A\rangle=\langle\ast F_A,\ast d_A\dot A\rangle=\langle\ast F_A,-\ast\Delta_AF_A-\ast d_A\nabla\V(A)\rangle.
\end{eqnarray*}
From this it follows that
\begin{eqnarray*}
L_{\Sigma}u&=&\big(\partial_s-\ast d\ast d\big)\frac{1}{2}\langle\ast F_A,\ast F_A\rangle\\
&=&-\langle\ast F_A,\ast\Delta_A F_A+\ast d_A\nabla\V(A)\rangle-\ast d\ast\langle\ast F_A,d_A\ast F_A\rangle\\
&=&-\langle\ast F_A,\ast\Delta_AF_A+\ast d_A\nabla\V(A)\rangle-\langle\ast F_A,\ast d_A\ast d_A\ast F_A\rangle-|d_A\ast F_A|^2\\
&=&-\langle\ast F_A,\ast d_A\nabla\V(A)\rangle-|d_A\ast F_A|^2.
\end{eqnarray*}
The formula for $u_p$ follows from that for $u$ and the further calculation
\begin{eqnarray*}
-\ast d\ast d\frac{1}{p}|\ast F_A|^p&=&-\ast d\big(|\ast F_A|^{p-2}\ast\langle\ast F_A,d_A\ast F_A\rangle\big)\\
&=&-\ast(p-2)|\ast F_A|^{p-4}\langle\ast F_A,d_A\ast F_A\rangle\wedge\langle\ast F_A,\ast d_A\ast F_A\rangle\\
&&-\ast|\ast F_A|^{p-2}d\langle\ast F_A,\ast d_A\ast F_A\rangle.
\end{eqnarray*}
The statement on non-negativity follows by a short calculation in local coordinates. Namely, if we write $\ast F_A=\beta$, $d_A\ast F_A=\alpha_1\,dx_1+\alpha_2\,dx_2$ with respect to local orthonormal coordinates $x_1,x_2$ and maps $\alpha_i,\beta\in C^{\infty}(U,\mathfrak g)$, then it follows that
\begin{eqnarray*}
\lefteqn{\langle\beta,d_A\ast F_A\rangle\wedge\langle\beta,\ast d_A\ast F_A\rangle}\\
&=&(\langle\beta,\alpha_1\rangle\,dx_1+\langle\beta,\alpha_2\rangle\,dx_2)\wedge(\langle\beta,\alpha_1\rangle\,dx_2-\langle\beta,\alpha_2\rangle\,dx_1)\\
&=&(\langle\beta,\alpha_1\rangle^2+\langle\beta,\alpha_2\rangle^2)\,dx_1\wedge dx_2,
\end{eqnarray*}
which is a non-negative multiple of the volume form.
\end{proof}

\begin{prop} 
Assume that $A$ satisfies \eqref{pertYMF} on $I\times\Sigma$. Consider the function $u\colon I\times\Sigma\to\mathbbm R$ defined by $u(s,z)\coloneqq\frac{1}{2}|\nabla_{A(s)}F_{A(s)}(z)|^2$. It satisfies
\begin{multline}\label{eq:curveevol1}
L_{\Sigma}u=-|\nabla_A^2F_A|^2+\big\langle\nabla_AF_A,\{\nabla_AF_A,F_A\}+\nabla_A\{R_{\Sigma},F_A\}+\{\nabla\V(A),F_A\}\\
-\nabla_Ad_A\nabla\V(A)\big\rangle.
\end{multline}
\end{prop}

\begin{proof}
We calculate
\begin{eqnarray*}
\frac{d}{ds}\nabla_AF_A&=&\nabla_Ad_A\dot A+\{\dot A,F_A\}\\
&=&\nabla_A\big(-d_Ad_A^{\ast}F_A-d_A\nabla\V(A)\big)+\{d_A^{\ast}F_A+\nabla\V(A),F_A\}\\
&=&\nabla_A\big(-\nabla_A^{\ast}\nabla_AF_A+\{F_A,F_A\}+\{R_{\Sigma},F_A\}-d_A\nabla\V(A)\big)\\
&&+\{d_A^{\ast}F_A+\nabla\V(A),F_A\}\\
&=&-\nabla_A^{\ast}\nabla_A\nabla_AF_A+\{\nabla_AF_A,F_A\}+\nabla_A\{R_{\Sigma},F_A\}\\
&&-\nabla_Ad_A\nabla\V(A)+\{\nabla\V(A),F_A\}.
\end{eqnarray*}
The third line is by the Bochner--Weitzenb\"ock formula \eqref{BWformula}, and the last line uses \eqref{commutatorid}. Combining this expression with \eqref{Laplaceformula} we obtain
\begin{multline*}
L_{\Sigma}u=-|\nabla_A^2F_A|^2+\big\langle\nabla_AF_A,\{\nabla_AF_A,F_A\}+\nabla_A\{R_{\Sigma},F_A\}\\
+\{\nabla\V(A),F_A\}-\nabla_Ad_A\nabla\V(A)\big\rangle,
\end{multline*}
as claimed.
\end{proof}

The following lemma is an adaption of \cite[Lemma 3.3]{Struwe} to the two dimensional situation considered here.
 
\begin{lem}\label{lem:slicewiseLpsection}
Let $A\in\A^{1,2}(P)$ be a fixed reference connection and let $p>1$. There exists a constant $c=c(p,P)$ such that for every form $\ph\in\Omega^k(\Sigma,\ad(P))$ there holds
\begin{eqnarray*}
\|\ph\|_{L^p(\Sigma)}^2\leq c(p,P)\big(\|d_A\ph\|_{L^2(\Sigma)}^2+\|d_A^{\ast}\ph\|_{L^2(\Sigma)}^2+\langle\{F_A,\ph\},\ph\rangle\big).
\end{eqnarray*}
\end{lem}

\begin{proof}
The proof of \cite[Lemma 3.3]{Struwe} applies with minor modifications stemming from the fact that the Sobolev embedding $W^{1,2}(\Sigma)\hookrightarrow L^p(\Sigma)$ holds for all $p<\infty$ (instead of only $W^{1,2}\hookrightarrow L^4$ in dimension $4$).  
\end{proof}

\begin{prop}\label{prop:curvL2Lp} 
Let $p>1$ and $I=[a,b]$. There exists a constant $c(I,p)$ such that if $A$ is a solution of \eqref{pertYMF} on $I\times\Sigma$, then
\begin{eqnarray*}
\int_I\|F_{A(s)}\|_{L^p(\Sigma)}^2\,ds\leq c(p,|I|,\|\V\|)\big(1+\YMV(A(a))^{\frac{3}{2}}\big).
\end{eqnarray*}
\end{prop}

\begin{proof}
We integrate the estimate of Lemma \ref{lem:slicewiseLpsection} with $\ph=F_A$ and use the Bianchi identity $d_AF_A=0$ to obtain 
\begin{multline*}
c(p,P)^{-1}\int_I\|F_{A}\|_{L^p(\Sigma)}^2\leq\int_I\big(\|d_A^{\ast}F_A\|_{L^2(\Sigma)}^2+\|F_A\|_{L^3(\Sigma)}^3\big)\\
\leq2\YMV(A(a))+2|I|\cdot\|\V\|^2+\|F_A\|_{L^3(I\times\Sigma)}^3.
\end{multline*}
The second inequality follows as in the proof of Proposition \ref{FALpuniversalbound}. An estimate for the remaining term $\|F_A\|_{L^3(I\times\Sigma)}^3$ is provided by Proposition \ref{FALpuniversalbound} with $p=3$.
\end{proof}

\begin{lem}[$L^p$ curvature estimate]\label{lem:univcurvatureestimate0} 
Let $1<p<4$ and $I=[a,b]$. There exists a constant $c(p,|I|,\|\V\|)$ such that if $A$ is a solution of \eqref{pertYMF} on $I\times\Sigma$, then
\begin{eqnarray*}
\|F_A\|_{L^p(I\times\Sigma)}^2\leq c(p,|I|,\|\V\|)\big(1+\YMV(A(a))^{1+\frac{1}{p}}\big).
\end{eqnarray*}
\end{lem}

\begin{proof}
H\"older's inequality yields for $p<4$ the estimate
\begin{eqnarray*}
\|F_A\|_{L^p(I\times\Sigma)}^p\leq\|F_A\|_{L^{\infty}(I,L^2(\Sigma))}^{p-2}\|F_A\|_{L^{2}(I,L^{\frac{4}{4-p}}(\Sigma))}^2.
\end{eqnarray*}
Estimates for the last two factors are provided by Propositions \ref{FAuniversalbound} and \ref{prop:curvL2Lp}. Putting these together, the claim follows.
\end{proof}

\begin{prop}\label{prop:curvscliceLp}
Let $p>1$. There exists a constant $c(p,P)$ such that if $A$ is a solution of \eqref{pertYMF} on $I\times\Sigma$, then
\begin{eqnarray*}
\int_I\|d_A^{\ast}F_A\|_{L^p(\Sigma)}^2\leq c(p,P)\int_I\|d_Ad_A^{\ast}F_A\|_{L^2(\Sigma)}^2+c(p,P)\int_I\int_{\Sigma}|F_A|\cdot|d_A^{\ast}F_A|^2.
\end{eqnarray*}
\end{prop}

\begin{proof}
We integrate the estimate of Lemma \ref{lem:slicewiseLpsection} with $\ph=d_A^{\ast}F_A$ over the interval $I$ and use $d_A^{\ast}d_A^{\ast}F_A=0$ to obtain
\begin{eqnarray*}
c(p,P)^{-1}\int_I\|d_A^{\ast}F_A\|_{L^p(\Sigma)}^2\leq\int_I\|d_Ad_A^{\ast}F_A\|_{L^2(\Sigma)}^2+\int_I\int_{\Sigma}\langle\{F_A,d_A^{\ast}F_A\},d_A^{\ast}F_A\rangle.
\end{eqnarray*}
Hence the claim follows.
\end{proof}

\begin{prop}\label{prop:estnablaFtime}
Suppose $A$ is a solution of \eqref{pertYMF} on $I\times\Sigma$. Then the map $s\mapsto R(s)\coloneqq\frac{1}{2}\|d_{A(s)}^{\ast}F_{A(s)}\|_{L^2(\Sigma)}^2$ satisfies the estimate
\begin{multline*}
\sup_{a\leq s\leq b}R(s)\leq R(a)+\int_I\|d_A\nabla\V(A)\|_{L^2(\Sigma)}^2\\
+\int_I\big(\big|\langle d_A^{\ast}F_A,\{d_A^{\ast}F_A,F_A\}\rangle\big|+\big|\langle  F_A,\{d_A\nabla\V(A),F_A\}\rangle\big|\big),
\end{multline*}
where $\{\,\cdot\,,\,\cdot\,\}$ denotes a certain bilinear expression with smooth time-independent coefficients.
\end{prop}

\begin{proof}
From equation \eqref{pertYMF} it follows for every $a\leq s\leq b$ that
\begin{eqnarray*}
\lefteqn{\frac{d}{ds}R(s)=\langle d_A^{\ast}F_A,d_A^{\ast}d_A\dot A-\ast[\dot A\wedge\ast F_A]\rangle}\\
&=&-\langle d_A^{\ast}F_A,d_A^{\ast}d_Ad_A^{\ast}F_A+d_A^{\ast}d_A\nabla\V(A)-\ast[(d_A^{\ast}F_A+\nabla\V(A))\wedge\ast F_A]\rangle\\
&=&-\|d_Ad_A^{\ast}F_A\|_{L^2(\Sigma)}^2-\langle d_Ad_A^{\ast}F_A,d_A\nabla\V(A)\rangle\\
&&+\langle d_A^{\ast}F_A,\ast[(d_A^{\ast}F_A+\nabla\V(A))\wedge\ast F_A]\rangle\\
&\leq&\|d_A\nabla\V(A)\|_{L^2(\Sigma)}^2+\big|\langle d_A^{\ast}F_A,\{d_A^{\ast}F_A,F_A\}\rangle\big|+\big|\langle F_A,\{d_A\nabla\V(A),F_A\}\rangle\big|.
\end{eqnarray*}
To obtain the final estimate, we applied the Cauchy--Schwarz inequality to the term $\langle d_Ad_A^{\ast}F_A,d_A\nabla\V(A)\rangle$. Now integrate this inequality over the interval $[a,s]\subseteq I$ and take the supremum over $s\in I$ to conclude the result.
\end{proof}

\begin{prop}\label{prop:estnablaFtime1}
Let $I=[a,b]$ and $I'=[a_1,b]$, where $a_1\in(a,b)$. There exists a constant $c(p,|I|,|I'|,\|\V\|)$ such that if $A$ is a solution of \eqref{pertYMF} on $I\times\Sigma$, then 
\begin{multline*}
\sup_{s\in I'}\|d_{A(s)}^{\ast}F_{A(s)}\|_{L^2(\Sigma)}^2\leq c(p,|I|,|I'|,\|\V\|)\Big(1+\YMV(A(a))\\
+\int_I\|d_A\nabla\V(A)\|_{L^2(\Sigma)}^2+\int_{I}\int_{\Sigma}|F_A|\cdot|d_A^{\ast}F_A|^2+|F_A|^2\cdot|d_A\nabla\V(A)|\Big).
\end{multline*}
\end{prop}

\begin{proof}
By Fubini's theorem we can find $s_0\in(a,a_1)$ such that 
\begin{multline*}
\|d_{A(s_0)}^{\ast}F_{A(s_0)}\|_{L^2(\Sigma)}^2\leq2(a_1-a)^{-1}\int_a^{a_1}\|d_{A(s)}^{\ast}F_{A(s)}\|_{L^2(\Sigma)}^2\,ds\\
\leq c(a_1-a)^{-1}\big(\YMV(A(a))+|I|\cdot\|\V\|^2\big).
\end{multline*}
The last step follows as in the proof of Proposition \ref{FALpuniversalbound}. Next we apply Proposition \ref{prop:estnablaFtime} (where now $s_0$ takes the role of the parameter $a$ there) to obtain  
\begin{multline*} 
\sup_{s\in I'}\frac{1}{2}\|d_{A(s)}^{\ast}F_{A(s)}\|_{L^2(\Sigma)}^2\leq c(a_1-a)^{-1}\big(\YMV(A(a))+|I|\cdot\|\V\|^2\big)\\
+\int_I\|d_A\nabla\V(A)\|_{L^2(\Sigma)}^2+\int_I\big(\big|\langle d_A^{\ast}F_A,\{d_A^{\ast}F_A,F_A\}\rangle\big|+\big|\langle F_A,\{d_A\nabla\V(A),F_A\}\rangle\big|\big).
\end{multline*}
The claim then follows.
\end{proof}

\begin{lem}[$L^p$ gradient estimate]\label{lem:univcurvatureestimate1}
Let $1<p<4$, $I=[a,b]$, and $I'=[a_1,b]$ such that $a_1\in(a,b)$. There exists a constant $C>0$ which depends only on $p$, $|I|$, $|I'|$, $\|\V\|$, $\YMV(A(a)))$ such that if $A$ is a solution of \eqref{pertYMF} on $I\times\Sigma$, then $\|d_A^{\ast}F_A\|_{L^p(I'\times\Sigma)}\leq C$.
\end{lem}

\begin{proof}
H\"older's inequality yields for $p<4$ the estimate
\begin{eqnarray*}
\|d_A^{\ast}F_A\|_{L^p(I\times\Sigma)}^p\leq\|d_A^{\ast}F_A\|_{L^{\infty}(I,L^2(\Sigma))}^{p-2}\|d_A^{\ast}F_A\|_{L^{2}(I,L^{\frac{4}{4-p}}(\Sigma))}^2.
\end{eqnarray*}
For the two factors appearing on the right-hand side we have established in Propositions \ref{prop:curvscliceLp} and \ref{prop:estnablaFtime1} bounds only involving $p$, $|I|$, $|I'|$, $\|\V\|$, $\YMV(A(a)))$, and the terms
\begin{align*}
\int_I\|d_A\nabla\V(A)\|_{L^2(\Sigma)}^2,\quad\int_{I}\int_{\Sigma}|F_A|^2\cdot|d_A\nabla\V(A)|,\quad\int_I\int_{\Sigma}|F_A|\cdot|d_A^{\ast}F_A|^2,\\
\int_I\|d_Ad_A^{\ast}F_A\|_{L^2(\Sigma)}^2.
\end{align*}
To required estimates for the latter terms then follow from Propositions \ref{prop:prodestimate0}, \ref{prop:prodestimate}, \ref{prop:W12boundFA}, and \ref{prop:boundDDAnablaV}.
\end{proof} 

In the remainder of this section we establish auxiliary results needed in the proof of Lemma \ref{lem:univcurvatureestimate1}.

\begin{prop}\label{prop:prodestimate0}
Let $A$ solve \eqref{pertYMF} on $I\times\Sigma$. Then the product $|F_A|^2|d_A\nabla\V(A)|$ admits the estimate 
\begin{eqnarray*}
\int_I\|d_A\nabla\V(A)\|_{L^2(\Sigma)}^2+\int_{I\times\Sigma}|F_A|^2|d_A\nabla\V(A)|\leq C 
\end{eqnarray*} 
for some constant $C=C(p,|I|,\|\V\|,\YMV(A(a)))$ independent of $A$.
\end{prop}

\begin{proof}
We combine \eqref{eq:L2estgrad1} and \eqref{eq:L2estgrad2} and use the Sobolev embedding $W^{1,2}(\Sigma)\hookrightarrow L^p(\Sigma)$ to obtain for any $p>2$ and fixed $\eps>0$ that
\begin{eqnarray}\label{eq:estFtime2}
\|d_A\nabla\V(A)\|_{L^p(\Sigma)}\leq c\big(1+\|F_A\|_{L^p(\Sigma)}+\|F_A\|_{L^{2+\eps}(\Sigma)}^2\big),
\end{eqnarray}
with constant $c=c(\eps,p,\|\V\|)$. For $p=3$ and $0<\eps\leq1$ we obtain, using  H\"older's inequality, that 
\begin{eqnarray*}
\lefteqn{\int_{I\times\Sigma}|F_A|^2|d_A\nabla\V(A)|\leq\int_I\big(\int_{\Sigma}|F_A|^3\big)^{\frac{2}{3}}\cdot\big(\int_{\Sigma}|d_A\nabla\V(A)|^3\big)^{\frac{1}{3}}}\\
&\leq&c\int_I\big(\int_{\Sigma}|F_A|^3\big)^{\frac{2}{3}}\cdot\big(1+\big(\int_{\Sigma}|F_A|^3\big)^{\frac{1}{3}}+\big(\int_{\Sigma}|F_A|^{2+\eps}\big)^{\frac{2}{2+\eps}}\big)\\
&\leq&c\int_I\big(\int_{\Sigma}|F_A|^3\big)^{\frac{2}{3}}+c\int_{I\times\Sigma}|F_A|^3+c\int_I\big(\int_{\Sigma}|F_A|^3\big)^{\frac{4}{3}}\\
&&+c\int_I\big(\int_{\Sigma}|F_A|^{2+\eps}\big)^{\frac{4}{2+\eps}}\\
&\leq&c\big(\|F_A\|_{L^3(I\times\Sigma))}^3+\|F_A\|_{L^4(I,L^3(\Sigma))}^4\big).
\end{eqnarray*}
The asserted bound for $\|F_A\|_{L^3(I\times\Sigma)}^3$ follows from Proposition \ref{FALpuniversalbound}. Concerning the last term, we apply H\"older's inequality to obtain  
\begin{eqnarray*}
\|F_A\|_{L^4(I,L^3(\Sigma))}^4\leq\|F_A\|_{L^{\infty}(I,L^2(\Sigma))}^2\|F_A\|_{L^2(I,L^6(\Sigma))}^2,
\end{eqnarray*}
and then estimate both factors separately, using Propositions \ref{FAuniversalbound} and \ref{prop:curvL2Lp}. It remains to estimate the term $\int_I\|d_A\nabla\V(A)\|_{L^2(\Sigma)}^2$. Inequality \eqref{eq:estFtime2} yields for $\eps>0$ and $p=2+\eps$ that
\begin{eqnarray*} 
\|d_A\nabla\V(A)\|_{L^2(\Sigma)}^2\leq c\big(1+\|F_A\|_{L^{2+\eps}(\Sigma)}^4\big).
\end{eqnarray*}
Apply the Cauchy--Schwarz inequality and integrate this estimate over $I$ to obtain
\begin{eqnarray*}
\int_I\|d_A\nabla\V(A)\|_{L^2(\Sigma)}^2\leq c\big(|I|+\int_I\|F_A\|_{L^2(\Sigma)}^{\frac{4}{2+\eps}}+\int_I\|F_A\|_{L^{2(1+\eps)}(\Sigma)}^{\frac{4(1+\eps)}{2+\eps}}\big).
\end{eqnarray*}
Proposition \ref{FAuniversalbound} gives the required bound for the first integral on the right-hand side. Choosing $\eps=\frac{1}{2}$ we can apply Lemma \ref{lem:univcurvatureestimate0} to obtain the asserted estimate for the second integral.  
\end{proof}

\begin{prop}\label{prop:prodestimate}
Let $I=[a,b]$ and $I'=[a_1,b]$ such that $a_1\in(a,b)$, and assume $A$ satisfies\eqref{pertYMF} on $I\times\Sigma$. Then the map $|F_A|\cdot|d_A^{\ast}F_A|^2\colon I\times\Sigma\to\mathbbm R$ satisfies the estimate
\begin{eqnarray*}
\int_{I'\times\Sigma}|F_A|\cdot|d_A^{\ast}F_A|^2\leq C 
\end{eqnarray*}
with constant $C=C(|I|,|I'|,\|\V\|,\YMV(A(a)))$ independent of $A$.
\end{prop}

\begin{proof}
Consider the function $u_3=\frac{1}{3}|\ast F_A|^3\colon I\times\Sigma\to\mathbbm R$. By Proposition \ref{prop:curvevol} it satisfies 
\begin{multline*}
L_{\Sigma}u_3=-|\ast F_A|\big(|d_A\ast F_A|^2+\langle\ast F_A,\ast d_A\nabla\V(A)\rangle\big)\\
-\ast|\ast F_A|^{-1}\langle\ast F_A,d_A\ast F_A\rangle\wedge\langle\ast F_A,\ast d_A\ast F_A\rangle,
\end{multline*}
where the term $\langle\ast F_A,d_A\ast F_A\rangle\wedge\langle\ast F_A,\ast d_A\ast F_A\rangle$ is non-negative. Lemma \ref{lem:subparabolic} thus yields the estimate
\begin{eqnarray*}
\int_{I'\times\Sigma}|F_A|\cdot|d_A^{\ast}F_A|^2\leq c(|I|,|I'|)\int_{I\times\Sigma}\frac{1}{3}|F_A|^3+|F_A|\cdot|\langle\ast F_A,\ast d_A\nabla\V(A)\rangle|.
\end{eqnarray*}
Now apply Lemma \ref{lem:univcurvatureestimate0} with $p=3$ and Proposition \ref{prop:prodestimate0} to obtain the result.  
\end{proof}

\begin{prop}\label{prop:W12boundFA}
Let $I=[a,b]$ and $I'=[a_1,b]$ such that $a_1\in(a,b)$. There exists a constant $C=C(|I|,|I'|,\|\V\|,\YMV(A(a)))$ such that if $A$ is a solution of \eqref{pertYMF} on $I\times\Sigma$, then
\begin{eqnarray*}
\int_{I'\times\Sigma}|\nabla_A^2F_A|^2\leq C.
\end{eqnarray*}
\end{prop}

\begin{proof}
We apply Lemma \ref{lem:subparabolic} in the situation of equation \eqref{eq:curveevol1}. After applying the Cauchy--Schwarz inequality to the inner product appearing in \eqref{eq:curveevol1} we obtain for a constant $c=c(|I|,|I'|)$ the bound
\begin{multline*}
\int_{I'\times\Sigma}|\nabla_A^2F_A|^2\leq c\int_{I\times\Sigma}\big(|\nabla_AF_A|^2+|F_A|^2+|\nabla_AF_A|^2|F_A|\\
+|F_A|\cdot|\nabla_AF_A|\cdot|\nabla\V(A)|+|\nabla_Ad_A\nabla\V(A)|^2\big).
\end{multline*}
We estimate each of the five terms on the right-hand side separately. The integral over $|\nabla_AF_A|^2$ is bounded in terms of $|I|$, $\|\V\|$, and $\YMV(A(a))$ as shown along the proof of Proposition \ref{FALpuniversalbound}. Likewise, a bound for the integral over $|F_A|^2$ is provided by Proposition \ref{prop:curvL2Lp}. The same type of estmate holds for $|\nabla_AF_A|^2|F_A|$ by Proposition \ref{prop:prodestimate}. To estimate the term $|\nabla_Ad_A\nabla\V(A)|^2$ we apply Proposition \ref{prop:boundDDAnablaV} with $p=3$. This yields a bound in terms of $\int_I\|F_A\|_{L^3(\Sigma)}^4$ (which can be dealt with as in the proof of Proposition \ref{prop:prodestimate0}) and again of $\int_I\|\nabla_AF_A\|_{L^2(\Sigma)}^2$. The remaining integral over $|F_A|\cdot|\nabla_AF_A|\cdot|\nabla\V(A)|$ can be estimated similarly, using that $\|\nabla\V(A)\|_{C^0(\Sigma)}\leq(1+\|F_A\|_{L^3(\Sigma)})\|\V\|$ as follows from property (iii) in Section \ref{Bspaceperturbations}.
\end{proof}

\begin{prop}\label{prop:boundDDAnablaV}
Let $I=[a,b]$. For every $p>2$ there exists a constant $c(|I|,p,\|\V\|)$ such that  
\begin{eqnarray*}
\int_{I\times\Sigma}|\nabla_Ad_A\nabla\V(A)|^2\leq c(|I|,p,\|\V\|)\big(1+\int_I\|F_A\|_{L^p(\Sigma)}^4+\int_I\|\nabla_AF_A\|_{L^2(\Sigma)}^2\big)
\end{eqnarray*}
holds for all $\V\in Y$ and continuous paths $A\colon I\to\A(P)$ of connections.
\end{prop}

\begin{proof}
Let $\V=\sum_{\ell}\lambda_{\ell}\V_{\ell}\in Y$. As one can check easily, it suffices to show the estimate for $\V=\V_{\ell}$. Recall the formula for $\nabla\V(A)$ in Proposition \ref{formuladifferential}. It follows, keeping the notation used there, that it suffices to estimate the terms $\nabla_Ad_A\alpha$ and $\nabla_Ad_A[\alpha\wedge\gamma]$, where $\gamma\coloneqq R_{A_0,\alpha}(\ast[\alpha\wedge\ast\eta])$. Consider first $\nabla_Ad_A\alpha$. Because of the identity $d_A\alpha=F_A-F_{A_0}+\frac{1}{2}[\alpha\wedge\alpha]$ it remains to estimate the term $|\nabla_A[\alpha\wedge\alpha]|^2$. Let $r\coloneqq\frac{p}{2}>1$ and $s$ be the H\"older conjugate exponent of $r$. It follows that 
\begin{multline*}
\int_{I\times\Sigma}|\nabla_A[\alpha\wedge\alpha]|^2\leq c\int_{I\times\Sigma}|\nabla_A\alpha|^2|\alpha|^2\leq c\int_I\|\nabla_A\alpha\|_{L^{2r}(\Sigma)}^4+c\int_I\|\alpha\|_{L^{2s}(\Sigma)}^4\\
\leq c(|I|,\|\V\|)\big(1+\int_I\|F_A\|_{L^p(\Sigma)}^4\big).
\end{multline*}
In the last inequality we made use of \eqref{eq:L2estgrad1}. The remaining estimate for
\begin{multline*} 
\nabla_Ad_A[\alpha\wedge\gamma]=\nabla_A[d_A\alpha\wedge\gamma]-\nabla_A[\alpha\wedge d_A\gamma]\\
=[\nabla_Ad_A\alpha\wedge\gamma]+[d_A\alpha\wedge\nabla_A\gamma]-[\nabla_A\alpha\wedge d_A\gamma]-[\alpha\wedge\nabla_Ad_A\gamma]
\end{multline*}
is obtained similarly as for $\nabla_Ad_A\alpha$. In particular, to estimate the integral over $|[\nabla_Ad_A\alpha\wedge\gamma]|$ we use Proposition \ref{prop:propertiesL} which implies boundedness of $\|\gamma\|_{W^{2,2}(\Sigma)}$, hence of $\|\gamma\|_{C^0(\Sigma)}$ in terms of $\|\alpha\|_{L^2(\Sigma)}$.
\end{proof}

\section{Auxiliary results}

Let $I$ be a compact interval and $A\in W^{2,p}(I\times\Sigma)$. We give an a priori estimate for the linearized (unperturbed) Yang--Mills gradient flow along the path $s\mapsto A(s)$ ($s\in I$) of connections. Recall from Section \ref{sec:Fredholm theorem} that this linearization is given by the operator $\mathcal D_A=\frac{d}{ds}+\H_A-(\beta+\beta's)\delta\colon\Z^{\delta,p}\to\mathcal L^{\delta,p}$ with $\H_A$ the augmented Hessian as in \eqref{eqsplitting1}. Let $A_0\in\A(P)$ be a fixed smooth reference connection. Recall the $L^2$ orthogonal splitting of $\alpha\in\Omega^1(\Sigma,T^{\ast}\Sigma\otimes\ad(P))$ into $\alpha=\alpha_0+\alpha_1$ where $d_{A_0}^{\ast}\alpha_0=0$ and $\alpha_1=d_{A_0}\ph$ for some $\ph\in\Omega^0(\Sigma,\ad(P))$.

\begin{prop}[Linear estimate] 
Fix $p>3$ and compact intervals $I=[a,b]$ and $I'=[a_1,b]$ such that $a<a_1<b$. For any path $A\in W^{2,p}(I\times\Sigma)$ of connections  there exists a constant $c=c(p,A,|I|,|I'|)$ such that the   estimate 
\begin{multline}\label{apriori2A}
\|\alpha_0\|_{W^{1,2;p}(I'\times\Sigma)}+\|\alpha_1\|_{W^{1,p}(I'\times\Sigma)}+\|\psi\|_{W^{1,p}(I'\times\Sigma)}\\
\leq c\big(\|\mathcal D_A\xi\|_{L^p(I\times\Sigma)}+\|\alpha\|_{L^p(I\times\Sigma)}+\|\psi\|_{L^p(I\times\Sigma)}\big)
\end{multline}
is satisfied for every $\xi=(\alpha_0,\alpha_1,\psi)\in\Z^{\delta,p}$.
\end{prop}

\begin{proof} 
To simplify notation we assume $\delta=0$ in $\mathcal D_A$. This does not affect the claimed estimate. Denote $\eta=(\eta_0,\eta_1,\eta_2)\coloneqq\mathcal D_A\xi$. Let $A_0\in\A(P)$ be the smooth reference connection as above. We remark that by the assumptions $A\in W^{2,p}(I\times\Sigma)$  and $p>3$ the difference $D_A\coloneqq\Delta_A-\Delta_{A_0}$ is a differential operator of order one with continuous coefficients. From the expression of $\H_A$ in \eqref{eqsplitting1} (using the notation introduced there) it follows that $\alpha_0$ satisfies 
\begin{eqnarray*}
\big(\frac{d}{ds}+\Delta_{A_0}\big)\alpha_0=\xi_0+d_{A_0}\omega-D_A\alpha_0-K_A\alpha_0-L_A\ph+[\theta\wedge\psi].
\end{eqnarray*}
The last four terms on the right-hand side of this equation involve derivatives of order at most one in $\alpha$ or $\psi$. The remaining term $\|d_{A_0}\omega\|_{L^p(I'\times\Sigma)}$ is dominated by $\|\alpha_0\|_{W^{1,p}(I\times\Sigma)}+\|\alpha_1\|_{L^p(I\times\Sigma)}$ as follows by elliptic regularity from equation \eqref{eqsplitting}. The estimate for $\alpha_0$ thus follows from a standard parabolic estimate for the linear heat operator $\frac{d}{ds}+\Delta_{A_0}$, together with the subsequent ones for $\|\alpha_1\|_{W^{1,p}(I'\times\Sigma)}$ and $\|\psi\|_{W^{1,p}(I'\times\Sigma)}$. To estimate the terms $\alpha_1$ and $\psi$ we define the linear operator
\begin{eqnarray*} 
B\coloneqq\left(\begin{array}{cc}0&-d_{A_0}\\-d_{A_0}^{\ast}&0\end{array}\right)
\end{eqnarray*}
acting on pairs $(\alpha_1,\psi)^T$. We set
\begin{eqnarray*}
L\coloneqq\big(-\frac{d}{ds}+B\big)\big(\frac{d}{ds}+B\big)=-\frac{d^2}{ds^2}+B^2+M\eqqcolon\diag(L_1,L_2)+M
\end{eqnarray*}
with $L_1=-\frac{d^2}{ds^2}+d_{A_0}d_{A_0}^{\ast}$ and $L_2=-\frac{d^2}{ds^2}+d_{A_0}^{\ast}d_{A_0}$. Recall the Laplace operator $\hat\Delta_{A_0}=-\frac{d^2}{ds^2}+\Delta_{A_0}$ introduced in Section \ref{sec:Compactness}, acting on $\ad(\hat P_I)$-valued $k$-forms over $I\times\Sigma$. Note that $L_2=\hat\Delta_{A_0}$ (acting on $0$-forms). Similarly, $L_1$ acts on $\alpha_1=d_{A_0}\ph$ as
\begin{eqnarray*}
L_1\alpha_1&=&\big(-\frac{d^2}{ds^2}+d_{A_0}d_{A_0}^{\ast}+d_{A_0}^{\ast}d_{A_0}\big)\alpha_1-d_{A_0}^{\ast}d_{A_0}d_{A_0}\ph\\
&=&\hat\Delta_{A_0}\alpha_1+[\ast F_{A_0}\wedge\ast\alpha_1]-[d_{A_0}^{\ast}F_{A_0}\wedge\ph].
\end{eqnarray*}
We consider $-\frac{d}{ds}+B$ as a bounded operator $L^p(I\times\Sigma)\to W^{-1,p}(I\times\Sigma)$ and denote by $K$ its norm. The claimed estimate then follows from elliptic regularity of the Laplace operator $\hat\Delta_A\colon W^{1,p}(I\times\Sigma)\to W^{-1,p}(I\times\Sigma)$, because
\begin{eqnarray*}
\lefteqn{c^{-1}\big(\|\alpha_1\|_{W^{1,p}}+\|\psi\|_{W^{1,p}}\big)}\\
&\leq&\|\hat\Delta_{A_0}\alpha_1\|_{W^{-1,p}}+\|\hat\Delta_{A_0}\psi\|_{W^{-1,p}}+\|\alpha_1\|_{L^p}+\|\psi\|_{L^p}\\
&\leq&\|L_1\alpha_1\|_{W^{-1,p}}+\|[\ast F_{A_0}\wedge\ast\alpha_1]-[d_{A_0}^{\ast}F_{A_0}\wedge\ph]\|_{W^{-1,p}}+\|L_2\psi\|_{W^{-1,p}}\\
&&+\|\alpha_1\|_{L^p}+\|\psi\|_{L^p}\\
&\leq&\|L(\alpha_1,\psi)^T\|_{W^{-1,p}}+(1+\|F_{A_0}\|_{L^{\infty}})\|\alpha_1\|_{L^p}+\|\psi\|_{L^p}\\
&&+\|d_{A_0}^{\ast}F_{A_0}\|_{L^p}\|\ph\|_{L^{\infty}}\\
&\leq&K\|(\frac{d}{ds}+B)(\alpha_1,\psi)^T\|_{L^p}+(1+\|F_{A_0}\|_{L^{\infty}})\|\alpha_1\|_{L^p}+\|\psi\|_{L^p}\\
&&+\|d_{A_0}^{\ast}F_{A_0}\|_{L^p}\|\ph\|_{L^{\infty}}\\
&\leq&K\|\dot\alpha_1-d_{A_0}\psi+d_{A_0}\omega\|_{L^p}+\|d_{A_0}\omega\|_{L^p}+K\|\dot\psi-d_{A_0}^{\ast}\alpha_1\|_{L^p}\\
&&+(1+\|F_{A_0}\|_{L^{\infty}})\|\alpha_1\|_{L^p}+\|d_{A_0}^{\ast}F_{A_0}\|_{L^p}\|\alpha_1\|_{L^p}+\|\psi\|_{L^p}.
\end{eqnarray*}
Note that the terms $\dot\alpha_1-d_A\psi+d_A\omega$ and $\dot\psi-d_A^{\ast}\alpha_1$ correspond to the last two components of the operator $\frac{d}{ds}+\H_A$ (up to a term $\ast[\theta\wedge\ast\alpha]$ which may be absorbed). The term $d_{A_0}\omega$ appearing here can be estimated as before. The asserted estimate now follows.
\end{proof}

\begin{lem}\label{apriori0}
Let $f\colon(-\infty,T]\to\mathbbm R$ be a bounded $C^2$ function with $f\geq0$, and such that the differential inequality
\begin{eqnarray*} 
f''\geq c_0f+c_1f'
\end{eqnarray*}
is satisfied for constants $c_0>0$ and $c_1\in\mathbbm R$. Then $f$ satisfies the decay estimate
\begin{eqnarray*}
f(s)\leq e^{-\lambda(T-s)}f(T)
\end{eqnarray*}
for a constant $\lambda=\lambda(c_0,c_1)>0$ and all $-\infty<s\leq T$.
\end{lem}

\begin{proof}
Set
\begin{eqnarray*}
k\coloneqq-\frac{c_1}{2}+\frac{1}{2}\sqrt{4c_0+c_1^2}>0\qquad\textrm{and}\qquad\lambda\coloneqq\frac{c_1}{2}+\frac{1}{2}\sqrt{4c_0+c_1^2}>0. 
\end{eqnarray*}
Then $k\lambda=-c_0$ and $k-\lambda=-c_1$. Assume by contradiction that $f'(s_0)-\lambda f(s_0)<0$ for some $s_0\leq T$ and set $g(s)\coloneqq e^{ks}(f'(s)-\lambda f(s))$. Then
\begin{eqnarray*}
g'=e^{ks}(f''+(k-\lambda)f'-k\lambda f)=e^{ks}(f''-c_1f'-c_0f)\geq0,
\end{eqnarray*}
so $g$ is monotone increasing. Therefore $g(s)\leq g(s_0)$ for all $s\leq s_0$ and hence
\begin{eqnarray*}
f'(s)\leq\lambda f(s)+e^{k(s_0-s)}(f'(s_0)-\lambda f(s_0)).
\end{eqnarray*}
Because $f$ is bounded and $f'(s_0)-\lambda f(s_0)<0$ by assumption, it follows that $f'(s)\to-\infty$ as $s\to-\infty$. This contradicts the boundedness of $f$ as $f(s_0)-f(s)=\int_s^{s_0}f'(\sigma)\,d\sigma$. Therefore the assumption was wrong and $f'(s)-\lambda f(s)\geq0$ holds for all $-\infty<s\leq T$. Then with $h(s)\coloneqq e^{-\lambda s}f(s)$ it follows that
\begin{eqnarray*}
h'=e^{-\lambda s}(f'-\lambda f)\geq0,
\end{eqnarray*}
which implies $f(s)\leq e^{-\lambda(T-s)}f(T)$ for all $s\leq T$.
\end{proof}

\begin{prop}\label{prop:l2w-12estimate}
For every map $u\in L^2(\Sigma)$ there holds the estimate $\|u\|_{W^{-1,2}(\Sigma)}\leq\|u\|_{L^2(\Sigma)}$.
\end{prop}

\begin{proof}
The assertion follows from  
\begin{multline*}
\|u\|_{W^{-1,2}(\Sigma)}=\sup_{\|\ph\|_{W^{1,2}(\Sigma)}=1}\big|\int_{\Sigma}u\ph\big|\leq\sup_{\|\ph\|_{W^{1,2}(\Sigma)}=1}\|u\|_{L^2(\Sigma)}\|\ph\|_{L^2(\Sigma)}\\\leq\|u\|_{L^2(\Sigma)}\sup_{\|\ph\|_{W^{1,2}(\Sigma)}=1}\|\ph\|_{W^{1,2}(\Sigma)}=\|u\|_{L^2(\Sigma)},
\end{multline*}
where we used the Cauchy--Schwarz inequality in the second step.
\end{proof}

\begin{thm}[Local slice theorem]\label{thm:locslicethm}
Let $M$ be a compact Riemannian $n$-manifold with smooth boundary (that might be empty). Let $1<p\leq q<\infty$ such that
\begin{eqnarray*}
p>\frac{n}{2}\qquad\textrm{and}\qquad\frac{1}{n}>\frac{1}{q}>\frac{1}{p}-\frac{1}{n},
\end{eqnarray*}
or $q=\infty$ in the case $p>n$. Fix a reference connection $\hat A\in\A^{1,p}(P)$ and a constant $c_0>0$. Then there exist constants $\delta>0$ and $C$ such that the following holds: For every $A\in\A^{1,p}(P)$ with
\begin{eqnarray*}
\|A-\hat A\|_{L^q(M)}\leq\delta\qquad\textrm{and}\qquad\|A-\hat A\|_{W^{1,p}(M)}\leq c_0
\end{eqnarray*}
there exists a gauge transformation $g\in\G^{2,p}(P)$ such that
\begin{compactenum}[(i)]
\item
$d_{\hat A}^{\ast}(g^{\ast}A-\hat A)=0$,
\item
$\ast(g^{\ast}A-\hat A)|_{\partial M}=0$,
\item
$\|g^{\ast}A-\hat A\|_{L^q(M)}\leq C\|A-\hat A\|_{L^q(M)}$,
\item
$\|g^{\ast}A-\hat A\|_{W^{1,p}}\leq C\|A-\hat A\|_{W^{1,p}(M)}$.
\end{compactenum}
\end{thm}

\begin{proof}
For a proof we refer to \cite[Theorem F]{Wehrheim}.
\end{proof}

In the following lemma, we let $\Delta_{\Sigma}=d^{\ast}d$ denote the (positive semidefinite) Hodge Laplacian on functions, and $L_{\Sigma}\coloneqq\partial_s+\Delta_{\Sigma}$ the corresponding heat operator. We also use the notation $B_r(x)$ for the open ball of radius $r>0$ around $x\in\Sigma$, and $P_r(x)\coloneqq(-r^2,0)\times B_r(x)$ for the corresponding parabolic cylinder.

\begin{lem}\label{lem:subparabolic}
Let $R,r>0$, $u\colon P_{R+r}\to\mathbbm R$ be a $C^2$ function and $f,g\colon P_{R+r}\to\mathbbm R$ be continuous functions such that
\begin{eqnarray*}
-L_{\Sigma}u\geq g-f,\qquad u\geq0,\quad f\geq0,\quad g\geq0.
\end{eqnarray*}
Then 
\begin{eqnarray*}
\int_{P_{R/2}}g\leq2\left(1+\frac{r}{R}\right)\left(\int_{P_{R+r}}f+\big(\frac{4}{r^2}+\frac{1}{Rr}\big)\int_{P_{R+r}}u\right).
\end{eqnarray*}
\end{lem}

\begin{proof}
For a proof we refer to \cite[Lemma B.5]{SalWeb}.
\end{proof}

%
%
%
%

\end{document}